\documentclass[11pt]{amsart}

\usepackage[top=1.05in, left=1in, right=1in, bottom=1.05in]{geometry}

\usepackage{amssymb,amsthm}
\usepackage{delarray}

\usepackage{ifpdf}
\ifpdf
\usepackage[pdftex]{graphicx}
\else
\usepackage[dvips]{graphicx}
\fi

\usepackage{ifpdf}
\usepackage{color}
\definecolor{webgreen}{rgb}{0,.5,0}
\definecolor{webbrown}{rgb}{.8,0,0}
\definecolor{emphcolor}{rgb}{0.95,0.95,0.95}

\usepackage{hyperref}
\hypersetup{%
          colorlinks=true,
          linkcolor=webbrown,
          filecolor=webbrown,
          citecolor=webgreen,
          breaklinks=true}

\renewcommand{\theequation}{\thesection.\arabic{equation}}
\numberwithin{equation}{section}

\newtheorem {prop}{Proposition}[section]
\newtheorem {lemma}{Lemma}[section]

\newtheorem {cor}{Corollary}[section]
\newtheorem {rem}{Remark}[section]

\newcommand{\R}{\mathbb R}
\newcommand{\N}{\mathbb N}
\renewcommand{\P}{\mathbb P}

\newcommand{\E}{\mathbb E}

\newcommand{\F}{\mathbb F}

\renewcommand{\bar}{\overline}

\renewcommand{\emptyset}{\varnothing}
\newcommand{\Fc}{\mathcal F}
\newcommand{\vP}{\vec{\Pi} }
\newcommand{\vp}{\vec{\pi} }
\newcommand{\vx}{\vec{x} }
\newcommand{\eps}{\varepsilon}
\newcommand{\e}{\mathrm{e}}

\author{Michael Ludkovski}
\address[M.\ Ludkovski]{Department of Statistics and Applied Probability, University of California Santa Barbara, CA 93106-3110}
\email{ludkovski@pstat.ucsb.edu}

\author{Semih O.\ Sezer}
\address[S.\ O.\ Sezer]{School of Engineering and Applied Sciences, Sabanci University, Istanbul}
\email{sezer@sabanciuniv.edu}

\title{Finite Horizon Decision Timing with Partially Observable Poisson Processes}

\keywords{Markov-modulated Poisson processes, Bayesian sequential analysis,
 optimal stopping, decision making}

\subjclass[2000]{Primary 62L10; Secondary 62L15, 62C10, 60G40}

\begin{document}

\begin{abstract}
We study decision timing problems on finite horizon with Poissonian information arrivals. In
our model, a decision maker wishes to optimally time her action in order to maximize her expected reward. The reward depends on an unobservable
Markovian environment, and information about the environment is collected through a (compound)
Poisson observation process. Examples of such systems arise in investment timing, reliability
theory, Bayesian regime detection and technology adoption models.
We solve the problem by studying an optimal stopping problem for a piecewise-deterministic process which gives the posterior likelihoods of the unobservable environment.
Our method lends itself to simple numerical
implementation and we present several illustrative numerical examples.
\end{abstract}

\maketitle

\section{Introduction}\label{sec:intro}

Decision timing under uncertainty is one of the fundamental problems in Operations Research. In a typical setting, an
economic agent (called the decision-maker or DM) has a set of possible actions $\mathcal{A}$ where each action has a (random) reward associated with it.
The objective of the DM is to select a single action and time it so as to maximize her expected reward. More precisely, the DM picks a stopping time $\tau$
and action $k$ from the set $\mathcal{A}$ at $\tau$. The reward $H$ that DM receives is a function of the pair $(\tau, k)$, as well as of some stochastic state variable $Y$. 
In classical examples (e.g.\ investment timing, American option pricing, natural resource management, etc.), $Y$ is an \emph{observable} stochastic process (e.g.\ asset prices, market demand etc.), and the DM's objective is a standard optimal stopping problem.

More complicated stopping problems involving \emph{unobserved} system states have
also been considered in the literature; see, for example, \cite{Bather}, 
\cite{Jensen}, \cite{McCardle85}, \cite{Mazziotto86}, \cite{JensenHsu}, \cite{Stadje-stopping}, \cite{Schottl},  \cite{Fuh02}, \cite{DecampsEtal}, \cite{dayanik-goulding}.
Such models
are especially natural when one wishes to capture the inherent
conflict between gathering of information (which makes waiting valuable) and the time-value of
money (which makes waiting costly). Indeed, most realistic settings involve a DM who is only
partially aware of the environment and must collect data before making a decision. In a
multi-period setting, it is natural to capture this uncertainty in
the environment through an unobservable stochastic process $M \equiv \{ M_t \}_{t \ge 0}$, where $M_t$ represents the state of the world
at time $t$. The DM starts with an initial guess about $M$, collects information via relevant
news, and updates her beliefs. At the time of decision she then receives a reward that depends
on the present environment, $H=H(\tau, k, M_\tau)$.

In such problems, a common approach is to postulate that the process $M$ is a partially observable Markov (decision) process (POMDP), in which case we have a hidden Markov model (HMM). We refer the reader to \cite{BertsekasBookOld}, \cite{ElliottBook} for a comprehensive treatment of discrete-time models and to \cite{BensoussanBook}, \cite{MR2001k:60001a} for continuous-time models and applications.

In both discrete- and continuous-time models the analysis separates the sub-problems of estimation (filtering of $M$) and control. The second ``control'' step requires re-formulating the problem under an equivalent fully observable system, where the conditional distributions/probabilities of the process $M$ constitute the new state variables. In discrete-time, the value function is typically a fixed point of the corresponding dynamic programming (DP) operator, and can be obtained via a recursive application of this operator; see, for example, the models and algorithms in \cite{BertsekasBookOld}, 
\cite{LovejoySurvey91} 
On the other hand, continuous-time formulations allow more sophisticated models, and the dynamic programming principle generally manifests itself in the form of a (partial) differential (delay) equation; see 
\cite{Friedman},
\cite{LenhartLiao}, 
 \cite{BensoussanBook}, 
 \cite[Chapter 6]{Peskir-Shiryaev-book} and the references therein for various examples.

The major distinction between discrete- and continuous-time models comes from the nature of the control and the observations; that is, is the system \emph{asynchronous} and observations/stopping can occur anytime, or are there fixed time epochs when new information is processed and stopping decisions are made. A similar distinction exists within  continuous-time models. If news (such as changes in asset prices) arrive in infinitesimal amounts, then it is intuitive to have a continuum of information, which is typically captured by the filtration of an observed diffusion process.  However, in many instances, a more realistic representation is to use ``discrete''
information amounts. Corporate developments, engineering failures, insurance claims, and
economic surveys are all discrete events and the corresponding news arrive in ``chunks''. Note
that discreteness of information is distinct from the discreteness of time. The model is  still
in continuous-time, since the events may take place at any instance. However, the event itself
carries a strictly positive amount of information. Moreover, ``no news'' is still informative
and affects the beliefs of the DM.

Mathematically, discrete information in continuous-time may be represented
by the filtration of an {observed} marked point process. 
In such a model, the instantaneous arrival intensity and the distribution of the marks of the point process typically depend on the current state of the process $M$. That is, the observable point process encodes information about the hidden environment $M$ via its arrival times and/or marks. Filtering with continuous-time point process observations has been considered in \cite{bremaud,Arjas92,ElliottMalcolm05}, and it is known that the dynamics of the conditional probabilities of $M$ are of the piecewise deterministic process (PDP) type. In other words, the DM beliefs evolve \emph{deterministically} between arrivals of new information, and experience random jumps at event times. From the control perspective, various aspects of optimal stopping of PDP's have been studied by \cite{LenhartLiao}, \cite{Gugerli} and \cite{CostaDavis88}.

In this paper, we study a class of finite-horizon
decision-making problems within the PDP framework by considering a general regime-switching model with Poisson information arrivals. Poissonian
information allows us to capture the discreteness of news while maintaining a rich framework
for the dependence of the observable $X$ on the unobservable state of $M$, which can manifest itself both in arrival
rate and mark distribution effects. In this context, our main contribution is the full characterization of the value function and optimal policy of the DM, with a direct proof of the dynamic programming principle and characterization of the optimal and $\epsilon$-optimal policies. Our approach also yields a numerical algorithm that can be readily implemented (see Section~\ref{sec:examples} for examples).
Within the PDP framework, related problems have been considered by \cite{JensenHsu} in
connection with system reliability studies, \cite{Jensen97} and \cite{Schottl} in the context
of insurance premium
re-pricing and \cite{PeskirShiryaev}, 
\cite{Gapeev},  \cite{bdk05}, \cite{dps} for classical Poisson disorder and regime detection problems.

Our model provides a non-trivial
generalization of previous analysis of decision making under Poissonian information structures.
More precisely, we extend existing literature in three directions. First, we consider a general
continuous-time finite-state Markov chain for the environment variable $M$ (without any
assumptions on the transition rates), and impose no restriction on the arrival rate and mark
distribution of the observed compound Poisson process $X$. The latter allows us to model any
setting where the DM also gets information via the \emph{size/type} of each event besides the
interarrival epochs. Second, we consider a general discount/cost structure, that can be used
to encode a variety of economic objectives.  Finally, we work
in the  context of finite horizon, where value functions are time-inhomogeneous. This is a more
realistic setting since a practicing DM typically has a well-defined ``window'' for making
their decision. The introduction of time-to-maturity as a state variable makes the numerical
computation more challenging and leads to appearance of new effects that are not possible with
stationary models. At the same time, our model allows a natural interpolation from finite to
infinite horizon; see Section \ref{sec:infinite}.

Before concluding our discussion here,
let us mention that the choice of ``discrete-time model'' versus ``continuous-time model with discrete information'' will be made according to the preferences of the modeler, as well as the nature of the problem. Accordingly, similar applications may invite different modeling approaches; for instance, the machine reliability problem discussed in Section~\ref{sec:catalogue} below was studied both in a discrete-time setting by \cite{Stadje}, a continuous-time setting  by \cite{JensenHsu} and even a hybrid continuous-time model with discrete-epoch observations in \cite{MakisJiang}. In this context, if the machine/production system is subject to major breakdowns, then continuous monitoring may be more desirable. In other cases, end-of-day inspections may be more than enough to restore the profitability of operations. While the aforementioned formulations are superficially similar (and in some specific cases even equivalent, see \cite{Feinberg04}), the respective solution methods utilize quite different tools. The solution of discrete time models generally relies on the Smallwood-Sondik property \cite{SmallwoodSondik} that shows that with finite state, observation, and action spaces the value function is piecewise linear and convex. In continuous-time this property no longer holds, and the smoothness of the value function must be independently established. Also
in discrete-time models decisions and controls are intrinsically paired with observations. In contrast, in the models considered here, the control may take place both at event time or between events, which is an important qualitative distinction.

\subsection{A catalogue of sample problems.}\label{sec:catalogue}
Since the framework studied throughout the paper is general, let us first
provide a number of motivating examples illustrating the applications in various
settings.

\subsubsection*{\textbf{Profit Maximization with Information Cost.}}
Let us consider an insurance company which is planning to launch a new policy/product to its
clients. The frequency of corresponding insurance claims and the severity of claim sizes are
not known precisely. Rather, they depend on the current quality of the insurance portfolio,
represented by a Markov process $ M = \{ M_t \}_{t \ge 0} $ taking values on some 
space $E \triangleq \{ 1, \ldots, n \} $. Once the policy is launched, it yields a random payoff
that depends on the current state of $M$ only. To model this, we say that when $M$ is at state
$i \in E$ at the launch-time, the random payoff is given by an independent random variable
$\Phi_i$ with some finite mean $\mu_i =\E[ \Phi_i]$.

Information about $M$ is obtained through the filed claims process $X = \{ X_t \}_{t \ge 0}$
received by the firm. The cumulative claim process has the form $X_t = \sum_{j=1}^{N_t} Y_j $
for $t \ge 0$. Here $N_t$ is the total number of claims up to time $t$, and $Y_j$ is the size
for the $j$'th claim for $j \in \N$. The process $N$ is a simple Poisson process with intensity
$\lambda_i $ whenever $M$ is at state $i \in E$. Moreover, if a claim is known to occur when
$M$ is at state $i$, the claim size is an independent random variable with distribution
$\nu_i$.

At any time prior to some terminal time $T < \infty$, the company may launch the product or
permanently abandon it. Alternatively, it can delay this decision to obtain more information on
$M$, and to increase the likelihood of catching $M$ at a favorable state. However, waiting for
additional information costs $c \le 0$ per unit time. Therefore, the company must decide how
long it observes $X$ prior to a decision, and what decision (launch vs. quit) should be
taken at that time.

Let $\tau \le T$ denote the decision time, and let the random variable $d \in \{ 0,1 \}$
indicate whether the product is released or abandoned. That is, on the event $\{ d=1 \}$ the
company launches the product, and on $\{ d=0 \}$ it quits. Clearly, the time
$\tau$ should be determined based on the observations from the claim process $X$, and the
choice of action $d$ should be determined solely by the information generated by $X$ until
$\tau$. 
Then, the objective of the company is to compute
\begin{align}
\label{def:R1}
\sup_{\tau , d} \E^{\vp} \left[ \int_0^{\tau} \! e^{- \rho t} c \,  dt + e^{- \rho \tau}  \,
1_{\{ d =1 \}}  \Bigl( \sum_{i \in E} \mu_i \cdot 1_{\{ M_{\tau} =i \}} \Bigr) \right]
\end{align}
over all such pairs $(\tau, d)$. In \eqref{def:R1}, $\rho > 0$ is a given discount rate used by
the company in reference to future revenues, and $\vp \equiv (\pi_1, \ldots , \pi_n ) \triangleq ( \,  \P(M_0 = i), \ldots ,  \P(M_0 = n) \, ) $ denote the initial beliefs of the company about the state of
$M$ at $t=0$.

A related problem has been considered on infinite horizon by \cite{Schottl} who maximizes
future risk reserves of the insurance company where at the time $\tau$ the company will
re-calculate its premiums. We also refer the reader to \cite{DecampsEtal}, and \cite{UluSmith} for
recent work on timing project commitment/abandonment in continuous and discrete time
respectively.

\subsubsection*{\textbf{Bayesian Regime Detection.}}
In this problem, a compound Poisson process $X = \{ X_t \}_{t \ge 0}$ is observed starting from $t=0$. The arrival rate $\lambda$ and mark distribution $\nu$ of $X$ are not known precisely. Rather they
depend on the static regime of the Markov process $M $ with $n$ absorbing states (i.e., $M_t
= M_0$ for all $t \ge 0$). Each state corresponds to the realization of one of the $n$ simple
hypotheses
\begin{align}
  \label{hypotheses}
  H_1 : ( \lambda, \nu) = (  \lambda_1, \nu_1), \quad
  \ldots\ldots \quad , \quad H_n : (\lambda, \nu) = (  \lambda_n,
  \nu_n),
\end{align}
with given prior likelihoods $\pi_i$, for $i = 1,\ldots,n$. The objective of the DM is to recognize the current
regime as quickly as possible, with minimal probability of wrong decision.

In earlier work on this problem, the trade-off between observing and stopping is generally modeled via the
Bayes risk
\begin{align}
\label{def:R3}
\E^{\vp} \left[
\tau + \sum_{k,i =1}^n \mu_{k,i} 1_{ \{ d = k , M_{0 } =i \} }  \right],
\end{align}
where $\tau$ is the decision time, $d \in \{ 1,\ldots, n\}$ represents the hypothesis selected
and $\mu_{k, i} \ge 0$ is the cost of selecting the wrong hypothesis $H_k$ when the correct one
is $H_i$.  The DM then needs to minimize \eqref{def:R3} and find a pair $(\tau ,
d)$, if one exists, that attains this infimum.

The infinite horizon version of \eqref{def:R3} was solved for the first time by
\cite{PeskirShiryaev} for a simple Poisson process with $n=2$. Later, \cite{Gapeev} provided
the solution (again with $n=2$), where the jump size is exponentially distributed under each
hypothesis, with the mean of the exponential distribution the same as the proposed arrival rate.
The solution for any jump distribution
and for $n \in \N$ was recently provided by \cite{dps}. Our model in this paper can be viewed
as the finite horizon version of that problem, where a decision must be made before a terminal
time $T<\infty$.

\subsubsection*{\textbf{Optimal Replacement Time of a Reliability System.}}
\cite{JensenHsu} consider an optimal stopping problem in reliability with a partially observed
Poisson process. The problem is to find when to discard or replace a machine/production-system
whose production quality deteriorates over time due to the usual wear-and-tear. The status of
the machine is modeled with a finite state Markov process $M$. The process moves from good
states to bad states over time. Eventually it ends in the $n$'th absorbing state which
represents an unacceptable quality level.

The DM observes the failure times $\sigma_1, \sigma_2, \ldots$ (the failures can also be
interpreted as defective items in the context of a machine); it is assumed that the
corresponding ``arrivals'' form a Poisson process whose intensity is $\lambda_i$ when the
current state of the process $M$ is $i \in E = \{ 1, \ldots, n\}$. Running the system in state
$i$ yields a net payoff $c_i \in \R$ per unit time. A high $c_i$ indicates that the machine is
profitable, while a negative $c_i$, including the assumed $c_n < 0$, means that the low quality
outweighs the benefits. At any time the DM can stop running the machine and replace it, with a
terminal cost of $\mu_i$ if the process $M$ happens to be in state $i \in E $ at that time.
\cite{JensenHsu} then solve the problem of maximizing
\begin{align}
\label{def:objective-function-in-reliability}
\E^{\vp} \left[  \int_0^{\tau}  \sum_{i \in E} c_i \, 1_{\{ M_{t} =i \}} dt
+   \sum_{i \in E} \mu_i \cdot 1_{\{ M_{\tau} =i \}}  \right],
\end{align}
over all random time $\tau$'s (whose value is determined by the history generated by the
arrival process) and under certain assumptions on the arrival rates $\lambda_i$'s, the
infinitesimal generator of $M$, and cost parameters $c_i, \mu_i$'s. Related models have
appeared in \cite{MakisJiang}, and \cite{Stadje} and go all the way to classical POMDP work by \cite{SmallwoodSondik}. In this paper, we consider that problem without any
parameter assumptions and with the additional finite horizon constraint $\tau \le T$.

\subsection{Problem description: a unifying framework.}
In the examples above, a DM observes a compound Poisson process $X$ with arrival rate $\lambda$, and mark/jump
distribution $\nu$. The local characteristics $(\lambda, \nu)$ of $X$
are determined by the current state of an \emph{unobservable} finite-state Markov process $M$.

At any time $\tau$ less than some $T< \infty$, the DM can stop and select an action $k$ from
the set $ \mathcal{A} \triangleq \{  1, \ldots, a \}$. If action $ k \in \mathcal{A}$ is taken,
this yields a terminal reward/payoff of
\begin{align*}
  \sum_{i \in E} \mu_{k,i} \cdot 1_{\{ M_{\tau} =i \}}
\end{align*}
as a function of the unobservable state of $M$. Here, $\mu_{k,i} $ is a given finite (not
necessarily positive) number. One can also interpret $\mu_{k,i}$ as the expected value of an
independent random variable $\Phi_{k,i}$ representing the uncertain payoff of taking action $k$
when $M_t =i$.
Also note
that if there is a time-lag between the decision and its realization, and if this delay is
independent, then $\mu_{k,i}$ can be assumed to be the expected discounted value
of this payoff.

The DM may alternatively delay her decision and continue to observe the process $X$ in order to
collect more information, or in order to stop later when $M$ appears to be in a better state.
Delaying the decision carries associated costs (rewards) due to the cost of observation or lost
opportunity (or operating revenues). We allow these terms to depend on $M$ and we assume that
an amount with present value
\begin{align*}
\int_0^{\tau} e^{- \rho t }  \left( \sum_{i \in E }  c_i 1_{ \{ M_{t} = i  \} } \right) dt
\end{align*}
is accumulated until the decision time $\tau$. Here $\rho \ge 0$ is the discount factor, and
$c_i$ is the instantaneous cost or revenue of running the system 
when $M$ is at state $i \in E$. We allow $\rho$ to be zero. This makes the formulation
suitable for non-financial application where the quality of the decision is more important than
its timing.

In this setup, the objective of the DM is to find an \emph{admissible} strategy that will
maximize her total expected reward and resolve the trade-off between \emph{exploring} (getting
more observations) and \emph{exploiting} (engaging in an action). An admissible strategy is a
pair $(\tau,d)$, where $\tau \le T$ is the decision time and $d \in A$ is the action selected
at this time. Since the DM collects information from observing $X$, the value of $\tau$ should
be determined by the information generated by $X$, namely $\tau$ must be a stopping time of the
filtration $\mathcal{F}^X$ of $X$. Also, the decision variable $d$ should be measurable with respect to the information $\mathcal{F}_{\tau}^X$ revealed by $X$ until $\tau$.
Let $\vp = (\pi_1, \ldots, \pi_n) \triangleq
\left( \P(M_0 = 1), \ldots, \P(M_0 = n) \right)$ be the initial (prior) beliefs of the DM about
$M$ and $\P^{\vp}$ the corresponding conditional probability law. Then the objective of the DM
is to compute
\begin{align}
\label{def:U}
U (T, \vp) \triangleq
 \sup_{ \tau \le T  ,\, d \in \Fc^X_{\tau}}
\E^{\vp} \left[  \int_0^{\tau} e^{- \rho t }   \left( \sum_{i \in E} c_i 1_{ \{  M_{t} = i  \} } \right) dt
+ e^{- \rho \tau}  \sum_{ k \in \mathcal{A} }  1_{\{  d =k \}}  \Biggl( \sum_{i \in E} \mu_{k,i} \cdot 1_{\{ M_{\tau} =i \}} \Biggr) \right] ,
\end{align}
and, if it exists, find an admissible pair $(\tau,d)$ attaining this value.

In Section \ref{sec:probStat} below we describe the formal setting of our model and show that the problem in \eqref{def:U} is equivalent to an optimal stopping problem in terms of the \emph{conditional probability process}, which is a piecewise deterministic process.
Section~\ref{sec:sequential} describes how the value function of this stopping problem can be computed  via a sequential procedure.
The results of Section~\ref{sec:sequential} are used in Section~\ref{sec:solution} in order to identify an optimal strategy and describe its properties.
Following this, Section~\ref{sec:otherCosts} explores alternative
objective functions that can be employed in our framework. 
Finally, in Section~\ref{sec:examples} we give numerical examples illustrating our results.
Most of the proofs are delegated to the Appendices at the end.

\section{Problem Statement}\label{sec:probStat}
\subsection{Model.}
Let $(\Omega, \mathcal{H}, \P)$ be a probability space hosting a continuous-time Markov process
$M$ taking values on $E \triangleq \{ 1, \ldots, n \}$, for $n \in \mathbb{N}$, and with infinitesimal generator
$Q=(q_{ij})_{i,j \in E }$. 
Also, we have a collection of independent compound Poisson
processes $X^{(1)} , \ldots, X^{(n)} $ with local parameters $ (\lambda_1, \nu_1), \ldots ,
(\lambda_n , \nu_n)$ respectively.
In terms of these independent processes, we define the observation process 
\begin{align}
\label{def:X}
X_t \triangleq X_0 + \int_{(0,t] }  \sum_{ i \in E} 1_{ \{ M_s =i \} } \, dX^{(i)}_s , \qquad t\ge 0,
\end{align}
which is a Markov-modulated Poisson process, also called a Cox process
(see \cite{Cox-Isham}).
In the remainder, we let $\sigma_0, \sigma_1 , \ldots$ denote the arrival times of the process
$X$:
\begin{align*}
\sigma_m \triangleq \inf \{ t > \sigma_{m-1} : X_t \ne X_{t- }\}, \qquad m \ge 1, \qquad \text{with $\sigma_0 \equiv 0$,}
\end{align*}
and the variables $Y_1, Y_2 ,\ldots$ denote $\R^d$-valued marks observed at these arrival times:
\begin{align*}
Y_m = X_{\sigma_m } - X_{\sigma_m- }, \qquad m \ge 1.
\end{align*}
Finally, to compute relative likelihoods of different marks, we introduce the total measure
$\nu$ defined as $\nu \triangleq  \nu_{1} + \ldots + \nu_n $, and we let $f_i(\cdot)$ be the
density of $\nu_i$ with respect to $\nu$.

\subsection{Conditional probability process.}
For a point in $D\triangleq \{ \vp \in \R_+^{n} : \pi_1 + \ldots + \pi_n =1 \}$, let $\P^{\vp}$ denote the probability measure (with the expectation operator $\E^{\vp}$) under which $M$ has initial distribution $\vp$. Moreover, let $\F \triangleq \{ \Fc^X_t \}_{t \ge 0} $ be the
filtration of the process $X$ in \eqref{def:X}. With this notation, we
define the $D$-valued \emph{conditional probability process} $\vP_t \triangleq \left(
\Pi_t^{(1)}, \ldots , \Pi_t^{(n)} \right)$ such that
\begin{align}
\label{def:Pi-i}
 \Pi^{(i)}_t = \P^{\vp} \{  M_t =i | \Fc^X_t \}, \quad \text{for $i \in E$, and $t \ge 0$}.
\end{align}
The process $\vP$ is clearly adapted to $\F$, and each component gives the conditional probability that the current state of $M$ is $\{
i\}$ given the information generated by $X$ until the current time $t$. Moreover, using standard arguments
as in \cite[pp. 166-167]{Sh78}, and
\cite[Proof of Proposition 2.1]{dps}, it can be shown that the problem in \eqref{def:U} is equivalent to a fully observed optimal stopping problem with the process $\vP$ as the new hyperstate.
More precisely,
the value function $U$ in (\ref{def:U}) can be written as
\begin{align}
\label{def:V}
 U(T, \vp) = V(T, \vp) \triangleq \sup_{ \tau \le T  } \E^{\vp} \left[   \int_0^{\tau} e^{- \rho t }  \, C(\vP_t ) dt +
  e^{- \rho \tau} H( \vP_{\tau})   \right] ,
\end{align}
in terms of the functions
\begin{align}
\label{def:h}
C(\vp) \triangleq  \sum_{i \in E }  c_i \pi_i \qquad \text{and} \qquad
H( \vp ) \triangleq \max_{k \in \mathcal{A} } H_k(\vp), \qquad \text{where} \qquad  H_k(\vp) \triangleq \sum_{i \in E} \mu_{k,i} \pi_i.
\end{align}
If there is a stopping time $\tau^*$ attaining the supremum in \eqref{def:V}, then the admissible strategy $( \tau^*, d(\tau^*) )$ is an optimal rule for the problem in \eqref{def:U} if we define
\begin{align}
\label{def:d-tau}
d(\tau) \in \arg\max_{k \in  \mathcal{A} } H_k(\vP_{\tau})  .
\end{align}

\subsection{Sample paths of $\vP$.}
Let us take a sample path of the observations process $X$, in which $m$-many arrivals are observed on
$[0,t]$. Let $(t_k)_{k \le m}$ denote those arrival times. If we know that the process $M$ stays
at the state $\{i\}$ without any transition, then the (conditional) likelihood of this path would be written
as $\P^{\vp} \{ \sigma_k \in dt_k , Y_k \in dy_k \, ; \, k\le m \, | \,M_s =i , s\le t \} = $
$$ [ \lambda_i e^{-\lambda_i t_1 } dt_1]
\cdots [\lambda_i e^{-\lambda_i (t_m - t_{m-1}) } dt_m]   e^{-\lambda_i (t_m - t_{m-1}) }  \prod_{k =1}^m [  f_i (y_k ) \nu(dy_k)]
= e^{-\lambda_i t } \prod_{k =1}^m  \lambda_i  dt_k  \cdot f_i (y_k ) \nu(dy_k).
$$ 
By construction, the observation process $X$ has independent increments conditioned on $M = \{ M_t\}_{t \ge 0}$.
Therefore,
we have 
\begin{multline}
\label{path-likelihood-given-M}
1_{\{ M_t =i\}} \cdot \P^{\vp} \Big\{ \sigma_i \in dt_i , Y_i \in dy_i \, ; \, i\le m \, \Big| \, M_s; s \le t \Big\} \\
= 1_{\{ M_t =i\}} \cdot \exp{ \left( - \int_0^t \sum_{i=1}^n \lambda_i  1_{ \{ M_{t_k} =i \} }  ds  \right) } \cdot  \prod_{k=1}^m \left( \sum_{j \in E}  1_{\{ M_{t_k} =j \}} [ \lambda_j  dt_k \cdot f_i (y_k ) \nu(dy_k)] \right).
\end{multline}
By taking the expectations of the expressions above, we obtain the unconditional likelihoods, in terms of which we give an explicit representation for the process $\vP$ in Lemma~\ref{lemm:vP-explicit} below.

\begin{lemma}
\label{lemm:vP-explicit}
For $i \in E$, let us define
\begin{align}
\label{def:L}
L_i^{\vp} ( t, m : (t_k, y_k ), k \le m)
\triangleq
\E^{\vp} \left[
1_{\{ M_t =i\}} \cdot e^{ - I(t) } \cdot  \prod_{k=1}^m \ell(t_k, y_k)
\right],
\end{align}
where
\begin{align}
\label{def:I} I(t) \triangleq \int_0^t \sum_{i=1}^n \lambda_i  1_{ \{ M_s =i \} }  \, ds \quad
\text{\normalfont and} \quad \ell(t, y) \triangleq \sum_{j \in E}  1_{\{ M_{t} =j \}} \lambda_j
\cdot f_j (y ) .
\end{align}
Also, let $L^{\vp} ( t, m : (t_k, y_k ), k \le m) \triangleq \sum_{j \in E} L_j^{\vp} ( t, m : (t_k, y_k ), k \le m)$. Then we have
\begin{align}
\label{vP-explicit}
\Pi^{(i)}_t &= \frac{L_i^{\vp} ( t, N_t : (\sigma_k, Y_k ), k \le N_t) }{ L^{\vp} ( t, N_t : (\sigma_k, Y_k ), k \le N_t)}
\equiv \left[  \frac{L_i^{\vp} ( t, m : (t_k, y_k ), k \le m) }{ L^{\vp} ( t, m : (t_k, y_k ), k \le m)}
\right]\Bigg|_{ m = N_t \, ;  \,  ( t_k=\sigma_k, y_k = Y_k )_{ k \le m} } ,
\end{align}
$ \P^{\vp}$-a.s., for all $t\ge 0$, and for $i  \in E$.
\end{lemma}

\begin{figure}
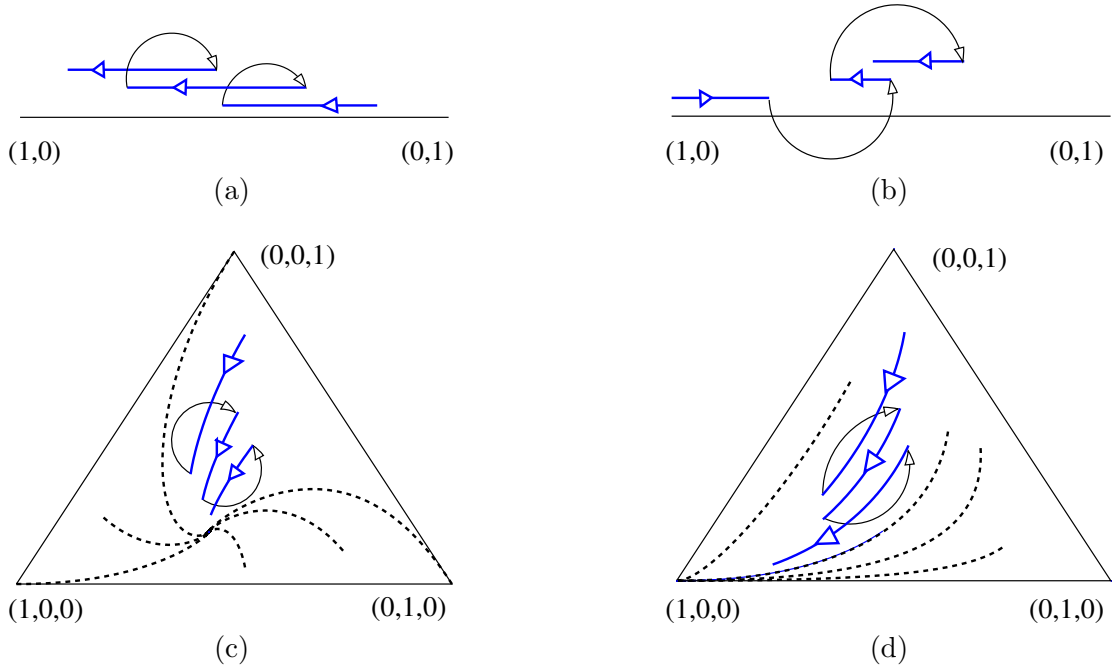

  \begin{center}
  \begin{tabular*}{0.9\textwidth}
     {@{\extracolsep{\fill}}cc}
   \includegraphics[width= 0.36\textwidth]{ex01-edited.mps}  &
     \includegraphics[width= 0.36\textwidth]{ex02-edited.mps} \\
 (a) & (b)\\
 &\\
   \includegraphics[width= 0.36\textwidth]{ex1-edited.mps}  &
     \includegraphics[width= 0.36\textwidth]{ex5-edited.mps} \\
      (c) & (d)
  \end{tabular*}
  \end{center}
  \small{
  \emph{\caption{
  Sample paths of the process $\vP$ for different examples. Solid lines represent actual sample paths. Dashed lines in panels (c) and (d) are the deterministic parts in \eqref{eq:x-i}. In panels (a) and (b), there are two hidden states, and in panels (c) and (d), there are three. In each example, jumps of the process $X$ are always of unit size. The parameters of each example:}
   \begin{align*}
   Q_a = \begin{pmatrix} 0 & 0  \\ 0 & 0 \end{pmatrix} ,
  \quad
   Q_b = \begin{pmatrix} -1 & 1  \\ 1 & -1 \end{pmatrix} ,
  \quad
  Q_c = \begin{pmatrix} -1 & 1  & 0 \\ 0 & -1 & 1 \\ 1 & 0 & 1 \end{pmatrix} ,
  \quad
  Q_d = \begin{pmatrix} 0 & 0  & 0 \\ 0 & 0 & 0 \\ 0 & 0 & 0 \end{pmatrix}
   \end{align*}
  with
$ \vec{\lambda}_a = [1 , 2], \,
 \vec{\lambda}_b = [1 , 4],  \,
  \vec{\lambda}_c = [1 , 2 , 3 ], \,
  \vec{\lambda}_d = [1 , 3 , 5 ]$.
     }}
     \label{fig:sample-paths} 
\end{figure}

Lemma~\ref{lemm:vP-explicit} indicates that the conditional probability of $M_t $ being in
state $i$ is simply the (unconditional) relative likelihood of the observed path until $t$ on the event $\{ M_t =i \}$. Using the explicit form in \eqref{vP-explicit}, we describe the behavior of the sample paths of $\vP$ in Remark~\ref{rem:sample-paths-of-pi} below. 

\begin{rem}
\label{rem:sample-paths-of-pi}
The process $\vP$ has piecewise-deterministic sample paths: between two arrival times 
of $X$, it moves deterministically, and at an arrival time, it jumps from one point to another depending on the observed mark size (see Figure~\ref{fig:sample-paths}). In precise terms, the sample paths have the characterization
     \begin{align}\label{eq:rel-pi-x} \left\{
     \begin{aligned}
    \vP(t)&=  \vx \left(t-\sigma_m,\vP({\sigma_m})\right), \qquad
\qquad \sigma_m \leq t< \sigma_{m+1}, \;\; m\in \mathbb{N} \qquad \\
\vP(\sigma_m)&=
\left(
\frac{  \lambda_1 f_1(Y_m) \Pi_1(\sigma_m-) }{ \sum_{j \in E} \lambda_j f_j(Y_m) \Pi_j(\sigma_m-) }, \ldots,
\frac{  \lambda_n f_n(Y_m) \Pi_n(\sigma_m-) }{ \sum_{j \in E} \lambda_j f_j(Y_m) \Pi_j(\sigma_m-) }\right)
\end{aligned} \right\},
\end{align}
where $\vx (t, \vp) \equiv (x_1(t, \vp), \ldots , x_n(t, \vp))$ is defined as
\begin{align}
\label{eq:x-i}
x_i(t, \vp) \triangleq \frac{    \P^{\vp}  \{ \sigma_1 > t , M_t =i  \}  }
 {   \P^{\vp}  \{ \sigma_1 > t  \} }
=\frac{  \E^{\vp} \left[
1_{\{ M_{t} =i\}} \cdot e^{ - I(t)}   \right]  }
{\E^{ \vp } \left[   e^{ - I(t)}   \right] }
 , \qquad \text{for $i \in E$,}
\end{align}
and satisfy 
the semigroup property $\vx(t+u , \vp ) = \vx (u , \vx(t,\vp) )$, for $t, u \ge 0$.
\end{rem}

The $i$'th component $x_i(\cdot, \cdot)$ indicates how likely it is to have a period of $[0,t]$ without any arrival on the event $\{ M_t =i\}$, as expected. Moreover, for $0 \le u_1 \le u_2 \le \ldots \le u_k$ and for a bounded function $g(\cdot)$, we have
\begin{multline}
\label{eq:Markov-justification}
\E^{\vp} \Bigl[ g( X_{t+u_1} - X_t , \cdots, X_{t+u_k} - X_t ) \big| \Fc^X_t \Bigr]\\
\begin{aligned}
&= \sum_{j \in E} \P \{ M_t =j \big| \Fc^X_t  \}
\cdot \E^{\vp} \Bigl[ g( X_{t+u_1} - X_t , \cdots, X_{t+u_k} - X_t ) \big| \Fc^X_t , M_t =j\Bigr]
\\
&= \sum_{j \in E} \Pi_{j}(t)
\cdot \E \Bigl[ g( X_{u_1} , \cdots, X_{u_k}  ) \big|  M_0 =j\Bigr]
= \mathbb{E}^{\vP_t}  \Bigl[  g( X_{u_1} , \cdots, X_{u_k} )  \Bigr] ,
\end{aligned}
\end{multline}
where the first equality in the last line follows from the construction of the process $X$ in
\eqref{def:X}. The equation \eqref{eq:Markov-justification} together with the characterization
in (\ref{eq:rel-pi-x}) implies that $\vP$ is a $(\P^{\vp}, \F)$-Markov process for every $\vp \in D$.

\begin{cor} \normalfont
\label{cor:vP-dyn}
Using infinitesimal last step analysis, it can be shown
(see, for example, \cite[page 416]{DarrochMorris}, and \cite[Chapter 6.7]{KarlinTaylor0})
that the vector
\begin{align}\label{def:m}
\vec{m} (t , \vp ) \equiv ( m_1 (t , \vp ), \ldots , m_n (t , \vp ) ) \triangleq \Bigl( \,  \E^{\vp} \left[ 1_{\{ M_{t} =1\}} \cdot e^{ - I(u)}   \right]  , \ldots ,
  \E^{\vp} \left[ 1_{\{ M_{t} =n\}} \cdot e^{ - I(u)}   \right] \, \Bigr)
\end{align}
has the form 
$
\vec{m} (t , \vp ) = \vp \cdot e^{ t( Q - \Lambda) }
$ 
where $\Lambda$ is the $n \times n$ diagonal matrix with $\Lambda_{i,i}  = \lambda_i$, and 
the components of $\vec{m} (t , \vp )$ solve $ d m_i (t, \vp ) / dt  = - \lambda_i m_i (t,\vp )
+ \sum_{j \in E}  m_j (t,\vp) \cdot q_{j, i}$. Then together with the chain rule and
(\ref{eq:x-i}) we obtain
\begin{align}
\label{eq:vx-dyn}
\frac{d x_i (t,\vp) }{dt} = \left( \sum_j^n q_{j,i}  x_j (t,\vp) - \lambda_i  x_i (t,\vp)  + x_i (t,\vp)   \sum_j^n \lambda_{j}  x_j (t,\vp)  \right). 
\end{align}
Hence, the process $\vP$ in \eqref{eq:rel-pi-x} has the dynamics 
\begin{align}
\label{dynamics-of-Pi}
d \Pi_t^{(i)} =   \left( \sum_j^n q_{j,i}  \Pi_{t-}^{(j)} - \lambda_i  \Pi_{t-}^{(i)}  + \Pi_{t-}^{(i)}   \sum_j^n \lambda_{j}  \Pi_{t-}^{(j)}  \right) dt +
\int_{\R^d } \left[ \frac{  \lambda_i f_i(y) \Pi^{(i)}_{t-} }{ \sum_{j \in E} \lambda_j f_j(y) \Pi^{(j)}_{t-} }   - 1 \right] p (dt , dy), \quad i \in E,
\end{align}
where $p (\cdot , \cdot)$ is the point process generated by $X$; that is
\begin{align*}
p \left(  (0,t] \times  B \right) = \sum_{i \in \N } 1_{  (0,t] \times  B  } (\sigma_i , Y_i ),
\qquad \text{for every Borel set $ B \in \mathcal{B} (\R^d)$ and $t \ge 0$}.
\end{align*}
\end{cor}

\section{Constructing the Value Function}
\label{sec:sequential}
The characterization of the sample paths in \eqref{dynamics-of-Pi}
and general theory of optimal stopping (see, for example, \cite{BensoussanBook,LenhartLiao}) imply that
the free-boundary problem associated with the optimal stopping problem in \eqref{def:V} has the form
\begin{align}
\label{eq:free-boundary-problem}
 \max \bigl\{  ( - \rho  + \mathcal{L} ) V(s, \vp) + C(\vp) \, ; \,  H(\vp) - V(s, \vp)  \bigr\}=0,
\end{align}
in terms of the infinitesimal generator
\begin{multline*}
  \mathcal{L}V (s, \vec{\pi}) =
\frac{\partial V( s, \vec{\pi} )}{\partial s}  +
 \sum_{i\in E}  \left( \sum_{j \in E}  q_{j,i}  \pi_j -\lambda_i \pi_i + \pi_i \sum_{j \in E } \lambda_j
\pi_j \right) \frac{\partial V(s,  \vec{\pi} )}{\partial \pi_i}  \\
 +
 \int_{y \in \R^d} \left[
 V\left( s, \frac{\lambda_1 \, \pi_1 \, f_1(y)}{ \sum_{j\in E} \lambda_j  \, \pi_j \,  f_j(y)},
 \ldots,  \frac{\lambda_n  \, \pi_n \, f_n(y)}{ \sum_{j \in E} \lambda_j \, \pi_j \, f_j(y)} \right) - V(s, \vec{\pi})
  \right] \sum_{i \in E} \pi_i \, \lambda_i \, \nu_i (dy),
\end{multline*}
of the process $\vP$. 
The infinitesimal generator $\mathcal{L}$ is a partial differential-difference operator on $[0,T] \times D \subset \R^{n+1}$. Hence, solving the equation $( - \rho  + \mathcal{L} ) V(s, \vp) + C(\vp)=0$ and determining the boundary of the region $\{ \vp \in D : V(T, \vp) = H(\vp) \}$ is not easy even when $n=2$; see, for example, \cite{PeskirShiryaev} who solve free-boundary problems similar to \eqref{eq:free-boundary-problem} for infinite horizon problems, and with $n=2$.

Instead of studying the problem in \eqref{eq:free-boundary-problem},
we will employ a sequential approximation technique to compute the value function
following \cite{Gugerli} and \cite[Chapter 5]{davis93}.
Similar approach is also taken in \cite{bdk05} and \cite{dps} 
for disorder-detection and hypothesis-testing problems respectively in infinite horizon. Since our problem is in finite-horizon, we work with time-dependent operators, and this requires non-trivial modifications of their arguments. 
The method is described in the sequel, and the proofs are given the Appendix.

\subsection{A sequential approximation}
Let us first define the 
functions
\begin{align}
\label{def:V-m-s}
\begin{aligned}
 V(s, \vp) &\triangleq  \sup_{\tau \le s } \E^{\vp} \left[
 \int_0^{\tau} \! e^{- \rho t }  \, C ( \vP_t  ) dt +
 e^{- \rho \tau   } H \left(\vP_{\tau  } \right) \right], \qquad \text{and}\\
V_m(s, \vp) &\triangleq  \sup_{\tau \le s } \E^{\vp} \left[
 \int_0^{\tau\wedge \sigma_m} \! e^{- \rho t }  \, C ( \vP_t  ) dt +
 e^{- \rho \tau \wedge \sigma_m  } H \left(\vP_{\tau \wedge \sigma_m } \right) \right] ,
\quad \text{for $m \in \N$, on $[0,T] \times D$,}
\end{aligned}
\end{align}
where the first argument `$s$' 
should be considered as the remaining time to maturity. 

Proposition~\ref{prop:uniform-convergence} below shows that $V_m$'s converge to $V$ uniformly; see also the proof of \cite[Theorem (53.40)]{davis93} and \cite[Proposition 3.1]{dps} for related results. Proposition~\ref{prop:uniform-convergence} is a generalization of these results in the finite horizon case.
\begin{prop}
\label{prop:uniform-convergence}
The sequence $\{ V_m \}_{m \ge 1}$ converges to $V$ uniformly on $[0,T] \times D$. More precisely, we have
\begin{align}
\label{eq:uniform-bound}
V_m(s, \vp) \le V (s, \vp) \le V_m(s, \vp)
+ \bigl( T\| C\| + 2 \| H \| \bigr) \left(   \frac{  \bar{\lambda} \,  T}{  m-1  }
\right)^{1/2} \! \cdot \left( \frac{ \bar{\lambda}  }{ 2 \rho + \bar{\lambda} } \right)^{m/2},
\end{align}
for all $(s, \vp) \in [0,T] \times D$ and $m \in \N$, where $\| C \| \triangleq \max_{\vp \in D} |C(\vp)|$, $\| H \| \triangleq \max_{\vp \in D} |H(\vp)|$ and $\bar{\lambda} \triangleq  \max_{i \in E} \lambda_i$.
\end{prop}

Let us consider the second problem in \eqref{def:V-m-s} for fixed $m \in \N$, and let $\tau \le s $ be a $\F$-stopping time. 
Then, the dynamic programming intuition suggests that $V(\cdot)$ should solve the equation $V_m(s, \vp) = J_0V_{m-1}(s,\vp)$, where the operator $J_0$ is defined as
\begin{align}
\label{eq:DPP}
\begin{aligned}
J_0w(s,\vp) \triangleq  \sup_{\tau \le s } \E^{\vp} \left[
\int_0^{\tau\wedge \sigma_1} e^{- \rho t }  \, C(\vP_t ) dt +
1_{ \{ \tau < \sigma_1  \} }   e^{- \rho \tau } H\left(\vP_{\tau} \right)
+
1_{ \{  \sigma_1  \le \tau \} }    e^{- \rho \sigma_1 } w  \left( s- \sigma_1 , \vP(\sigma_1 ) \right)
\right],
\end{aligned}
\end{align}
for a bounded function $w: [0,T] \times D  \mapsto \R $.

The following characterization of $\F$-stopping times is from \cite[Theorem T33, p.
308]{bremaud} and \cite[Lemma A2.3, p. 261]{davis93}.
\begin{lemma}\label{lem:bremaud}
For every $\F$-stopping time (bounded as $\tau \le s \le T$), and for every $m \in
\mathbb{N}$, there exists a $\Fc^X_{\sigma_m}$-measurable random variable $R_m $ such that
$\tau \wedge \sigma_{m+1}=(\sigma_m+R_m) \wedge
 \sigma_{m+1}$, $\P$-almost surely on $\{\tau \geq \sigma_m\}$.
\end{lemma}
Lemma \ref{lem:bremaud} implies that the supremum in \eqref{eq:DPP} can equivalently be taken
over deterministic times, in which case the same problem becomes 
\begin{align}
\label{def:J-0}
V_m(s, \vp) =   J_0 V_{m-1}(s, \vp) \triangleq  \sup_{ t  \in [0,s] } J V_{m-1}(t, s, \vp),
\end{align}
where the operator $J$ has the form
\begin{align}
\label{def:J}
J w (t, s, \vp) \triangleq
\E^{\vp} \left[ \int_0^{ t \wedge \sigma_1} e^{- \rho t }  \, C(\vP_t ) dt +
1_{ \{ t < \sigma_1  \} }   e^{- \rho t }    H\left(\vP_{t} \right)
+
1_{ \{  \sigma_1  \le t \} }   e^{- \rho \sigma_1 } w \left( s- \sigma_1 , \vP(\sigma_1 ) \right)
\right].
\end{align}
Note that, with the notation in \eqref{def:m},  we have
\begin{align*}
\P^{\vp} \left[ \sigma_1 > u \right] = \E^{\vp} \left[ e^{- I(u)} \right]
\quad \text{and} \quad
\P^{\vp} \left[ \sigma_1 \in du , M_u =i  \right] =  \E^{\vp}  \left[  \lambda_i 1_{ \{ M_u =i
\} } e^{- I(u)} \right] du =\lambda_i \, m_i(u,\vp) \, du,
\end{align*}
and using
the characterization of the paths in \eqref{eq:rel-pi-x} and \eqref{eq:vx-dyn}
the operator $J$ in \eqref{def:J} can be rewritten as 
\begin{multline}
\label{eq:J-expectations}
J w (t, s, \vp) =  \E^{\vp} \left[ e^{- I(t)} \right]
 \cdot e^{- \rho t }  \cdot  H\left(\vx( t, \vp ) \right)   \\
+  \int_{0}^{t} e^{- \rho u}     \sum_{i \in E}  
m_i(u,\vp) \cdot \Bigl(  C(\vx(u, \vp)) + \lambda_i \cdot S_i w(s-u, \vx(u, \vp)) \Bigr) du,
\end{multline}
in terms of the operators
\begin{align}
\label{def:S} S_i w(t, \vp) \triangleq \int_{\R^d} w \left(t, \, \frac{  \lambda_1 f_1(y) \pi_1
}{ \sum_{j \in E} \lambda_j f_j(y) \pi_j }, \ldots, \frac{  \lambda_n f_n(y) \pi_n }{ \sum_{j
\in E} \lambda_j f_j(y) \pi_j }\right) f_i(y) \nu(dy), \quad \text{for $i \in E$.}
\end{align}
\noindent The following lemmas provide basic properties of the operator $J_0$.
\begin{lemma}
\label{lem:prop-of-J} If $ w (\cdot,  \cdot) $ is bounded, then so is $  J_0w (\cdot,  \cdot) $
on $ [0,T] \times D $. If $w_1( \cdot, \cdot  ) \le w_2( \cdot, \cdot  ) $, then $J_0w_1(\cdot,
\cdot ) \le J_0w_2(\cdot, \cdot  ) $. Moreover, if the mapping $\vp \mapsto w(s, \vp )$ is
convex for each $s \in [0,T]$, so is $\vp \mapsto J_0w(s, \vp )$ for each $s \in [0,T]$.
\end{lemma}

\begin{rem}
\label{rem:inf-J-is-attained}
For a bounded continuous function $w(\cdot, \cdot)$ on $[0,T]\times D$,
the mapping $t \to Jw(t,s,\vp)$ is continuous on $[0,s]$ and 
$\sup_{t \in [u,s]}Jw(t,s,\vp)$
is attained for all $u \in [0,s]$.
\end{rem}

\begin{lemma}
\label{rem:continuity-of-J}
The operator $J_0$ preserves the continuity. That is, if $w(\cdot, \cdot)$ is a continuous function defined on $[0,T] \times D $, then $J_0 w(\cdot, \cdot)$ is also continuous.
\end{lemma}

Let us now define the sequence
\begin{align}
\label{def:v-n}
v_0 (s, \vp) \triangleq  H(\vp), \quad \text{and} \quad v_{m +1}(s, \vp) \triangleq J_0 v_m(s, \vp), \quad \text{for } m \ge 0 ,
\quad \text{on $[0, T] \times D$.}
\end{align}
\begin{lemma}
\label{lem:properties-of-v-m-s}
The sequence $ \{ v_m( \cdot, \cdot ) \}_{ m \in \N }$ is non-decreasing, hence the pointwise limit $v(\cdot, \cdot) \triangleq \sup_{m \in \N} v_m(\cdot, \cdot) $ is well defined on $[0, T] \times D$. Each $ v_m(\cdot, \cdot)$ is bounded and continuous on $[0, T] \times D$, and the mapping $\vp \mapsto  v_m(s, \vp)$ is convex for each $s \in [0,T]$.
\end{lemma}
\begin{proof}
Note that $v_1(s,\vp) = J_0 v_0(s, \vp) = J_0 H(s, \vp) = \sup_{t \in [0,s]} J_0 H(t, s, \vp) \ge J_0 H(0, s, \vp) = H(\vp)$. Let us assume that $v_{m} \ge v_{m-1}$ for some $m \in \N$. Then we get $v_{m+1}(s,\vp) = J_0 v_m(s, \vp) \ge J_0 v_{m-1}(s, \vp) = v_m(s, \vp)$ where the inequality follows due to Lemma~\ref{lem:prop-of-J}. Hence, the sequence is non-decreasing by induction.

The claim on continuity, boundedness and convexity clearly hold for $v_0(\cdot, \cdot) = H(\cdot)$. Then using Lemmas~\ref{lem:prop-of-J} and \ref{rem:continuity-of-J} it can be verified
inductively that these properties also hold for each $v_m $.
\end{proof}

\begin{prop}\label{prop:V-n-epsilon}
The sequences defined in \eqref{def:V-m-s} and \eqref{def:v-n} coincide. That is, we have $v_m (\cdot, \cdot) = V_m (\cdot, \cdot)$ for every $m \in \N$. 
\end{prop}

\begin{cor}
\label{cor:V-equal-v}
Propositions~\ref{prop:uniform-convergence} and \ref{prop:V-n-epsilon} imply $v(\cdot, \cdot) \triangleq \lim_{m \in \N}v_m(\cdot, \cdot) = \lim_{m \in \N} V_m(\cdot, \cdot) = V(\cdot, \cdot)$. 
By Lemma~\ref{lem:properties-of-v-m-s}, each $V_m(\cdot, \cdot) $ is continuous on $[0, T] \times D$. 
Then, the uniform convergence in Proposition~\ref{prop:uniform-convergence} implies that $V(\cdot, \cdot)$ is also continuous.
Finally, as the upper envelope of convex mappings $\vp \mapsto v_m(s,\vp)= V_m(s,\vp)$, the
mapping $\vp \mapsto  V(s, \vp)$ is again convex for each $s \in [0,T]$.
\end{cor}

Proposition~\ref{prop:V-n-fixed-point} below characterizes the value function $V(\cdot, \cdot )$ as the fixed point of the operator $J_0$ defined in (\ref{def:J-0}-\ref{eq:J-expectations}), which can also be thought of as the dynamic programming equation for the value function $V(\cdot, \cdot )$.

\begin{prop}\label{prop:V-n-fixed-point}
The value function satisfies $V(s, \vp) = J_0 V(s,\vp)$, and it is the smallest bounded solution of this equation greater than $H(\cdot)$.
\end{prop}
\begin{proof}
Using Lemma~\ref{lem:properties-of-v-m-s} and Corollary~\ref{cor:V-equal-v} we get
$V(s,\vp) =v(s,\vp)  = \sup_{n\ge 1} v_n(s,\vp) =  $
\begin{align*}
& \sup_{n \ge 1} \sup_{t \in [0,s]} Jv_{n-1}(t,s,\vp)   =
 \sup_{t \in [0,s]} \sup_{n \ge 1} Jv_{n-1}(t,s,\vp) = \sup_{t \in [0,s]} \sup_{n \ge 1} \E^{\vp} \left[ e^{- I(t)} \right]
 \cdot e^{- \rho t }  \cdot  H\left(\vx( t, \vp ) \right) \\
&\qquad \qquad \qquad \qquad  +  \int_{0}^{t} e^{- \rho u}     \sum_{i \in E}  m_i(u,\vp)
\cdot \Big(  C(\vx(u, \vp)) + \lambda_i \cdot S_i v_{n-1}(s-u, \vx(u, \vp)) \Big) du \\
 & = \sup_{t \le s} \E^{\vp} \left[ e^{- I(t)} \right]
 \cdot e^{- \rho t }  \cdot  H\left(\vx( t, \vp ) \right) \\
&\qquad \qquad \qquad \qquad +  \int_{0}^{t} e^{- \rho u}     \sum_{i \in E} m_i(u,\vp)
 \cdot \Big(  C(\vx(u, \vp)) + \lambda_i \cdot S_i v(s-u, \vx(u, \vp)) \Big) du \\
& = \sup_{t \in [0,s]} Jv(t,s,\vp) = \sup_{t \in [0,s]} JV(t,s,\vp),
\end{align*}
where the fifth equality is from \eqref{eq:J-expectations} and the sixth equality is by the
bounded convergence theorem since we have $\| v_m(\cdot, \cdot) \| \le \| v(\cdot, \cdot) \| \le \| H(\cdot) \| + T \| C(\cdot) \| $ for all $m \in \N$.

Let $W(\cdot, \cdot)$ be another solution of $W(s,\vp) = J_0 W(s,\vp)$, such that $W(s,\vp)
\ge H(\vp)= v_0(s,\vp)$. Applying Remark \ref{lem:prop-of-J} we obtain
$ W(s,\vp) =  J_0 W(s,\vp)  \ge \sup_{t \in [0,s]} Jv_0(t,s,\vp) = v_1(s,\vp)$. 
By induction, $W(s,\vp) \ge v_n(s,\vp)$ for all $n$ and hence $W(s,\vp) \ge  \lim_{n\to\infty} v_n(s,\vp) = V(s,\vp)$.
\end{proof}
We finally close this section with the following result which will be useful in
Section~\ref{sec:solution} in establishing an optimal stopping time.

\begin{lemma}\label{lemm:J-dyn-prog}
For deterministic times $u \le t \le s$, and for a bounded function $w(\cdot, \cdot)$ we have
\begin{align}
\label{eq:J-dyn-prog}
Jw(t,s,\vp) = Jw(u,s,\vp) + \P^{\vp} \left\{ \sigma_1 > u \right\} \cdot e^{-\rho u}\cdot
\Big( Jw(t-u, s-u,x(u,\vp) ) - H(x(u, \vp)) \Big)
\end{align}
\end{lemma}

\begin{cor}
\label{cor:J-dyn-prog-with-J-0}
Let $w$ be a bounded function as in Lemma~\ref{lemm:J-dyn-prog}. Taking the supremum in \eqref{eq:J-dyn-prog} for fixed $u$ and $s$ we obtain
\begin{align*}
\sup_{t \in[u,s] } Jw(t,s,\vp) = Jw(u,s,\vp) + \P^{\vp} \left\{ \sigma_1 > u \right\} \cdot e^{-\rho u}\cdot
\Big( J_0w( s-u,x(u,\vp) ) - H(x(u, \vp)) \Big),
\end{align*}
where $J_0$ is as defined in \eqref{def:J-0}.
\end{cor}

\section{An Optimal Strategy}
\label{sec:solution}

Recall that the process $\vP$ has right-continuous paths (with left limits), and the functions $V(\cdot, \cdot)$ and $H(\cdot)$ are continuous due to Corollary~\ref{cor:V-equal-v}. Hence the paths of the process $V(t, \vP_t) - H(\vP_t)$ are also right-continuous and have left limits. Therefore, for $\varepsilon \ge 0$ the random time
\begin{align}
\label{def:U-eps}
U_{\varepsilon } (s,\vp) \triangleq \inf \left\{ t \in [0,s] \, :\,  V(s-t, \vP_t)  - \varepsilon \le H(\vP_t)  \right\}
\end{align}
is a well-defined $\F$-stopping time. Observe that we have $U_{\varepsilon } (s,\vp) \wedge \sigma_1 = r_{\varepsilon } (s,\vp) \wedge \sigma_1$, where
\begin{align}
\label{def:r-eps}
r_{\varepsilon } (s,\vp) \triangleq \inf \left\{ t \in [0,s] \, :\,  V(s-t, \vx(t, \vp))  - \varepsilon \le H(\vx(t, \vp))  \right\},
\end{align}
which can be considered as the deterministic counterpart of \eqref{def:U-eps}.

\begin{rem}\label{rem:V-t-u-supremum}
For $r_{\varepsilon } (s,\vp)$ defined in \eqref{def:r-eps} we have
\begin{align}\label{eq:V-t-u-supremum}
\sup_{t \in[0,s]} JV(t,s,\vp) = \sup_{ t \in [ r_{\varepsilon } (s,\vp) ,  s] } JV(t,s,\vp).
\end{align}
\end{rem}
\begin{proof}
For $t < r_{\varepsilon } (s,\vp)$, Proposition ~\ref{prop:V-n-fixed-point} and Corollary~\ref{cor:J-dyn-prog-with-J-0} give
\begin{align*}
JV(t,s,\vp) = \sup_{ u \in [t,s]} JV(u,s,\vp) - \P^{\vp} \{  \sigma_1 >t \} e^{-\rho t} \Big( V(s-t,\vx(t,\vp) ) - H(\vx(t,\vp)) \Big).
\end{align*}
Since $t < r_{\varepsilon } (s,\vp)$ we have $V(s-t, \vx(t,\vp)) - H( \vx(t,\vp) ) > \varepsilon$. Hence
\begin{align*}
JV(t,s,\vp) &\le  \sup_{ u \in [t,s]} JV(u,s,\vp) - \varepsilon \P^{\vp} \{  \sigma_1 >t \} e^{-\rho t} \\ &\le
\sup_{ u \in [0,s]} JV(u,s,\vp) - \varepsilon \P^{\vp} \{  \sigma_1 >t \} e^{-\rho t}
< \sup_{ u \in [0,s]} JV(u,s,\vp) .
\end{align*}
Therefore the supremum in
$\sup_{t \in[0,s]} JV(t,s,\vp)$ must be achieved on $[r_{\varepsilon } (s,\vp), s]$ and \eqref{eq:V-t-u-supremum} follows.
\end{proof}

\begin{prop}
\label{prop:U-eps-optimal-rule}
The stopping time $U_{\varepsilon } (s,\vp)$ defined in \eqref{def:U-eps} is an $\varepsilon$-stopping time for the problem in \eqref{def:V}, i.e.,
\begin{align}
 \E^{\vp} \left[
 \int_{0 }^{ U_{\varepsilon}(s,\vp) } \! e^{- \rho t} C(\vP_t ) \, dt +
 e^{-\rho\,  U_{\varepsilon}(s,\vp) }   H\left(\vP( U_{\varepsilon}(s,\vp)  \right) \right] \ge V(s, \vp) - \varepsilon, 
\end{align}
for all $\varepsilon \ge 0$ and $(s,\vp) \in [0,T] \times D$.
\end{prop}

Before proceeding with the proof of Proposition~\ref{prop:U-eps-optimal-rule}, we first state an immediate consequence of this result.

\begin{cor}
\label{cor:U-0}
The stopping time $U_0 (T,\vp)$ is an optimal rule for the stopping problem of \eqref{def:V}, and
the pair $( U_0 (T,\vp) , d(U_0 (T,\vp)) ) $ is an optimal admissible strategy for the problem in \eqref{def:U}.
\end{cor}

\begin{proof}[\text{Proof of Proposition~\ref{prop:U-eps-optimal-rule}}]
Let us define
\begin{align}
\label{def:martingale-M} Z_t \triangleq \int_{0 }^{ t } e^{- \rho u} C(\vP_u ) \,du + e^{-\rho t}
\, V(s-t,\vP_t), \qquad t \in [0,s],
\end{align}
which is a bounded process on $t \in [0,s] \subseteq [0,T]$.
We will show that  the stopped process $\{ Z_{t \wedge U_{\varepsilon}(s,\vp)} \}_{t \in [0,s]} $ is
a martingale and satisfies
\begin{align}
\label{eq:M-martingale} \E^{\vp}[ Z_{U_{\varepsilon} (s,\vp) }] = Z_0 = V(s,\vp).
\end{align}
The process $Z$ captures the natural idea that one should not stop as long as the value
function (i.e.\ the continuation value) is larger than the immediate reward. Note that
$\varepsilon$-optimality of $U_{\varepsilon} (s,\vp) $ follows easily from
\eqref{eq:M-martingale} since this equality would imply
$V(s,\vp) 
=  \E^{\vp} \left[ Z_{U_\varepsilon (s,\vp) }\right]  =$
\begin{multline}
\label{eq:eps-optimality}
 \E^{\vp}\left[
  \int_{0 }^{ U_\varepsilon(s,\vp) } \!\!e^{- \rho t} C(\vP_t )\, dt + e^{-\rho U_\varepsilon(s,\vp)} V(s-U_\varepsilon(s,\vp), \vP_{U_\varepsilon(s,\vp)}) \right]
\le \E^{\vp}\left[   \int_{0 }^{ U_\varepsilon(s,\vp) } \!\!e^{- \rho t} C(\vP_t ) \,dt +
\right. \\
\left. e^{-\rho U_\varepsilon (s,\vp) } \left( H( \vP_{ U_\varepsilon (s,\vp ) } ) + \varepsilon \right) \right]
\le \E^{\vp}\left[   \int_{0 }^{ U_\varepsilon(s,\vp) } e^{- \rho t} C(\vP_t ) \, dt +
 e^{-\rho U_\varepsilon (s,\vp) } H( \vP_{ U_\varepsilon (s,\vp ) } )  \right]+ \varepsilon,
\end{multline}
due to regularity of the paths $t \mapsto V(t, \vP_t) - H(\vP_t)$. In the remainder of the proof we will show \eqref{eq:M-martingale} by establishing
\begin{align}
\label{eq:M-local-martingale} \E^{\vp}[ Z_{U_\varepsilon(s,\vp) \wedge \sigma_m}] = Z_0 ,
\qquad \text{for $m=1,2,\ldots,$ }
\end{align}
inductively. After taking the limit as $m \to \infty$ in the equality above, we will then obtain
\eqref{eq:M-martingale} due to bounded convergence theorem.

First, consider the equality \eqref{eq:M-local-martingale} for $m=1$. Recall that
$U_{\varepsilon}(s,\vp) \wedge \sigma_1 = r_{\varepsilon}(s,\vp) \wedge \sigma_1$. Then
$\E^{\vp}[ Z_{U_\varepsilon(s,\vp) \wedge \sigma_1}]= \E^{\vp} [ Z_{r_{\varepsilon}(s,\vp) \wedge
\sigma_1} ]=$
\begin{align*}
&  \E^{\vp}\Biggl[
 \int_{0 }^{  r_{\varepsilon}(s,\vp) \wedge \sigma_1} \!\! e^{- \rho t} C(\vP_t ) dt +
1_{\{ \sigma_1 \le r_{\varepsilon}(s,\vp)\}} \cdot e^{-\rho
\sigma_1} V(s-\sigma_1, \vP_{\sigma_1}) + 1_{\{ \sigma_1 > r_{\varepsilon}(s,\vp) \} } \cdot e^{-\rho r_{\varepsilon}(s,\vp)} H(\vP_{r_{\varepsilon}(s,\vp) }) \\
& \hspace{1.7in} +
1_{\{ \sigma_1 > r_{\varepsilon}(s,\vp) \} } \cdot e^{-\rho \, r_{\varepsilon}(s,\vp)} \left(
V ( s-r_{\varepsilon}(s,\vp), \vP_{r_{\varepsilon}(s,\vp) } ) - H( \vP_{r_{\varepsilon}(s,\vp) } ) \right) \Biggr] \\
& = JV(r_{\varepsilon}(s,\vp),s,\vp) + e^{-\rho \, r_{\varepsilon}(s,\vp)} \cdot \P^{\vp}\{ \sigma_1 > r_{\varepsilon}(s,\vp) \} \cdot \bigl(V(s-r_{\varepsilon}(s,\vp), \vx(r_{\varepsilon}(s,\vp),\vp)) -
H(\vx(r_{\varepsilon}(s,\vp),\vp)) \bigr)
\\ & = \sup_{ u \in [r_{\varepsilon}(s,\vp)  ,s]} JV(u,s,\vp),
\end{align*}
where we used Proposition ~\ref{prop:V-n-fixed-point} and
Corollary~\ref{cor:J-dyn-prog-with-J-0} for the last equality. By
Remark~\ref{rem:V-t-u-supremum}, we get
\begin{align*}
\E^{\vp}\left[ Z_{U_\varepsilon(s,\vp) \wedge \sigma_1}\right]  = \sup_{u \in
[r_{\varepsilon}(s,\vp) , s]} JV(u,s,\vp) = J_0 V(s,\vp) = V(s,\vp) = Z_0,
\end{align*}
and this establishes the result for $m=1$.

Now suppose by induction that \eqref{eq:M-local-martingale}
is true for $m \ge 1$ and consider the equality
\begin{align}
\label{eq:for-m-ge-1}
\begin{aligned}
&\E^{\vp} \left[ Z_{U_{\varepsilon}(s,\vp) \wedge \sigma_{m+1}} \right]  = \E^{\vp} \Bigl[
1_{\{ U_{\varepsilon}(s,\vp) < \sigma_1 \} } Z_{U_{\varepsilon}(s,\vp) } + 1_{\{
U_{\varepsilon}(s,\vp) \ge \sigma_1 \} } Z_{  U_{\varepsilon}(s,\vp) \wedge \sigma_{m+1} }
 \Bigr]  \\
& \hspace{1.2in}= \E^{\vp} \Biggl[
  1_{\{ U_{\varepsilon}(s,\vp) < \sigma_1 \} }
  \Bigl(
\int_{0 }^{ U_\varepsilon(s,\vp) } \!\! e^{- \rho t} C(\vP_t ) \, dt + e^{-\rho
U_{\varepsilon}(s,\vp)} V(s-U_\epsilon(s,\vp), \vP_{U_{\varepsilon}(s,\vp) } )
\Bigr) \\
&+ 1_{ \{ U_{\varepsilon}(s,\vp) \ge  \sigma_1  \} } \left( \int_{0 }^{ U_{\varepsilon}(s,\vp)
\wedge \sigma_{m+1} } \!\! e^{- \rho t} C(\vP_t ) \, dt +
 e^{-\rho \, U_{\varepsilon}(s,\vp) \wedge \sigma_{m+1} }
V \left( s- U_{\varepsilon}(s,\vp) \wedge
\sigma_{m+1} , \vP_{U_{\varepsilon}(s,\vp) \wedge \sigma_{m+1}} \right)
\right)
\Biggr].
\end{aligned}
\end{align}
On the event $\{ U_{\varepsilon}(s,\vp) \ge \sigma_1 \} $, we have $U_{\varepsilon}(s,\vp) \wedge \sigma_{m+1} = \sigma_1 + [U_{\varepsilon}(s,\vp) \wedge
\sigma_m]\circ \theta_{\sigma_1}$, where $\theta $ is the time-shift operator on $\Omega$; i.e., $X_t \circ \theta_s = X_{t + s}$. Using the strong Markov property of $\vP$, 
equation \eqref{eq:for-m-ge-1} becomes $\E^{\vp}[ Z_{U_{\varepsilon}(s,\vp) \wedge \sigma_{m+1}}]  =$
\begin{multline}
\label{eq:strong-markov}
 \E^{\vp} \Bigl[ 1_{\{ U_{\varepsilon}(s,\vp) < \sigma_1 \} }
\left( \int_{0 }^{ U_\varepsilon(s,\vp) } \! \! e^{- \rho t} C(\vP_t ) \, dt + e^{-\rho
U_{\varepsilon}(s,\vp)} V(s-U_\epsilon, \vP_{U_{\varepsilon}(s,\vp) } )
\right)\\
+ \int_{ 0 }^{ \sigma_1 } \! e^{- \rho t} C(\vP_t ) dt + 1_{ \{ U_{\varepsilon}(s,\vp) \ge
\sigma_1  \} } e^{-\rho \sigma_1}  \, f(s- \sigma_1 , \vP_{\sigma_1}) \Bigr],
\end{multline}
where $f(u, \vp ) \triangleq$
\begin{align}
\label{def:f-and-equal-V}
  \E^{\vp}
\left[ \int_{ 0 }^{ U_\varepsilon(s,\vp) \wedge \sigma_m } \! \! e^{- \rho t} C(\vP_t ) dt +
 e^{-\rho \, U_{\varepsilon}(s,\vp) \wedge \sigma_m }  V(u - U_{\varepsilon}(s,\vp) \wedge \sigma_m,
\vP_{U_{\varepsilon}(s,\vp) \wedge \sigma_{m}} ) \right] = V(u, \vp ),
\end{align}
by the induction hypothesis for $m$. Combining \eqref{eq:strong-markov} and
\eqref{def:f-and-equal-V} we get $ \E^{\vp}[ Z_{U_{\varepsilon}(s,\vp) \wedge \sigma_{m+1}}] =$
\begin{align*}
   &\E^{\vp}\Biggl[ 1_{ \{ U_{\varepsilon}(s,\vp) < \sigma_1 \} }
\left( \int_{0 }^{ U_\varepsilon(s,\vp) } \!\! e^{- \rho t} C(\vP_t ) dt + e^{-\rho
U_{\varepsilon}(s,\vp)} V(s-U_\epsilon, \vP_{U_{\varepsilon}(s,\vp) } )
\right) \Biggr] \\
&\hspace{1.5in}+ \E^{\vp}\Bigl[ 1_{ \{  U_{\varepsilon}(s,\vp) \ge \sigma_1  \} } \left( \int_{
0 }^{ \sigma_1 } e^{- \rho t} C(\vP_t ) dt + e^{-\rho \sigma_1} V(s-\sigma_1,\vP_{\sigma_1})
\right)
\Bigr] \\
& = \E^{\vp}\Bigl[ 1_{ \{ U_{\varepsilon}(s,\vp) < \sigma_1 \} } Z_{ U_{\varepsilon}(s,\vp) } +
1_{ \{  U_{\varepsilon}(s,\vp) \ge \sigma_1  \} } Z_{ \sigma_1 } \Bigr]
= \E^{\vp}\left[ Z_{U_{\varepsilon}(s,\vp) \wedge \sigma_1} \right] = Z_0,
\end{align*}
where the last equality follows from our result for $m=1$. Hence we have $ \E^{\vp}\left[
Z_{U_{\varepsilon}(s,\vp) \wedge \sigma_{m+1}} \right] = Z_0 $ and this gives
\eqref{eq:M-local-martingale} for $m+1$.
\end{proof}

\subsection{Stopping and continuation regions.} Let
\begin{align}
\label{def:stop-cont-regions}
\begin{aligned}
\mathcal{C}_T &\triangleq \left\{ (s,\vp) \in [0,T] \times D : V(s, \vp) > H(\vp) \right\}, \\
\Gamma_{T} &\triangleq \left\{ (s,\vp) \in [0,T] \times D : V(s, \vp) = H( \vp ) \right\}
\end{aligned}
\end{align}
denote the continuation and stopping regions respectively. The stopping region can further be
decomposed as the union $\cup_{k \in \mathcal{A} } \Gamma_{T,k}$ of the regions 
\begin{align}
\label{def:accept-reject-regions}
\begin{aligned}
\Gamma_{T,k} &\triangleq \left\{ (s,\vp) \in [0,T] \times D : V(s, \vp) = H_k(\vp) \right\},
\qquad k \in \mathcal{A},
\end{aligned}
\end{align}
where $H_k$ is defined in \eqref{def:h}.
Corollary~\ref{cor:U-0} states that in the optimal solution $\big( U_0(T,\vp) , d(U_0(T,\vp))
\big)$, one observes the process $\vP$ until $U_0(T,\vp)$, whence it enters the region
$\Gamma_T$. At this time, if $\vP$ is in the set $\Gamma_{T,k}$ we take $d(U_0(T,\vp))=k$; that
is, we select the $k$'th action in the action set $\mathcal{A}$.
\begin{rem}\label{rem:stoppingHorizon} \normalfont
The definition of the value function $V$ in \eqref{def:V} implies that the mapping $ s \mapsto
V(s,\vp)$ is non-decreasing.
Therefore if $(s,\vp) \in \Gamma_{T,k}$ for some
$(s,\vp) \in [0,T] \times D$, then we have $(t,\vp) \in \Gamma_{T,k}$ for all $t \le s$. In
other words, each region $\Gamma_{T,k}$ is growing
and the continuation region $\mathcal{C}_T$ is shrinking as time to maturity decreases.
\end{rem}
\begin{rem} \normalfont \label{lem:convexAdoption}
For fixed $s \le T$, let $(s,\vp_1) $ and $(s,\vp_2) $ be two points in the 
region
$\Gamma_{T,k}$, and let $\alpha \in(0,1)$. As the upper envelope of convex mappings $\vp \to
v_m(s,\vp)$ (see Lemma~\ref{lem:properties-of-v-m-s} and Corollary~\ref{cor:V-equal-v}), the
mapping $\vp \to V(s,\vp)$ is convex for each $s \in [0,T]$. Using this property we obtain
\begin{multline*}
H_k(\alpha \cdot \vp_1 + (1- \alpha )\cdot \vp_2) ) \le V(s, \alpha \cdot \vp_1 + (1- \alpha )\cdot \vp_2)
\le \alpha \cdot V(s , \vp_1) + (1- \alpha ) \cdot V(s , \vp_2) \\ = \alpha  \cdot H_k( \vp_1) + (1- \alpha )  \cdot H_k(  \vp_2) = H_k (\alpha \cdot \vp_1 + (1- \alpha )\cdot \vp_2) ),
\end{multline*}
which implies that 
$(s,  \alpha \cdot \vp_1 + (1- \alpha )\cdot \vp_2) \in \Gamma_{T,k}$, and the region $\Gamma_{T,k} \cap (  \{ s\} \times D)$ is convex for each fixed $s \le T$ and $k\in A$.
\end{rem}

\begin{rem}\label{rem:stoppingNeverEmpty} \normalfont
The stopping region is never empty since the decision maker has to select an action
eventually, the latest at the terminal time $T$. That is, $\Gamma_{T} \supseteq \{  (0, \vp );
\vp \in D \} \ne \emptyset  $. The region $\{ (s,\vp) \in \Gamma_{T} : s >0\} $  may however be empty. In an example where $ \min_{i \in E} c_i   > 0 $ and $ \mu_{k,i} $'s are all the same it is never optimal to stop prior to terminal time $T$.

Note that the region $\{ (s,\vp) \in \Gamma_{T} : s >0\} $ may be non-empty but still may
have an empty interior. For example, let us consider the hypothesis testing in \eqref{def:R3}.
In this \emph{minimization} problem, all the states of the unobservable Markov process are
absorbing, and each component $\Pi^{(i)}_t = \P \{ M_t =i | \Fc^X_t \}=   \P \{  M_0=i | \Fc^X_t
\}$ of process $\vP$ is a martingale. Since the terminal reward function of the corresponding
stopping problem (see \eqref{def:h}) $H(\cdot) = \min_{k \in E} H_k(\cdot)$ is concave, the
process $H(\vP_t)$ is a supermartingale on $[0,T]$. If we select $\rho=0$ and $c_i =0 $ for all
$i \in E$ in \eqref{def:R3}, it is therefore never optimal to stop early on the interior of $\{
(s,\vp) \in \Gamma_{T} : s >0\} $. In this case, there is no penalty associated with a delay in
the decision. Hence the DM will choose to observe it as much as possible prior to a decision
unless she knows for sure which hypothesis is correct.
\end{rem}

\begin{lemma}
\label{lemm:continue-at-good-states}
For $i \in E$, let $\mathcal{A}^*(i) \triangleq \{ k \in \mathcal{A} : \, \mu_{k,i} = \max_{j \in \mathcal{A} } \mu_{j,i} \}$. If the inequality $c_i - \rho \mu_{k,i} + \sum_{j \ne i} (\mu_{k,j}- \mu_{k,i}) q_{i,j} > 0$ holds for all $k \in \mathcal{A}^*(i)$, then there exists $\pi^c_{i} < 1$ such that $\{ (s , \vp) \in (0,T] \times D: \, \pi_i \ge \pi^c_i  \} \subseteq \mathcal{C}_T$. Moreover, $\pi_i^c $ can be selected independent of $T$. 
\end{lemma}
If the hidden process $M$ is known to be in state $i \in E$, then the expression $ -\rho
\mu_{k,i}$ is the instantaneous decay of the payoff from selecting action $k \in \mathcal{A}$ immediately, and $c_i$ is the
instantaneous cost of waiting. Moreover, under action $k \in \mathcal{A}$, the term $\sum_{j \ne i} (\mu_{k,j}- \mu_{k,i})
q_{k,j}$ is the marginal rate of return from waiting for the hidden process $M$ to jump to
another state. Therefore the sum of these three terms appearing in Lemma
\ref{lemm:continue-at-good-states} is the instantaneous net return enjoyed by the DM under
action $k \in \mathcal{A}$. Lemma~\ref{lemm:continue-at-good-states} indicates that if there is strong posteriori evidence that $M$ is in state $i$, and if the instantaneous net return is positive under all favorable actions (whose terminal reward $H_k$ dominates others around the $i$'th corner of $D$), the decision maker should not stop at that point (unless $T=0$).

\subsection{Stopping regions for reward maximization with running cost.}
Here, we consider the problem in \eqref{def:V} with the assumption $c_i \le 0$ (running costs)
for $i \in E$, and $\bar{\mu} \triangleq  \max_{k,i} \mu_{k,i} >  0$ (terminal rewards).
The second condition is not restrictive if $\rho =0$ since we can always add (and
subtract) the same constant to (and from) the terminal reward function.

Let us define
\begin{align}
\label{optimal-notation}
   I^* \triangleq \{ i \in E : \, \max_{k \in \mathcal{A}} \mu_{k,i} = \bar{\mu} \},
 \end{align}
 which is the set of the states of $M$, at which the DM can get the highest terminal reward.
Since $c_i \le 0$ for all $i \in E$, we have $\cup_{i \in I^* } \{ (s,\vp) : s\in [0,T] \, , \,
\pi_{i} =1  \}  \subset \Gamma_{T} $. That is the DM stops whenever the process $\vP$ reaches a
point of global maximum of the terminal reward function $H(\cdot)$.

In general, if there is a penalty associated with waiting, we expect that it is optimal to stop on the points $(s,\vp)$ for which the ``best'' component $\pi_{i}$, $i \in I^*$, 
is sufficiently high, for any $s> 0$. Lemma~\ref{lem:nonempty-acceptance-region} provides a sufficient condition for this to be true. It implies that if the discount rate is  strictly positive, or if the cost of waiting for the highest reward is strictly positive, then we stop whenever $\pi_{i}$, for $i \in I^*$, is relatively high regardless of the remaining time to maturity.

\begin{lemma}
\label{lem:nonempty-acceptance-region}
Let $i \in I^*$. If $\rho > 0 $, or $ c_i < 0$, then
there exists a number $\pi^{s}_{i}  < 1$ such that
\begin{align*}
\Gamma_{T} 
\supseteq \{ (s, \vp)  \in [0,T] \times D \, : \, \pi_{i} \ge \pi^{s}_{i} \},
\end{align*}
and the value of $\pi^{s}_{i} $ can be selected free of the time to maturity $T$.
\end{lemma}

\begin{rem}\label{rem:VincreasesInRho}\normalfont
If $H(\cdot) \ge 0$, the statement of the stopping problem in \eqref{def:V} implies
that the value function $V$ is non-increasing as a function of the discount factor $\rho$. If
we denote the dependence of the stopping region on $\rho$ with $\Gamma_T(\rho)$, then we have
$\Gamma_T(\rho_1) \subseteq \Gamma_T(\rho_2)$ whenever $\rho_1 \le \rho_2$.
Moreover, the dynamics of the process $\vP$ are independent of $\rho$ and $U_0(s,\vp)$ is the hitting time of $\vP$ to $\Gamma_T$. Therefore, the time that the DM can afford for observing the process $X$ in the presence of a lower discount factor is no less than that spent under heavier discounting.

A similar claim also holds for dependence of $U_0(s,\vp)$ and $\Gamma_T$ on the running costs $c_i$. Namely, an observer with lower (in absolute value) running costs stops no sooner than another one with heavier running costs.
\end{rem}

\subsection{A nearly-optimal strategy.}
\label{subsec:nearly-optimal-rules}
On a practical level, one cannot compute $V$ directly, but instead computes the approximate
value functions $V_m$'s defined in \eqref{def:V-m-s} and employs the corresponding
nearly-optimal strategies (see \ref{def:m-eps-stopping-rule}). It is therefore important to
know the error associated with this approximation.

For a given error level $\varepsilon > 0$, let us fix
\begin{align*}
m = \inf \left\{ k \in \N :\; ( T \| C \| + 2\,  \| H \| ) \left(   \frac{  \bar{\lambda} \,  T}{  k-1  }
\right)^{1/2} \cdot \left(  \frac{ \bar{\lambda}  }{ 2 \rho + \bar{\lambda} } \right)^{k/2} \le
\varepsilon /2 \right\} ,
\end{align*}
such that $\| V_m - V \| \le \varepsilon /2  $ on $[0,T] \times D$ via
\eqref{eq:uniform-bound}. Next, let us define the stopping times
\begin{align}
\label{def:m-eps-stopping-rule}
U^{(m)}_{\varepsilon/2}(s,\vp) \triangleq \inf \{ t \in [ 0, s ]  :\,   V_m(s,\vP_t) -\varepsilon/2 \le H(\vP_t)  \}.
\end{align}
The regularity of the paths $t \mapsto \vP_t$ implies that $V \left(  U^{(m)}_{\varepsilon/2}(s,\vp),
\vP_{U^{(m)}_{\varepsilon/2}(s,\vp)} \right) - H\left( \vP_{U^{(m)}_{\varepsilon/2}(s,\vp)} \right) \le
\varepsilon$. Then the arguments in the proof of  Proposition~\ref{prop:U-eps-optimal-rule}
(see \eqref{def:martingale-M}, \eqref{eq:M-martingale}, and \eqref{eq:eps-optimality}) can easily be
modified to show that
\begin{align}
\label{ineq:m-eps-optimality}
\begin{aligned}
V(s,\vp) &=  \E^{\vp}\left[
  \int_{0 }^{ U^{(m)}_\varepsilon(s,\vp) } \!\! e^{- \rho t} C(\vP_t ) \,dt + e^{-\rho U^{(m)}_\varepsilon(s,\vp)} V \left(s-U^{(m)}_\varepsilon(s,\vp), \vP_{U^{(m)}_\varepsilon}(s,\vp)\right) \right]
\\
&\le \E^{\vp}\left[
  \int_{0 }^{ U^{(m)}_\varepsilon(s,\vp) } \!\! e^{- \rho t} C(\vP_t ) \,dt + e^{-\rho U^{(m)}_\varepsilon(s,\vp)} H\left( \vP_{U^{(m)}_\varepsilon(s,\vp)}\right) \right] + \varepsilon.
\end{aligned}
\end{align}
Hence, if we apply the admissible strategy $\left(U^{(m)}_\varepsilon(T,\vp), d(U^{(m)}_\varepsilon(T,\vp))\right)$, which requires computing \eqref{def:V-m-s} only up to $m$ defined above, the resulting error is no more than $\varepsilon$.

\subsection{Infinite horizon problem as an approximation}
\label{sec:infinite}
In general, if there is a strict penalty for waiting,
it is likely that the DM will make a decision prior to the final time $T$
for moderate or large values of $T$. In this case,
the constraint $\tau \le T$ in \eqref{def:V} is of less importance, and
one essentially faces an \emph{infinite horizon} stopping problem. Solving the infinite horizon problem can be computationally more appealing since we eliminate the time-dimension of the state space $[0,T] \times D$. Below, we show
that the value function of the finite-horizon problem converges uniformly to that of the infinite horizon under the assumption
\begin{align}
\label{in-finite-convergence-assumptions}
\text{``either $\rho > 0$'' or ``$ \max_{i \in E} c_i < 0$''.}
\end{align}
The infinite horizon problem is defined as in \eqref{def:V} (and \eqref{def:U}) by removing the constraint $\tau \le T$.
With the notation in \eqref{def:V}, let 
$ V(\infty, \vp)$  be the value function of this stopping problem.

\begin{lemma}
\label{lem:convergence-of-finite-to-infinite}
As $T \nearrow \infty$, the function $V(T,\vp)$ 
converges to
$V(\infty,\vp)$ uniformly on $D$, and we have
\begin{align}
\label{eq:convergence-of-finite-to-infinite}
V(T, \vp) \le V(\infty, \vp) \le V(T, \vp) +    Err(T)  , \quad \; \text{for all $\vp \in D$ and $T \ge 0$,}
\end{align}
where
\begin{align*}
Err(T) \triangleq \left\{
\begin{aligned}
&e^{- \rho T} ( \|  C\| + 2 \cdot \|  H \|) &,& \quad \text{\normalfont if $\rho > 0$} \\
 &
\frac{2 \cdot \|  H \|}{T} \, \frac{\big( \min_{k,i}\mu_{k,i} - \max_{k,i} \mu_{k,i} \big)}{ \max_{i \in E} c_i}
 &,& \quad \text{\normalfont if $\rho = 0$ and $ \max_{i \in E} c_i < 0$.}
\end{aligned}
\right\}
\end{align*}
\end{lemma}

The explicit error bounds for the rate of convergence allows to approximate $V(T,\cdot)$ with the value function of the infinite horizon problem when $T$ is large. The function $V(\infty, \vp)$ can be computed sequentially as in Section~\ref{sec:sequential}. That is, if we define the non-decreasing sequence
\begin{align}
\label{def:V-m-infty}
 V_m( \infty , \vp) \triangleq   \sup_{ \tau \ge 0 } \E^{\vp} \left[   \int_0^{\tau \wedge \sigma_m} \! e^{- \rho t} C(\vP_t) \, dt +   e^{- \rho \tau \wedge \sigma_m } H( \vP_{\tau  \wedge \sigma_m })   \right] ,
 \qquad m \in \N,
\end{align}
then it can be shown that the elements of this sequence can be computed by applying a functional operator $\hat{J}_0$, which is obtained from the operator $J_0$ in \eqref{eq:J-expectations} after replacing the constraint $t \in [0,s] $ with $t \ge 0 $. Also, note that the new operator $\hat{J}_0$ is defined on the domain of functions defined on $D$ only. The proof of these statements can be obtained by modifying the arguments of Section 3, or those in
\cite[Section 3]{dps}. Moreover, following the proof of Proposition~\ref{prop:uniform-convergence}  and the arguments of Section~\ref{subsec:nearly-optimal-rules}, we have
\begin{align*}
\| V_m( \infty , \cdot) -  V(\infty, \vp) \| \le
Err_{\infty}(m) \triangleq
\left\{
\begin{aligned}
&\left( \frac{ \bar{ \lambda} }{ \rho + \bar{ \lambda} } \right)^m &,& \quad \text{\normalfont if $\rho > 0$,} \\
 &  \left(\frac{\max_{k,i}\mu_{k,i}}{ \max_{i \in E} c_i}\cdot  \frac{ \bar{ \lambda} }{ m-1  }\right)^{1/2} &,& \quad \text{\normalfont if $\rho = 0$ and $ \max_{i \in E} c_i < 0$,}
\end{aligned}
\right.
\end{align*}
and the stopping time
\begin{align}
\label{def:eps-stopping-time-for-truncated-inifinite-horizon}
U^{(m)}_{\varepsilon}(\infty, \vp) \triangleq \inf \left\{ t \ge 0 : \; V_m (\infty , \vP_t) - \varepsilon  \le  H(\vP_t )\right\}
\end{align}
is $\varepsilon$-optimal for the infinite horizon problem (see also \cite[Section 4.1]{dps}).

Note that for large $m$, the function $V_m (\infty, \cdot)$ approximates the function
$V(\infty, \cdot)$, and for large $T$, $V(\infty, \cdot)$ is a good approximation for $V(T,
\cdot)$. However, the stopping rule in
\eqref{def:eps-stopping-time-for-truncated-inifinite-horizon} is not a good substitute for the
optimal time $U_{0}(T, \vp)$ since the former may not be less than $T$ almost surely. Moreover,
since $U^{(m)}_{\varepsilon}(\infty, \vp)$ may be greater than $U_{0}  (T, \vp)$,
Proposition~\ref{prop:U-eps-optimal-rule} is not necessarily true. In particular, the
martingale property \eqref{eq:M-martingale} may fail. Nevertheless, if we apply the rule
$U^{(m)}_{\varepsilon}(\infty, \vp) \wedge T $, we can still control the error for large $T$.
Indeed, in Appendix~\ref{sec:appendix}, we show that 
  \begin{multline}
  \label{explicit-error-for-T-truncated-stopping-time}
V(T,\vp) \le
  \E^{\vp}\left[
  \int_{0 }^{ U^{(m)}_{\varepsilon}(\infty, \vp)  \wedge T } \!\! e^{- \rho t} C(\vP_t ) \,dt +
  e^{-\rho ( U^{(m)}_{\varepsilon}(\infty, \vp)  \wedge T)} H\left(  \vP_{U^{(m)}_{\varepsilon} (\infty, \vp) \wedge T} \right)
  \right] \\
  +  \varepsilon + Err_{\infty}(m) +    Err_{\infty}(0) \cdot Err(T).
    \end{multline}
Hence, if $T$ is large enough (so that $Err_{\infty}(0) \cdot Err(T)$ is small), by taking
$\varepsilon $ in \eqref{def:eps-stopping-time-for-truncated-inifinite-horizon} small for a
large value of $m$, the error associated with applying  $U^{(m)}_{\varepsilon}  \wedge T$ can
be reduced to acceptable levels.

\section{Discrete information costs}
\label{sec:otherCosts}
As the case studies of Section~\ref{sec:intro} 
demonstrate, the objective function in \eqref{def:U} is applicable to a variety of economic
settings. This has allowed us to provide a unified treatment of many disparate models.
Returning to the economic interpretation of the running costs appearing in the first term in
\eqref{def:U},
in a typical setting they represent \emph{information acquisition expenses}, such as
observation expenses, subscription costs to market data and holding outlays.
In such a case, it is natural to model the total cost incurred by decision time $\tau$ as the
sum $\int_0^\tau \e^{-\rho t}c \, dt$ where $c$ is interpreted as nominal running cost and
$\rho$ is the interest rate.

Alternatively, the costs can correspond to \emph{opportunity costs}, e.g.\ if $M$ is the
profitability of a new product then the opportunity costs of not launching the product should
depend on $\{M_t\}_{t \in [0,\tau]}$. This
motivates the consideration of $\int_0^\tau \e^{-\rho t}c_i 1_{ \{ M_t = i \} }\, dt$ where $c_i \in \R$ 
and $\rho$ can again be interpreted as the discount factor.

Finally, observation costs may be discrete and be incurred only when new information arrives.
This, for example, happens if new information corresponds to opportunities lost (e.g.\ deals
signed by competitors), leading to a cost structure of the form $\sum_{j=1}^{N_\tau} e^{-\rho
\sigma_j} K(Y_j)$. Here, $N_\tau$ is the number
of arrivals by time $\tau$, $(\sigma_j, Y_j)$ are the arrival times and marks respectively, and
$K(Y_j)$ is the cost incurred upon an arrival of size $Y_j$ (with $K: \R^d \mapsto \R$ satisfying $\nu_i K^+ \triangleq \int_{\R^d} K^+(y) \nu_i (dy) < \infty $, $\forall i \in E$).
%

In the third case, one deals with the objective function
\begin{align}
\label{def:U-alternative} \hat{U} (T, \vp) \triangleq
 \sup_{ \tau \le T  \, ,\, d \in \Fc^X_{\tau}}
\E^{\vp} \left[
\sum_{j=1}^{N_\tau} e^{-\rho \sigma_j} K(Y_j)+ e^{- \rho \tau}  \sum_{ k=1 }^a 1_{\{  d =k \}}
\Bigl( \sum_{i \in E} \mu_{k,i} \cdot 1_{\{ M_{\tau} =i \}} \Bigr) \right] ,
\end{align}
by solving the equivalent stopping problem
\begin{align*}
\hat{V} (T, \vp) \triangleq
 \sup_{ \tau \le T } 
\E^{\vp} \left[
\sum_{j=1}^{N_\tau} e^{-\rho \sigma_j} K(Y_j)+ e^{- \rho \tau} H \left( \vP_{\tau }  \right)
\right] ,
\end{align*}
as in Proposition~\ref{def:V}. One can verify that the sequential approximation method
of Section~\ref{sec:sequential} holds for the function $\hat{V}$. Namely, if we define the
sequence
\begin{align*}
\hat{V}_m (s, \vp) \triangleq
 \sup_{ \tau \le s  } 
\E^{\vp} \left[
\sum_{j=1}^{m \wedge N_\tau}   e^{-\rho \sigma_j} K(Y_j) + e^{- \rho \tau \wedge \sigma_m } H
\left( \vP_{\tau \wedge \sigma_m}  \right) \right] ,
\qquad m \in \N,
\end{align*}
it can be shown (see (\ref{def:J-0}-\ref{def:S}), Proposition~\ref{prop:V-n-epsilon}) that we
have $\hat{V}_{m+1} (s,\vp) = \hat{J}_0 \hat{V}_{m} (s,\vp) $ where the operator $\hat{J}_0 $
is defined as
\begin{multline*}
\hat{J}_0 w ( s, \vp) = \sup_{t \in [0,s]} \E^{\vp} \left[ e^{- I(t)} \right]
 \cdot e^{- \rho t }  \cdot  H\left(\vx( t, \vp ) \right)   \\
+  \int_{0}^{t} e^{- \rho u}     \sum_{i \in E}  m_i(t,\vp) \cdot \lambda_i \left(\int_{\R^d}
K(y) \nu_i(dy) +  S_i w(s-u, \vx(u, \vp)) \right) du,
\end{multline*}
for a bounded function $w: [0,T] \times D \mapsto \R $.

Clearly $\{V_m\}_{m \ge 0}$ is an increasing sequence. Using the inequality $\E\left[
\sum_{j=1}^{N_T} K^+ (Y_j) \right] \le ( \max_{i \in E} \lambda_i )T \cdot (\max_{i \in E
}\nu_i K^+)$ and the truncation arguments in the proof of Proposition~\ref{prop:uniform-convergence},
one can show that the sequence converges to $\hat{V}$ uniformly with the error bound
\begin{align*}
0 \le  V - V_m \le \left( ( \max_{i \in E} \lambda_i ) T \cdot (\max_{i \in E }\nu_i K^+) + 2
\| H \| \right) \left(   \frac{  \bar{\lambda} \,  T}{  m-1  }  \right)^{1/2}  \left(
\frac{ \bar{\lambda}  }{ 2 \rho + \bar{\lambda} } \right)^{m/2}.
\end{align*}
Arguments in Sections~\ref{sec:sequential} and \ref{sec:solution} can then be replicated to
conclude that
\begin{align*}
 \E^{\vp} \left[
 \sum_{j=1 }^{ N_{ \hat{U}_{\varepsilon}(s,\vp) } } e^{- \rho \sigma_j } K( Y_j)  +
 e^{-\rho\,  \hat{U}_{\varepsilon}(s,\vp) }   H\left(\vP( \hat{U}_{\varepsilon}(s,\vp)  \right) \right] \ge \hat{V} (s, \vp) - \varepsilon, 
\end{align*}
for the stopping time
$ \hat{U}_{\varepsilon } (s,\vp) \triangleq \inf \left\{ t \in [0,s] \, :\,  \hat{V} (s-t,
\vP_t)  - \varepsilon \le H(\vP_t)  \right\} $. Hence, the admissible strategy $ (
\hat{U}_{\varepsilon } (s,\vp), d(\hat{U}_{\varepsilon } (s,\vp)) )$ is an optimal strategy for
the problem in \eqref{def:U-alternative}, as expected.

Furthermore,
other results of
Section~\ref{sec:solution} can be adjusted for this new objective function. Below, we summarize
these results in a remark.
\begin{rem}
Let $\nu_j K \triangleq \int_{\R^d} K(y) \nu_j(dy)$, for $j \in E$.
\begin{enumerate}
\item For a given index $i \in E$, Define $\mathcal{A}^*(i) \triangleq \{ k \in \mathcal{A} : \, \mu_{k,i} = \max_{j \in \mathcal{A} } \mu_{j,i} \}$ as in Lemma \ref{lemm:continue-at-good-states}. If $ - \rho \mu_{k,i } + \lambda_i \cdot \nu_i K+ \sum_{j
\ne i } (\mu_{k,j} - \mu_{k,i}) q_{i,j} > 0$ holds for all $k \in \mathcal{A}^*(i)$, then there exists
some $\hat{\pi}_i^c< 1$ (for all $T>0$) such that it is optimal to continue on the region $\{
(0,T] \times D ; \, \pi_i \ge \hat{\pi}_i^c \}$.

\item Assume $\nu_j K \le 0$ for all $j \in E$, and $ \bar{\mu} \triangleq  \max_{k,i}
\mu_{k,i}
> 0$, and let $I^*$ be as in \eqref{optimal-notation}. For $i \in I^*$, if $\nu_i K < 0$ or $\rho
> 0$ there exists a number $\hat{\pi}_i^s< 1$ (free of $T$) such that it is optimal to stop at
the points $\vp$ for which $\pi_i \ge \hat{\pi}_i^s$. That is: $\Gamma_{T,i} \supseteq \{ [0,T]
\times D ; \, \pi_i \ge \hat{\pi}_i^c \}$ for all $T \ge 0$.

\item In the case where $\nu_j K \le 0$ for all $j \in E$, and $H(\cdot)\ge 0$, the stopping
region is monotone in $\rho $ and $\nu_j K $, for $j \in E$. Namely, if we increase one of these factors in absolute terms (keeping everything else fixed), the stopping region
expands, and the DM is forced to make a decision sooner.

\item For a given $\varepsilon > 0 $, let $m \in \N$ such that $\| \hat{V} (T, \cdot) - \hat{V} (T, \cdot)  \| \le \varepsilon /2 $. Then the stopping time $\hat{U}^{(m)}_{\varepsilon /2 } (s,\vp) \triangleq \inf \left\{ t \in [0,T] \, :\,  \hat{V}_m (T-t,
\vP_t)  - \varepsilon \le H(\vP_t)  \right\} $ gives an $\varepsilon$-optimal strategy.

\item If ``$\rho > 0$'' or ``$K(\cdot) \le 0$ with $\max_{i \in E} \nu_i K(\cdot) < 0$'', then $\hat{V} (T, \cdot) \nearrow \hat{V} (\infty, \cdot)$ uniformly 
as in \eqref{eq:convergence-of-finite-to-infinite} if we redefine
\begin{align*}
Err(T) \triangleq \left\{
\begin{aligned}
&e^{- \rho T} \big( \max_{i \in E} \lambda_i \cdot \max_{i \in E} \nu_i K^+ + 2 \cdot \|  H( \cdot  ) \| \big) &&, \quad \text{\normalfont if $\rho > 0$} \\
 &
\frac{2 \cdot \|  H( \cdot  ) \|}{T} \,
\frac{\big( \min_{k,i}\mu_{k,i} - \max_{k,i} \mu_{k,i} \big) }
{ \min_{i \in E} \lambda_i \cdot \max_{i \in E} \nu_i K}
 &&, \quad \text{\normalfont if $\rho = 0$, $K(\cdot) \le 0$ and $ \max_{i \in E } \nu_i K <  0$.}
\end{aligned}
\right\}
\end{align*}

\end{enumerate}
\end{rem}


\section{Examples}
\label{sec:examples}

Below we provide numerical examples illustrating the use of our sequential approximation
approach developed in Section~\ref{sec:sequential}.
In each example, we approximate the value function by repeatedly (finitely many times) applying the operator $J$ in \eqref{def:J-0} starting with the initial function $H(\cdot)$. We set the number of iterations $m \in \N$ such that the error $\| V_m(\cdot) -V(\cdot)  \|$ is negligible
(see \eqref{eq:uniform-bound}).

\subsection{Insurance launch.}\label{sec:insurance}
Our first example illustrates profit maximization with information cost, which is the first example in Section \ref{sec:catalogue}. 
Here, $M_t$ represents the state of the economy with three
major states $E = \{ 1,2,3\} \equiv \{Boom, Growth, Recession\}$, and with the generator
$$ Q = \begin{pmatrix} -4 & 3 & 1 \\ 2 & -4 & 2 \\ 0 & 3 & -3 \end{pmatrix}.$$
Let $ \vec{\lambda} =[ \lambda_1, \lambda_2 , \lambda_3 ] = [1, 2, 5]$ and
$\vec{\nu} =[ \nu_1, \nu_2 , \nu_3 ] = [ Gamma(3, 2) , Gamma(4, 2) , Gamma(5, 2)] $. 
Conditional on the state of $M$ being $i \in E$,
the frequency of claims is $\lambda_i$ and their common distribution is $\nu_i$.
Here, we consider the objective function in \eqref{def:R1}
with $\vec{\mu}  \equiv [\mu_{B}, \mu_{G}, \mu_{R}] = [6,1,-3]$, $\rho = 0.1$ and $c=
- 0.3$. As before, $d=1$ represents the decision to launch the new policy; $d=0$ represents the
decision to abandon, and does not involve any cashflows. The horizon is taken to be $T=0.8$ (whose unit is to be consistent with that of $\lambda_i$'s; e.g., if $\lambda_i$ is in ``customers per month'', T is in months).

\begin{figure}
\includegraphics[width=0.95\textwidth]{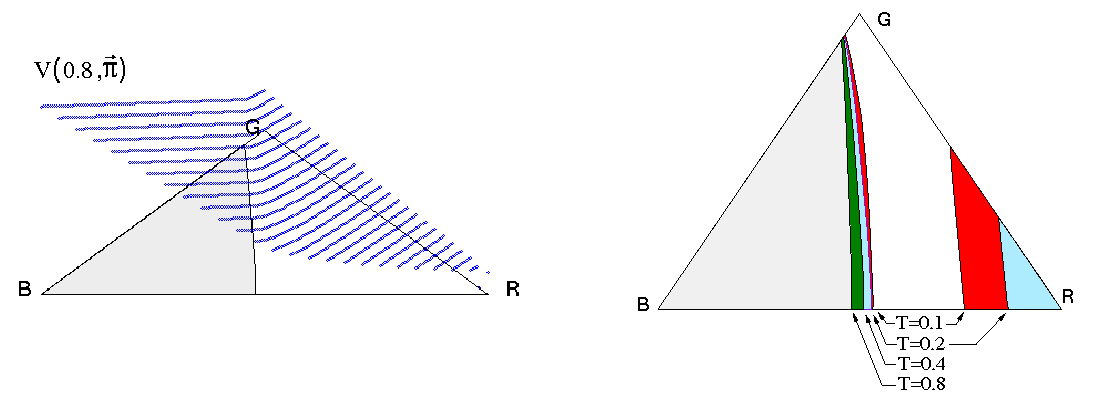}
\caption{\emph{\small{Value function and stopping regions of the insurance launch example of Section
\ref{sec:insurance}. The left panel displays the value function $V(T,\vp)$, for $\vp \in D$ and
$T=0.8$. At $T=0.8 $, if the conditional likelihood process $\vP$ is in the shaded region, the DM stops and selects action $d=1$. Otherwise, she continues observing until the first time $V(T - t, \vP_t) = H(\vP_t)$. 
The right panel shows the dependence of the stopping regions on horizon $T$.}} \label{fig:insurance}}
\end{figure}

For this example, we discretized $D=\{ \vp \in \R^3_+: \pi_B  + \pi_G + \pi_R = 1 \}$ using $100$ grid points in each dimension and computed $V_{m}$ such that $\| V_m - V_{m-1}\| \le 10^{-4}$. The triangular regions in Figure \ref{fig:insurance} show the region $D$. The corners $\{ B,G,R \}$ corresponds to points where the states $\{ Boom, Growth, Recession \}$ have posterior probabilities equal to 1 respectively. The left panel of Figure \ref{fig:insurance} shows the value function $V(0.8,\vp)$ and the shaded region is $\{ \vp \in D: \, V( 0.8,\vp  ) = H(\vp)  \}$. 
Recall that it is optimal to stop as soon as $V(T-t, \vP_t) = H(\vP_t)$ and the corresponding
stopping region is time-dependent. The right panel of Figure \ref{fig:insurance} illustrates
this point by varying the problem horizon $T$. As expected from Remark
\ref{rem:stoppingHorizon}, when $T$ decreases, stopping regions expand. In particular, we see
that with very little time left ($T=0.1$ and $T=0.2$), it is optimal to stop whenever $\pi_B$
(where action $d=1$ is chosen) \emph{or} $\pi_R$ is high (where quitting $d=0$ is optimal). For longer horizons, the DM can afford to wait for
favorable circumstances and release the product then. That is, stopping and selecting $d=0$ is never optimal when time-to-maturity is not small. 
Also note that the terminal reward associated with $d=1$ is higher than that of $d=0$ around the corner $G$. Moreover, with the notation in Lemma~\ref{lemm:continue-at-good-states} we have $r_G = c_G - \rho \, \mu_{G} + (\mu_{B} - \mu_{G}) \,  q_{G,B} + (\mu_{R} - \mu_{G}) \, q_{G,R} = 1.6 > 0$. Then by Lemma~\ref{lemm:continue-at-good-states}, it is never optimal to stop around the corner $G$ (unless $T=0$) as shown the in right panel of Figure \ref{fig:insurance}. 

\subsection{Bayesian regime detection.}\label{sec:hypothesis}
Recall the hypothesis testing problem in \eqref{def:R3}.
Let $V(\infty, \vp)$ denote the value function of this \emph{minimization} problem on
infinite-horizon. With the notation in \eqref{def:stop-cont-regions}, it is shown
in \cite{dps} that it is optimal to stop the first time the conditional probability process
$\vP$ enters the region $ \cup_{k \in E} \Gamma_{\infty, k}$ where
$\Gamma_{\infty, k} \triangleq \{ \vp \in D: \, V(\infty, \vp) = H_k(\vp) \} $
in terms of the functions $H_k (\vp) = \sum_{i \in E} \mu_{k,i} \pi_i$. 
Each $ \Gamma_{\infty, k} $ is a convex region with non-empty interior around $k$'th corner of the simplex $D$.
Namely, an observer stops whenever the conditional likelihood of
one of the hypotheses is sufficiently high. This structure also extends to the finite-horizon
problem. Since $V(\infty, \vp) \le V(T , \vp) $, we have $\Gamma_{\infty,k} \subseteq
\Gamma_{T,k} $, for $k \in E$ and $T < \infty$. In plain words, regardless of the remaining
time to maturity, the observer selects immediately one of the hypotheses when the conditional likelihoods process $\vP$ is around the corners of $D$ (i.e., if there is sufficient
posterior statistical evidence).


In Figure \ref{fig:hypothesis}, we illustrate the time-dependence of the solution structure
using a simple example with two hypotheses $H_1: \Lambda= \lambda_1$ and $H_2: \Lambda= \lambda_2$  on the arrival rate only. The problem in infinite horizon where there are two hypotheses on the arrival rate was solved for the first time by
\cite{PeskirShiryaev} (with $\lambda_2 > \lambda_1$ without loss of generality). The authors showed 
that the immediate stopping is optimal if and only if $ \mu_{2,1} \mu_{1,2} (\lambda_2 - \lambda_1 ) \le \mu_{2,1} + \mu_{1,2}$ (see \cite[Theorem 2.1]{PeskirShiryaev}). Hence the inequality $ \mu_{2,1} \mu_{1,2} (\lambda_2 - \lambda_1 ) > \mu_{2,1} + \mu_{1,2}$ has to be satisfied in any finite-horizon problem with non-trivial solution. 

In Figure \ref{fig:hypothesis}, under $H_1$ the arrival rate is
$\lambda_1 = 1$ while under $H_2$ it is $ \lambda_2 = 5$.
For the Bayes risk given in \eqref{def:R3}, we select $\mu_{1,2} = \mu_{2,1} = 2$ for the penalty costs for selecting the wrong
hypothesis.
This numerical example corresponds to the one considered in \cite[Figures 2-3]{PeskirShiryaev}.
The left panel of Figure \ref{fig:hypothesis} shows the value functions $V(T,\cdot)$ 
with horizons $T=0.1, T=0.2, T=0.4$ and $T=2$ respectively, and the terminal reward $H(\vp)=  \min \{ \mu_{1,2} \pi_2 \, ; \,  \mu_{2,1} (1-\pi_2) \}$ on the state space of $\pi_2 \in [0,1]$. We see that as more time is available to make the decision, the value function decreases, as expected.
The right panel of Figure \ref{fig:hypothesis} shows that the continuation region widens as time to maturity increases. We also observe that the boundary curves approaches the solution structure of problem with infinite
horizon. \cite{PeskirShiryaev} obtain a continuation region of $[0.22, 0.70]$, very close to
ours of $[0.230, 0.705]$ for $T>1$.

\begin{figure}[ht]
\begin{tabular*}{\textwidth}{lr}
\begin{minipage}{0.5\textwidth}
\center{\includegraphics[width=0.6\textwidth]{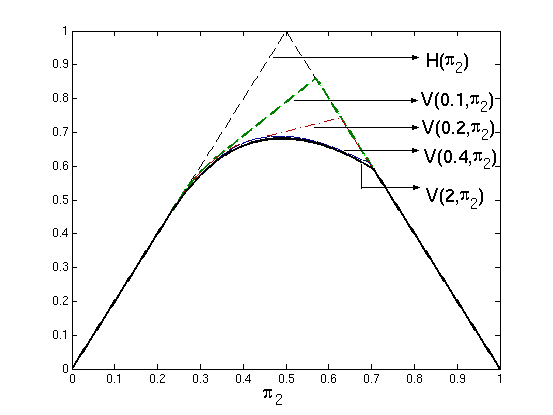}}

\end{minipage} &
\begin{minipage}{0.55\textwidth}
\center{\includegraphics[width=0.6\textwidth]{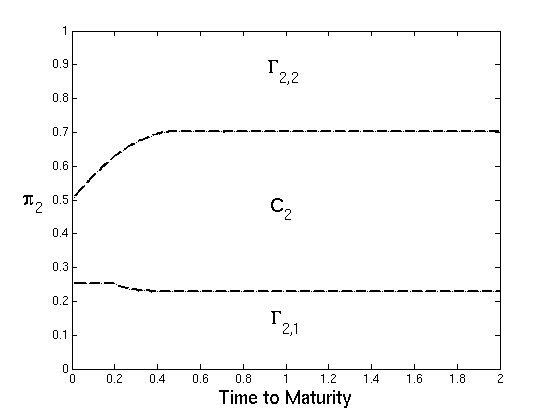}}

\end{minipage}
\end{tabular*}\caption[Bayesian Regime Detection]{\small{\emph{Bayesian regime detection example
of Section \ref{sec:hypothesis}. The left panel shows the value functions $V(T,\vp)$ for various time
horizons $T$. The right panel shows the stopping regions  $\Gamma_{T,k}$ (namely $\Gamma_{T,0}$
below the lower curve and $\Gamma_{T,1}$ above the higher curve)
for $T=2$.
}
\label{fig:hypothesis}}}
\end{figure}


Let us define the lower boundary curve
$T \mapsto b_1(T) \triangleq \sup \{ \pi_2 \in [0,1] :\, V(T, \vp ) = 2 \pi_2 \} $. Clearly $b_1(0)= 0.5$. In the right panel, we observe that
%
the lower boundary curve $b_1(\cdot)$ has a
discontinuity at $T=0$ (jumping from $\pi_2 = 0.5$ to approximately $ \pi_2 = 0.25$) and then
remaining constant until about $T=0.2$.
Note that the point $\vp = (\pi_1 , \pi_2) = (0.5, 0.5)$ is the global maximum of the
terminal cost function $H(\vp)$. Starting at the point $( 0.5 + \varepsilon , 0.5 - \varepsilon)$,
for $\varepsilon \ge 0$ and small, as long as there is no jump, the conditional likelihood
process $\vP$ drifts  (quickly) toward the point $\vp= (\pi_1 , \pi_2) =(1,0)$ and away from this maximum.
For very small values of $T$, the probability of observing a jump is low and thus it is optimal
to continue. Therefore, the lower curve in Figure~\ref{fig:hypothesis} is discontinuous around $T=0$.  
The rate of drift of the process $\vP$ to the point $(1,0)$ decreases as $\pi_2$ decreases and approaches the point $(1,0)$ (see \eqref{eq:vx-dyn}). As
a result, at points $\vp$ where $\pi_2$ is small, the effect of waiting cost becomes
dominant and it is optimal to stop even if $T$ is small. 

The following remark summarizes our discussion on this problem and states that the behavior of the lower boundary curve around $T=0$ holds for any set of parameters $\lambda_2 > \lambda_1$, $\mu_{1,2}$, $\mu_{2,1}$. Its proof can be found in the Appendix.

\begin{rem}
\label{rem:two-hypothesis}
Consider the hypothesis-testing problem in \eqref{def:R3} with two simple hypotheses on the arrival rate: $H_1: \Lambda= \lambda_1$ and $H_2: \Lambda= \lambda_2$ (with $\lambda_2 > \lambda_1$).
The continuation region $\mathcal{C}_T$ is non-empty (for $T> 0$) if and only if $ \mu_{2,1} \mu_{1,2} (\lambda_2 - \lambda_1 ) > \mu_{2,1} + \mu_{1,2}$.
The boundary curve $T \mapsto b_1(T) \triangleq \sup \{ \pi_2 \in [0,1] :\, V(T, \vp ) = \mu_{1,2}\,  \pi_2 \}$ is discontinuous at $ T=0 $, and there is an interval around $ T=0 $  at which $ b_1(\cdot)$ is constant.
\end{rem}

\subsection{Optimal replacement of a system.}\label{sec:jensen}

Here we consider the reliability problem in \eqref{def:objective-function-in-reliability}. In this problem, the unobservable  Markov process $M$ represents the current productivity of a given machine, and the $n$'th state (defective state) of $M$ is absorbing. The objective is to find the best time to replace the equipment in order to maximize the net lifetime earnings. The problem is studied by \cite{JensenHsu} under certain assumptions on $(q_{i,j})_{i,j \in E}$, $\vec{\lambda}$, $\vec{\mu}$ and $\vec{c}$ such that the infinitesimal look-ahead (ILA) rule $\tau^{ILA} := \inf\{ t\ge 0 \colon \sum_i r_i \Pi_t^{(i)} < 0 \}$ is optimal where $r_i \triangleq c_i + \sum_{j \ne i} (\mu_{j}- \mu_{i}) q_{i,j}$ (cf. Lemma \ref{lemm:continue-at-good-states}). More precisely these assumptions are (i) $q_i \ne 0$ for $i = 1, \ldots , n-1$, with $q_n =0$ (ii) $r_1 \ge r_2 \ge \ldots \ge r_n = c_n $, with $c_n < 0$ (iii) $0 < \lambda_1 \le \ldots \le \lambda_n  $, (iv) $q_{in} > \lambda_n!
  - \lambda_i $ for $i = 1, \ldots , n -1 $.

It follows as a corollary to \cite[Theorem 3.1]{Jensen89} that $\tau^{ILA} \wedge T$ is an optimal stopping rule for the finite horizon problem under these assumptions. Therefore, the region $\{  \vp \in D: \, V(T,\vp) = H(\vp) \}$ does not depend on $T$. This occurs because 
the instantaneous revenue rates $r_i$'s completely summarize the relative worth of different machine states, and the sum $\sum_{i \in E} r_i \Pi_t^{(i)}$ 
is monotonically non-increasing over time $\P^{\vp}$-almost surely for all $\vp \in D$ (see \cite[Theorem 2]{JensenHsu}).
Thus, $T$ only plays a role insofar as allowing the DM to collect profits
before the machine deteriorates.

\begin{figure}
\includegraphics[width=1\textwidth]{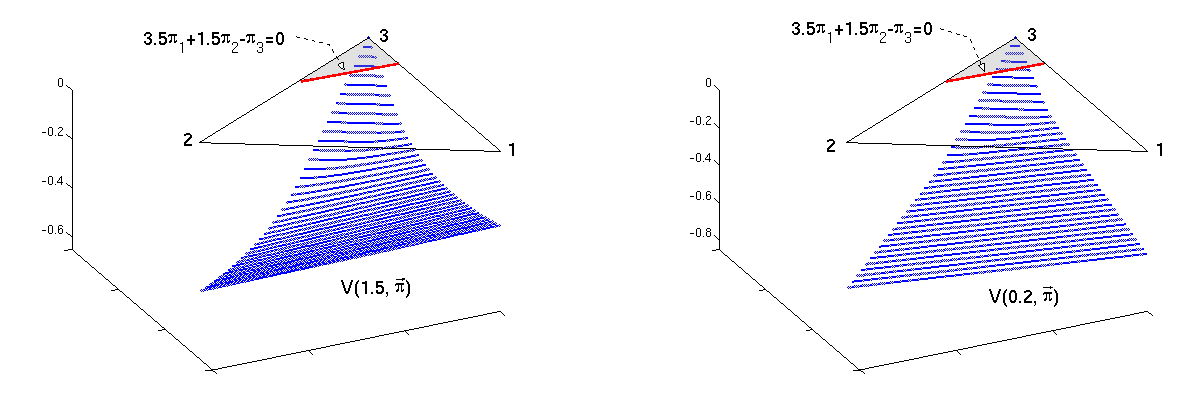}
\caption[Reliability Example]{\small{\emph{Value function $V(T,\vp)$ of the reliability example
of Section~\ref{sec:jensen}. The shaded regions represent the computed stopping regions $\{  \vp \in D: \, V(T,\vp) = H(\vp) \}$.
Left panel shows $T=1.5$, right panel shows $T=0.2$. The shaded regions are the same in both panels. Note however the different $z$-scales. The panels also show the line $3.5 \pi_1 + 1.5 \pi_2- \pi_3 = 0 $, which is the stopping boundary of the ILA rule in \eqref{eq:jensenTau}.} \label{fig:jensen}}}
\end{figure}

We illustrate this degeneracy in Figure \ref{fig:jensen}. In this example, we select the parameters to fit the framework of \cite{JensenHsu}. We have a machine that moves through three regimes $E = \{1, 2, 3\} \equiv \{Good, Average, Poor\}$ with transition matrix
$$ Q = \begin{pmatrix} -4 & 1.5 & 2.5 \\ 0 & -1.5 & 1.5 \\ 0 & 0 & 0 \end{pmatrix}.$$
At different states, the running profit from operating the machine is $\vec{c} = [1, 0,
-1]$, 
and shutting down the machine for maintenance involves a cost of
$\vec{\mu} = [-1, -1, 0]$. Thus, it is costly to shutdown a machine until it is in the $Poor$
state. In each state, the breakdowns occur according to independent Poisson processes with
intensities $\vec{\lambda} = [2, 3, 4]$. In this setting we have $\vec{r} = \{ r_1, r_2, r_3 \} = \{3.5, 1.5, -1\}$ so that
\begin{align}\label{eq:jensenTau}
\tau^{ILA} = \inf\{ t \ge 0 \colon 3.5 \Pi_t^{(1)} + 1.5 \Pi_t^{(2)} - \Pi_t^{(3)}  < 0 \}.\end{align}
The left and right panels of Figure \ref{fig:jensen} show the functions $V(T,\vp)$ and the regions $\{ \vp \in D:\, (T,\vp) \in \Gamma_T \}$ for $T= 1.5$ and $T= 0.2$ respectively.
We see that $V(0.2, \vp ) < V(1.5, \vp ) $ but the regions $\{  \vp \in D: \, V(T,\vp) = H(\vp) \}$ for $T=0.2$ and $T=1.5$ completely matches the region $\{ \vp \in D : \, 3.5 \pi_1 + 1.5 \pi_2- \pi_3 \le 0 \} $, 
at least modulo the $D$-discretization necessary for numerical implementation.

\begin{figure}
\includegraphics[width=1\textwidth]{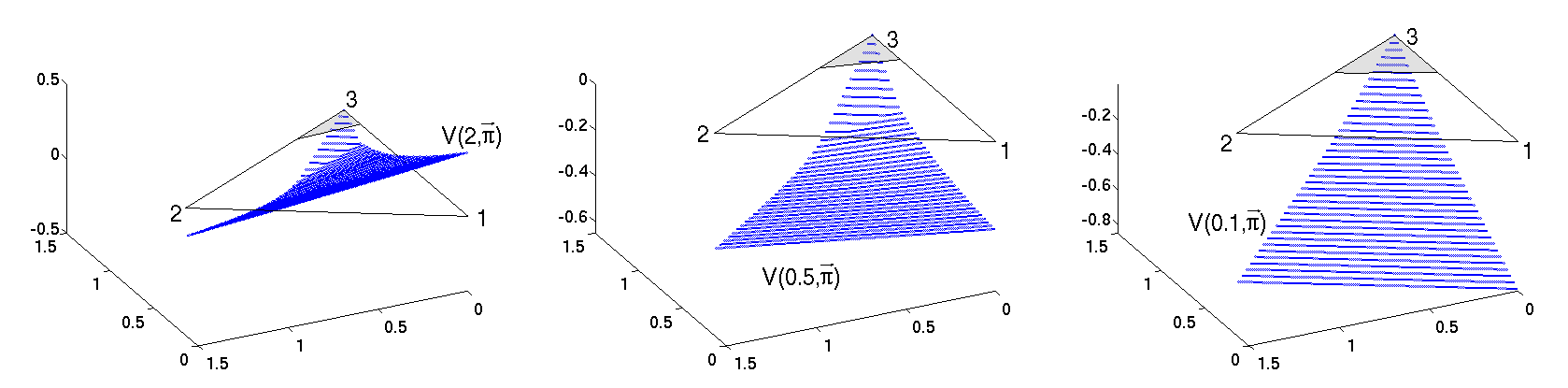}
\caption[Technology Adoption
Example]{\small{\emph{The second example for the reliability problem of Section \ref{sec:jensen} with the new parameters in \eqref{matrix:second-problem}. In the left panel $T=2$, in the middle $T=0.5$, and in the right panel $T=0.1$. In each picture, the function $V( T, \vp) $ is plotted on $D$. The shaded regions are the sets $\{  \vp \in D: \, V(T,\vp) = H(\vp) \}$.} \label{fig:techAdopt2}}}
\end{figure}

This degenerate structure would disappear if
one removes some of the assumptions in \cite{JensenHsu}, for example the special form of generator $Q$ and/or the arrival rates $\vec{\lambda}$ above. We give an example in Figure \ref{fig:techAdopt2} where 
\begin{align}
\label{matrix:second-problem}
Q = \begin{pmatrix} -1 & 0.5 & 0.5 \\ 0 & -0.5 & 0.5 \\ 0 & 0 & 0 \end{pmatrix}
\qquad \text{and} \qquad \vec{ \lambda } = [\lambda_1 , \lambda_2, \lambda_3]=  [1 ,4 , 7].
\end{align}
We keep other parameters the same as in the previous example. In this example, the instantaneous net gain $\sum_{i \in E} r_i \Pi_t^{(i)} = 1.5 \Pi_t^{(1)} + 0.5 \Pi_t^{(2)} - \Pi_t^{(3)}$ is not monotonically non-increasing $\P^{\vp}$-almost surely for all $\vp \in D$ anymore. For example, using \eqref{eq:vx-dyn} it can be shown that $d( 1.5 x_1(t,\vp) + 0.5 x_2(t,\vp) - x_3(t,\vp) ) / dt |_{t =0 } > 0 $ at the point $\vp = (\pi_1 , \pi_2, \pi_3) = (0.45, 0.45, 0.1)$. Figure \ref{fig:techAdopt2} shows that the structure of the stopping region is indeed time dependent. The stopping region expands as time to maturity decreases. Moreover, in this problem the transition rates of $M$ are lower. Therefore, the DM can obtain positive net gain when $M$ starts from the state $\{1\}$ and there is enough time to operate the system. Indeed, the first panel in Figure \ref{fig:techAdopt2} shows that for $T=2$ the value function is positive around the corner $\{ 1 \}$.

\subsection{Technology adoption example.}\label{sec:techAdopt}
To illustrate an example for the discrete cost structure of Section~\ref{sec:otherCosts}, we consider an IT company, which is planning to add a new technological
feature to its products. The benefit of the technology is unknown, but will improve over time
as customer awareness grows and production is streamlined. The company wishes to adopt the
technology at the optimal time that best resolves the tension between early adoption (with high
production costs) and late adoption (with opportunity costs due to late market entry). A
similar setting has been studied recently by \cite{UluSmith} and goes all the way back to
\cite{McCardle85}.

Suppose that after $T$ years the technology becomes obsolete and let $M = \{ M_t \}_{ t \ge 0}$ represent the profitability/value of the technology with state space $E = \{1, 2, 3\} \equiv \{ Low,
Med, High \}$. The generator of $M$ is $$ Q = \begin{pmatrix} -2 & 2 & 0 \\ 0 & -2 & 2 \\
0 & 0  & 0 \end{pmatrix}.$$ Thus, $M$ sequentially moves through the phases $Low \rightarrow
Med \rightarrow High$. The firm may incorporate the feature at the minimal level (action
$d=1$), at the maximum level ($d=2$), or not at all $(d=0)$. The profit functions are given by
$$ \mu_{k,i} = \begin{bmatrix} -1 & 3 & 4\\ -4 & 2 & 10 \end{bmatrix}, \qquad \qquad k \in \{ 1,2\}, \quad i \in
E,$$ with zero profit when  $d=0$.

The observation process $X$ corresponds to competitor contract sales and is represented by a
compound Poisson process with mark space $Y_k \in B = \{1, 2\} \equiv \{Large, Small \}$. The
$M$-modulated intensity of $X$ is $\vec{\lambda} = [ \lambda_1, \lambda_2 , \lambda_3] = [3, 5, 3]$ and the mark distributions on $B$
are $[0.2, 0.8], [0.5, 0.5], [0.8, 0.2]$ respectively. Contracts signed by competitors are
opportunity costs and the objective function is of the type \eqref{def:U-alternative} (with
zero discounting $\rho=0$):
$$ V(T,\vp) = \sup_{\tau \le T, d \in \mathcal{F}^X_\tau} \E^{\vp} \left[  \sum_{j=1}^{N_\tau} K(Y_j)+ \sum_{ k=1 }^2 1_{\{  d =k \}}
\Bigl( \sum_{i \in E} \mu_{k,i} \cdot 1_{\{ M_{\tau} =i \}} \Bigr) \right],
$$ where $T=1, K(1) = -3, K(2) = -1$.

\begin{figure}
\includegraphics[width=0.95\textwidth]{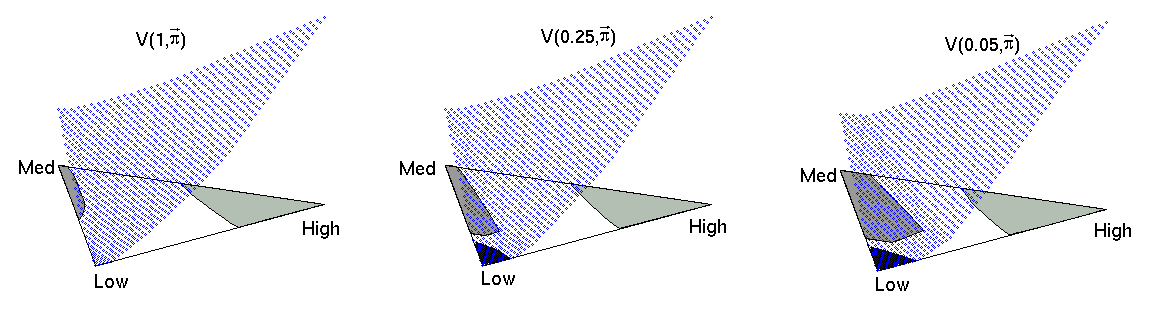}
\caption[Technology Adoption
Example]{\small{\emph{Value function $V(T,\vp)$ of the technology adoption example \ref{sec:techAdopt}
plotted together with the stopping regions (shaded: $d=2$ lighter color, $d=1$ darker, $d=0$ black).
Left panel: $T=1$, middle panel: $T=0.25$, right panel: $T=0.05$.} \label{fig:techAdopt}}}
\end{figure}

The triangular regions in Figure \ref{fig:techAdopt} are the state space $D = \{ \vp \in \R_+^3 :\, \pi_{Low} + \pi_{Med} + \pi_{High} =1\} $. In the panels, we show how the stopping regions expand as the time to maturity approaches (from left to right) as indicated in Remark
\ref{rem:stoppingHorizon}. When $T=1$, (left panel) we see that if the DM stops, she either selects $d=1$, or $d=2$ if there is sufficient evidence that $M$ is at $Med$ or $High$ respectively. For $T=1$, the decision $d=0$ is never considered since the DM can wait for $M$ to move to better states. Note that, if $T$ is small (middle and right panels) and if $M$ seems to be at $Low$ state, the DM does not have enough time to wait for $M$ to jump to a new state. By stopping immediately, she at least gets rid of the opportunity costs.

Around the $Med$ corner there is high competitor activity ($\lambda_2 =5$), and this 
increases in the opportunity costs (given by $K(\cdot)$). As a result the DM always stops, she does not wait for $M$ to move to $High$ state. Since the expected reward of minimal commitment is higher than that of maximum commitment around this corner, she selects $d=1$. The DM selects $d=2$ only if there is sufficient statistical evidence that the technology has reached its $High$ benefit.

\subsection{A targeting problem.}
\label{sec:lobby}

As a final illustration we present a \emph{targeting} example, where the objective is to
maximize the probability of $M$ belonging to some favorable set $B \subseteq E$. 
\begin{figure}
\includegraphics[width=0.95\textwidth]{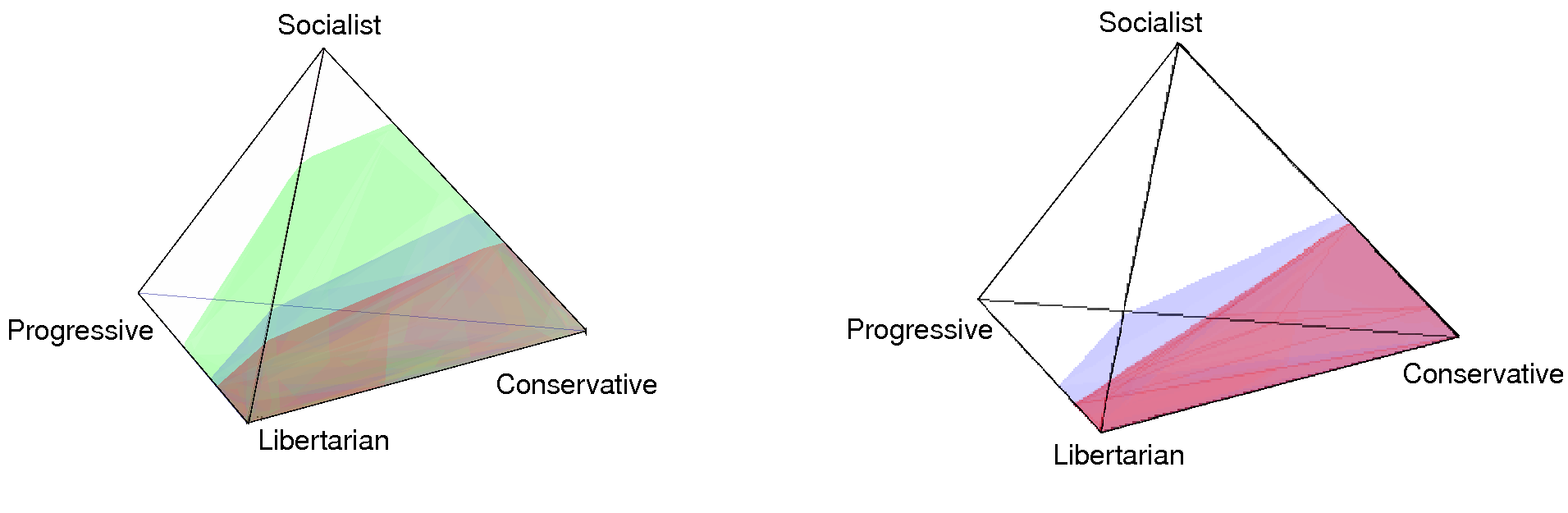}
\caption{\small{\emph{Stopping regions $\{\vp \in D :\, V(T,\vp)= H(\vp) \} \subset \Gamma_T$ of the targeting example of Section \ref{sec:lobby} for $T=2$. On the left
panel we illustrate the effect of the waiting cost $c$, with the shaded polyhedra representing
stopping regions for $c=-0.1, c=-0.2, c=-0.4$ respectively.
On the right panel we take $c= -0.2$, and we display the effect of changing the arrival rate from $\lambda_{L}=4$ (blue/lighter
stopping region) to $\lambda_{L} = 10$ (red/darker stopping region). } \label{fig:lobbyLam}}}
\end{figure}

An industrial conglomerate is seeking a business-favorable government legislation and employs a
lobbyist for that purpose. The lobbyist maintains government contacts and will try to time her
action to maximize the probability of the law passing. Suppose the passage of legislation
depends on the current political climate $M_t$ in the country that can be one of the following
four states: $E = \{1,2,3,4\} \equiv \{ Libertarian,  Conservative, Progressive, Socialist \}$.
For simplicity we assume that the law will pass if the climate is in $B = \{Libertarian,
Progressive\}$ and fail otherwise. Suppose that the generator of $M$ is
$$ Q = \begin{pmatrix} -1 & 0.5 & 0.5 & 0 \\ 0.5 & -1.5 & 0.5 & 0.5 \\ 1 & 0.5 & -2 & 0.5 \\
0 & 1 & 0.5 & -1.5 \end{pmatrix}.$$ We postulate that the objective function is $\E^{\vp}[ c
\tau] + \P^{\vp} (M_{\tau} \in B)$, where the constant $c \le 0$ denotes the running cost of
maintaining the lobby. Information is obtained via a simple Poisson process counting the
passing of other business-friendly legislation, with $M$-modulated intensities $\vec{\lambda}
\equiv [\lambda_{L}, \lambda_{C}, \lambda_{P},\lambda_{S}] = [4, 3, 2, 1]$. The time horizon is
$T=2$ years.

Figure \ref{fig:lobbyLam} shows the stopping regions of this example inside the tetrahedron $D$.
The left panel shows the effect of changing the waiting cost $c$; as $c$ increases in absolute value, the DM is
more ``impatient'' and will stop sooner, compare with Remark \ref{rem:VincreasesInRho}. The
right panel of Figure \ref{fig:lobbyLam} shows the effect of increasing $\lambda_{L}$ to
$\lambda_{L} = 10$. As intuition suggests, this shrinks the continuation region because the
data is now more informative. We see that the continuation region $\mathcal{C}_T $ expands
especially around the 'Libertarian' corner, as the DM can now be fairly confident in detecting
that regime (as it has a much higher arrival intensity). 

\appendix

\renewcommand{\thesection}{A1}
\refstepcounter{section}
\makeatletter
\renewcommand{\theequation}{\thesection.\@arabic\c@equation}
\makeatother

\section{Sample Paths of $\vP$}
\label{sec:sample-paths-of-pi}

In this appendix, we prove Lemma~\eqref{lemm:vP-explicit}, and we derive the characterization of the sample paths given in (\ref{eq:rel-pi-x}-\ref{eq:x-i}).

 \begin{proof}[\textbf{Proof of Lemma~\ref{lemm:vP-explicit}.}]
Let $\Xi$ be a set of the form
\begin{align*}
{\Xi} = \{  N_{t_1} = m_1, \ldots, N_{t_k} = m_k ; (Y_1,\ldots, Y_{m_k}) \in B  \}
\end{align*}
where $0 = t_0 \le t_1 \le \ldots \le t_k = t $ with $0 \le m_1 \le \ldots \le m_k$ for $k \in \N$, and $B$ is a Borel set in $\mathcal{B}(\R^{m_k})$. Since $t_j$ and $m_j$'s are arbitrary, to prove (\ref{vP-explicit}) it is then sufficient to establish  
\begin{align*}
\E^{\vp} \left[ 1_{\Xi} \cdot \P^{\vp} \{ M_t =i | \Fc^X_t \} \right] =
\E^{\vp} \left[ 1_{\Xi} \cdot \frac{L_i^{\vp} ( t, N_t : (\sigma_k, Y_k ), i \le N_t) }{ L^{\vp} ( t, N_t : (\sigma_k, Y_k ), i \le N_t)} \right].
\end{align*}
Conditioning on the path of $M$, the left-hand side (LHS) above equals
\begin{multline*}
LHS  = \E^{\vp} \left[ 1_{\{M_t = i\}} \P^{\vp} \left\{ N_{t_1} = m_1, \ldots, N_{t_k} = m_k ; (Y_1, \ldots, Y_{m_k}) \in B \, \Big| M_s ; \, s \le t \right\}  \right] \\
 = \E^{\vp} \left[ 1_{\{M_t = i\}}  \int_{B \times \Upsilon (t_1, \dotsc, t_k)} \P^{\vp} \left\{ \sigma_1 \in ds_1, \ldots, \sigma_{m_k} \in s_{m_k} ; Y_1 \in dy_1, \ldots, dY_{m_k} \in dy_{m_k} \, \Big| M_s ; \, s \le t \right\}  \right]
\end{multline*}
where
\begin{align*}
\Upsilon (t_1, \dotsc, t_k) = \left\{ s_1, \ldots, s_{m_k} \in \R_+^{m_k}: \; s_1\le \ldots \le
s_{m_k} \le t \; \; \text{and}\; \; s_{m_j} \le t_j < s_{m_j+1} \; \text{for}\;  j=1, \ldots k
\right\}.
\end{align*}
Then, by Fubini's theorem we have
\begin{align*}
LHS & =\E^{\vp} \left[ 1_{\{M_t = i\}}\int_{B \times \Upsilon (t_1, \dotsc, t_k)}  e^{-I(t)}  \prod_{l=1}^{m_k} \sum_{j \in E }  1_{\{M_{s_l} = i\}} \lambda_j f_j(y_l) \,  ds_l \, \nu(dy_l)  \right] \\
& =\int_{B \times \Upsilon (t_1, \dotsc, t_k)}  L_i^{\vp} ( t, m_k : (s_j, y_j ), j \le m_k) \prod_{l=1}^{m_k}ds_l \cdot \nu(dy_l)\\
& =\int_{B \times \Upsilon (t_1, \dotsc, t_k)} \frac{L_i^{\vp} ( t, m_k : (s_j, y_j ), j \le m_k) }{ L^{\vp} ( t, m_k : (\sigma_j, Y_j ), j \le m_k) } \cdot  L^{\vp} ( t, m_k : (s_j, y_j ), j \le m_k) \prod_{l=1}^{m_k}ds_l \cdot \nu(dy_l)
\end{align*}
Another application of Fubini's theorem gives LHS = 
\begin{align*}
&\E^{\vp} \left[ \sum_{i \in E }1_{\{M_t = i\}}
\int_{B \times \Upsilon (t_1, \dotsc, t_k)}
\frac{L_i^{\vp} ( t, m_k : (s_j, y_j ), j \le m_k) }{ L^{\vp} ( t, m_k : (\sigma_j, y_j ), j
\le m_k) } \cdot     e^{-I(t)}  \prod_{l=1}^{m_k} \sum_{j \in E }  1_{\{M_{s_l} = i\}}
\lambda_j f_j(y_l)
\cdot \prod_{l=1}^{m_k}ds_l \cdot \nu(dy_l) \right]\\
&=\E^{\vp} \left[ \sum_{i \in E }1_{\{M_t = i\}}  \E^{\vp} \left[ 1_{\{  N_{t_1} = m_1, \ldots,
N_{t_k} = m_k ; (Y_1,\ldots, Y_{m_k}) \in B  \}} \cdot \frac{L_i^{\vp} ( t, N_t : (\sigma_j,
Y_j ), j \le N_t) }{ L^{\vp} ( t, N_t : (\sigma_j, Y_j ), j \le N_t) } \;
\Bigg| M_s ; \, s \le t \right] \right] \\
&=\E^{\vp} \left[   \E^{\vp} \left[ 1_{\{  N_{t_1} = m_1, \ldots, N_{t_k} = m_k ; (Y_1,\ldots,
Y_{m_k}) \in B  \}} \cdot \frac{L_i^{\vp} ( t, N_t : (\sigma_j, Y_j ), j \le N_t) }{ L^{\vp} (
t, N_t : (\sigma_j, Y_j ), j \le N_t) }
\; \Bigg| M_s ; \, s \le t \right] \right] .
\\
&=\E^{\vp} \left[   1_{\Xi} \cdot \frac{L_i^{\vp} ( t, N_t : (\sigma_j, Y_j ), j \le N_t) }{
L^{\vp} ( t, N_t : (\sigma_j, Y_j ), j \le N_t) }
 \right],
\end{align*}
and this concludes the proof.
\end{proof}

\begin{proof}[\textbf{Proof of Remark~\ref{rem:sample-paths-of-pi}.}]
In order to establish (\ref{eq:rel-pi-x}-\ref{eq:x-i}), let $\E_j[\cdot]$ denote the expectation operator $\E^{\vp}[\cdot \, | M_0 =j] $, and let $t_m \le t \le t+ u < t_{m+1}$. Here $t_m $ and $t_{m+1}$ can be considered as the sample realization $\sigma_m (\omega) $ and $\sigma_{m+1} (\omega)$ of the m'th and m+1'st arrival times respectively. Using the definition of $L_i^{\vp}$ in \eqref{def:L} we have
$L_i^{\vp} ( t+u  , m : (t_k, y_k ), k \le m)
= \sum_{j \in E} \pi_j  \cdot
\E_j \left[
1_{\{ M_{t+u} =i\}} \cdot e^{ - I(t+u) } \cdot  \prod_{k=1}^m \ell(t_k, y_k)
\right]$
\begin{align}
\label{eq:derivation-of-vx}
\begin{aligned}
=& \sum_{j \in E} \pi_j  \cdot
\E_j \left[
\E_j \Biggl[
1_{\{ M_{t+u} =i\}} \cdot e^{ - I(t+u) } \cdot  \prod_{k=1}^m \ell(t_k, y_k) \,
\Bigg| M_s : s\le t \Biggr]   \right] \\
=& \sum_{j \in E} \pi_j  \cdot
\E_j \left[ e^{ - I(t) } \left( \prod_{k=1}^m \ell(t_k, y_k) \right)
\E_j \left[
1_{\{ M_{t+u} = i \}} \cdot e^{ - (I(t+u)-I(t) )}  \Bigg| M_s : s\le t \right]   \right] .
\end{aligned}
\end{align}
Using the Markov property of $M$, the last expression in \eqref{eq:derivation-of-vx} can be written as
\begin{multline*}
\begin{aligned}
 =& \sum_{j \in E} \pi_j
\cdot \E_j \left[ e^{ - I(t) } \left( \prod_{k=1}^m \ell(t_k, y_k) \right)
\cdot
\sum_{l \in E}  1_{\{ M_{t} =l\}} \cdot \E_l \left[
1_{\{ M_{u} =i\}} \cdot e^{ - I(u)}   \right]   \right] \\
=& \sum_{l \in E}  \E_l \left[
1_{\{ M_{u} =i\}}  e^{ - I(u)}   \right] \cdot
\E^{\vp} \left[ 1_{\{ M_{t} =l\}}  \cdot e^{ - I(t) } \prod_{k=1}^m \ell(t_k, y_k)
 \right]
\\
&
= \sum_{l \in E}  \E_l \left[
1_{\{ M_{u} =i\}}  e^{ - I(u)}   \right] \cdot L_l^{\vp} ( t  , m : (\sigma_k, y_k ), k \le m).
\end{aligned}
\end{multline*}
Then the explicit form of $\vP$ in (\ref{vP-explicit}) implies that for $\sigma_m \le t \le t+ u < \sigma_{m+1}$, we have
\begin{multline}
 \label{eq:semi-group}
\Pi_i (t+u) = \frac{\sum_{l \in E}  L_l^{\vp} ( t  , m : (\sigma_k, y_k ), k \le m) \cdot \E_l \left[
1_{\{ M_{u} =i\}} \cdot e^{ - I(u)}   \right]  }
{\sum_{j \in E} \sum_{l \in E}  L_{l}^{\vp} ( t  , m : (\sigma_k, y_k ), k \le m) \cdot  \E_{l} \left[
1_{\{ M_{u} =j\}}  e^{ - I(u)}   \right] } \\
= \frac{\sum_{l \in E}  \Pi_l(t) \cdot \E_l \left[
1_{\{ M_{u} =i\}}  e^{ - I(u)}   \right]  }
{\sum_{j \in E} \sum_{l \in E}  \Pi_{l}(t) \cdot  \E_{l} \left[
1_{\{ M_{u} =j\}}  e^{ - I(u)}   \right] }
=  \frac{  \E^{\vP_t} \left[
1_{\{ M_{u} =i\}}  e^{ - I(u)}   \right]  }
{\sum_{j \in E}  \E^{\vP_t} \left[
1_{\{ M_{u} =j\}}  e^{ - I(u)}   \right] }
= \frac{    \P^{\vp}  \{ \sigma_1 > u , M_u =i  \}  }
 {   \P^{\vp}  \{ \sigma_1 > u  \} }\Bigg|_{\vp = \vP_t}.
\end{multline}
On the other hand, the expression in (\ref{def:L}) gives
\begin{multline}
\label{eq:derivation-of-jump-part}
L_i^{\vp} ( \sigma_{m+1} , m+1 : (\sigma_k, Y_k ), k \le m+1)
 = \E^{\vp} \left[
1_{\{ M_t =i\}}   e^{ - I(t) }    \prod_{k=1}^{m+1} \ell(t_k, y_k)
 \right] \Bigg|_{ \substack{ t = \sigma_{m+1} \\  (t_k=\sigma_k, y_k = Y_k )_{k \le m+1}  } } \\
\begin{aligned}
= \lambda_i f_i(Y_{m+1})   \E^{\vp} \left[
1_{\{ M_t =i\}}   e^{ - I(t) }    \prod_{k=1}^m \ell(t_k, y_k)
\right] \Bigg|_{ \substack{ t = \sigma_{m+1} \\  (t_k=\sigma_k, y_k = Y_k )_{ k \le m } }}.
\end{aligned}
\end{multline}
Observe that for fixed time $t$, we have $M_t = M_{t-}$, $\P^{\vp}$-a.s. and $L_i^{\vp} (t , m : (t_k, y_k ), k \le m)= L_i^{\vp} (t- , m : (t_k, y_k ), k \le m)$ when $t_m < t$.
Then we have 
\begin{align*}
L_i^{\vp} ( \sigma_{m+1} , m+1 : (\sigma_k, Y_k ), k \le m+1) =
\lambda_i f_i(Y_{m+1}) \cdot  L_i^{\vp} ( \sigma_{m+1}- , m : (\sigma_k, Y_k ), k \le m),
\end{align*}
due to \eqref{eq:derivation-of-jump-part}. Hence, 
at arrival times $\sigma_1, \sigma_2, \ldots$ of $X$, the process $\vP$ exhibits a jump behavior and satisfies the recursive relation $\Pi_i (\sigma_{m+1} ) =$
\begin{align}
\label{eq:jumps-of-vP}
 \frac{ \lambda_i f_i(Y_{m+1})  L_i^{\vp} ( \sigma_{m+1}- , m : (\sigma_k, Y_k ), i \le m) }
{ \sum_{j \in E} \lambda_j f_j(Y_{m+1})  L_j^{\vp} ( \sigma_{m+1}- , m : (\sigma_k, Y_k ), k \le m) }
=\frac{ \lambda_i f_i(Y_{m+1})  \Pi_i (\sigma_{m+1} - ) }
{ \sum_{j \in E} \lambda_j f_j(Y_{m+1})  \Pi_j (\sigma_{m+1} - ) }
\end{align}
for $m \in \N $.

The identities in \eqref{eq:semi-group} and \eqref{eq:jumps-of-vP} give (\ref{eq:rel-pi-x}-\ref{eq:x-i}). By repeating (\ref{eq:derivation-of-vx}-\ref{eq:semi-group}) with $m=0$ (i.e., with no arrivals on $[0,t+s]$), we see that the paths $t \mapsto \vx(t,\vp)$ have the semigroup property $\vx(t+u , \vp ) = \vx (u , \vx(t,\vp) )$.
\end{proof}

\renewcommand{\thesection}{A2}
\refstepcounter{section}
\makeatletter
\renewcommand{\theequation}{\thesection.\@arabic\c@equation}
\makeatother

\section{Supplementary Results and Other Proofs}
\label{sec:appendix}

\begin{proof}[\textbf{Proof of Proposition~\ref{prop:uniform-convergence}.}]
The inequality $V_m(s, \vp) \le V (s, \vp)$ is immediate. To show the second inequality, let
$\tau $ be an $\F$-stopping time less than $s$ $\P$-a.s.. Then we have
\begin{multline}
\label{proof-ineq1} \E^{\vp} \left[  \int_0^{\tau} e^{- \rho t }  \, C(\vP_t ) dt + e^{- \rho
\tau } H\left(\vP_{\tau} \right) \right] = \E^{\vp} \left[  \int_0^{\tau\wedge \sigma_m} e^{-
\rho t }  \, C(\vP_t ) dt + e^{- \rho \tau \wedge \sigma_m  } H\left(\vP_{\tau \wedge \sigma_m
} \right) \right.
\\
\begin{aligned}
& \hspace{1in} \left. + 1_{ \{  \tau > \sigma_m \} } \left[ \int_{  \sigma_m}^{\tau} e^{- \rho
t } \, C(\vP_t ) dt + e^{- \rho \tau } H\left(\vP_{\tau} \right)
- e^{- \rho \sigma_m  } H\left(\vP_{ \sigma_m} \right)  \right] \right] \\
 &\le
\E^{\vp} \left[  \int_0^{\tau\wedge \sigma_m} e^{- \rho t }  \, C(\vP_t ) dt +
e^{- \rho \tau \wedge \sigma_m  } H\left(\vP_{\tau \wedge \sigma_m } \right) \right.\\
&\hspace{1in} \left. +1_{ \{  \tau > \sigma_m \} }e^{- \rho \sigma_m  }  \left[ \| C
\|\int_0^{T - \sigma_m } e^{- \rho t }\, dt + e^{- \rho (\tau - \sigma_m) } H\left(\vP_{\tau}
\right)
-  H\left(\vP_{ \sigma_m } \right)  \right] \right] \\
&\le \E^{\vp} \left[ \int_0^{\tau\wedge \sigma_m} e^{- \rho t }  \, C(\vP_t ) dt + e^{- \rho
\tau \wedge \sigma_m  } H\left(\vP_{\tau \wedge \sigma_m } \right) \right] +  ( T\| C \| + 2 \|
H \| ) \cdot \E^{\vp} \left[e^{- \rho \sigma_m  }  \, 1_{ \{ T > \sigma_m  \}} \right]
\end{aligned}
\end{multline}
where the last line follows since
$\tau \le s \le T$ and $\{ \tau > \sigma_m \} \subseteq \{ T > \sigma_m \}$. Using the
Cauchy-Schwarz inequality and the inequalities $\P^{\vp} \{ T > \sigma_m  \} \le  \E^{\vp}[ 1_{
\{ T > \sigma_m  \} } ( T/ \sigma_m ) ] \le T \cdot \E^{\vp}[ 1/ \sigma_m ]$ we obtain
\begin{multline}
\label{proof-ineq2} \E^{\vp} \left[  \int_0^{\tau} e^{- \rho t }  \, C(\vP_t ) dt + e^{- \rho
\tau } H\left(\vP_{\tau} \right) \right] \\ \le \E^{\vp} \left[ \int_0^{\tau\wedge \sigma_m}
e^{- \rho t }  \, C(\vP_t ) dt + e^{- \rho \tau \wedge \sigma_m  } H\left(\vP_{\tau \wedge
\sigma_m } \right) \right] + ( T\| C \| + 2 \| H \| ) \sqrt{ T \, \E^{\vp}[ 1/ \sigma_m ] \,
\E^{\vp}[ e^{ -2 \rho \sigma_m } ] }.
\end{multline}

Note that given $M$, 
we have
$\P^{\vp} \left[ \sigma_1 > t | M \right] = e^{- I(t)}$,
where $I(\cdot)$ is defined as in (\ref{def:I}). This implies $\E^{\vp} \left[ e^{- u \sigma_1
} | M \right] = \E^{\vp} \left[ \int_{\sigma_1}^{\infty} u \cdot e^{- u t } dt \, \big| M
\right] =  \int_{0}^{\infty} \P^{\vp} \left[ \sigma_1 \le t | M \right] u \cdot e^{- u t } dt
=$
\begin{align*}
 \int_{0}^{\infty}  \left[ 1- e^{-I(t) }\right] u \cdot e^{- u t } dt \le  \int_{0}^{\infty}  \left[ 1- e^{-\bar{\lambda} t }\right] u \cdot e^{- u t } dt = \frac{ \bar{\lambda} }{ u + \bar{\lambda} }.
\end{align*}
The process $X$ has independent increments conditioned on $M$. Then, the inequality $\E^{\vp}
\left[ e^{- u \sigma_m  } | M \right] \le \left(  \frac{ \bar{\lambda} }{ u + \bar{\lambda} }
\right)^m $ follows by induction and we have
\begin{align}
\label{eq:bound-on-Fourrier} \E^{\vp} \left[ e^{- u \sigma_m  } \right] \le \left(  \frac{
\bar{\lambda} }{ u + \bar{\lambda} }  \right)^m ,
\end{align}
for all $m \in \N$. Moreover, since $1/ \sigma_m = \int_0^{\infty} e^{- \sigma_m u } du $, the
inequality in \eqref{eq:bound-on-Fourrier} gives 
$
\E^{\vp} \left[ 1 /  \sigma_m    \right] \le \int_0^{\infty}
 ( \bar{\lambda}^m /     u+ \bar{\lambda} )^m  du = \bar{\lambda} / ( m-1 ),
$
for $m \ge 2$. By using this upper bound in (\ref{proof-ineq2})
and taking the supremum of both sides we obtain (\ref{eq:uniform-bound}).
\end{proof}

\begin{proof}[\textbf{Proof of Lemma~\ref{lem:prop-of-J}.}]
Boundedness and monotonicity are immediate by the definition of the operator $J$ in \eqref{def:J}. To establish the convexity, we will show that expression in (\ref{eq:J-expectations}) is convex (in $\vp$) for each $t$ and $s$. 

We first note that 
$
\E^{\vp} \left[ e^{- I(t)} \right] =  \sum_{j \in E}  \pi_j  \E_j \left[ e^{- I(t)}  \right]
$ and
$m_i(t,\vp) 
=  \sum_{j \in E}  \pi_j \E_j \left[ 1_{ \{  M_t =i  \}  } e^{- I(t)} \right]
$
are linear in $\vp$ where $m_i(t,\vp)$ is defined in \eqref{def:m} for $i \in E$ and $\E_j$ is the expectation operator $\E \left[ \cdot \; | M_0 =j \right]$ for $j \in E$. Then we see that the expression
$
\E^{\vp} \left[ e^{- I(t)} \right]
   e^{- \rho t }    H\left(\vx( t, \vp ) \right) = \max_{k \in \mathcal{A}} e^{- \rho t }
    \sum_{i \in E}  \mu_{k,i} \, m_i(t,\vp) 
$ 
is convex as the upper envelope of convex functions. Next we let $\vp \mapsto w(s, \vp )$ be a
convex mapping for each $s \ge 0$. Then we have
$
w(s, \vp ) = \sup_{k \in K_s } \beta_{k,0} (s) + \beta_{k,1} (s) \pi_1 + \ldots + \beta_{k,n} (s) \pi_n ,
$
for some index set $K_s$, and each $\beta_{k,i} (s)$ is a function in $s$. Using this characterization with the definition of the operator $S_i$ in (\ref{def:S})
we obtain $\int_{0}^{t} e^{- \rho u}     \sum_{i \in E}  \E^{\vp}  \left[   1_{ \{ M_u =i \} } e^{- I(u)} \right] \cdot \lambda_i S_i w(s-u, \vx(u, \vp))du =
\int_{0}^{t} e^{- \rho u}     \sum_{i \in E} \lambda_i  \, m_i(u,\vp) \; \cdot
$
\begin{align*}
\begin{aligned}
&\qquad \qquad
\left[
\int_{\R^d}\,   \sup_{k \in K_{s-u} }
\left(
\beta_{k,0} (s-u) + \sum_{j \in E} \beta_{k,j} (s-u)
\frac{
\lambda_j f_j(y) \, m_j(u,\vp) 
}
{
\sum_{l \in E} \lambda_l f_l(y) \, m_l(u,\vp) 
}
\right) f_i(y) \nu(dy) \right] du \\
&= \int_{0}^{t} e^{- \rho u} \left[
\int_{\R^d}\,   \sup_{k \in K_{s-u} }
\left(
\sum_{j \in E} \left[ \beta_{k,j} (s-u) + \beta_{k,0} (s-u) \right]
\lambda_j f_j(y) \, m_j(u, \vp) 
\right) \nu(dy) \right]
%
%
du.
\end{aligned}
\end{align*}
Since the expression inside the supremum operator are linear in $\pi$, the integrand in the inner integral is convex, and therefore so is the expression above. 
Also note that
$
 \int_{0}^{t} e^{- \rho u}     \sum_{i \in E}     m_i(u, \vp)
 C( \vx(u, \vp)) du $ $= \int_{0}^{t} e^{- \rho u}     \sum_{i \in E}  c_i \,  m_i(u, \vp) 
 du
$,  
where both the integrand and the integral 
are linear in $\vp$. Finally, as the sum of three convex functions $\vp \mapsto Jw(t,s,\vp)$ is convex. Since $J_0w(s,\vp)$ is the supremum of convex functions, it is again convex.
\end{proof}

\begin{proof}[\textbf{Proof of Lemma~\ref{rem:continuity-of-J}.}]
Let us define $ \Upsilon_T \triangleq \{ (t,s) \in \R_+^2 : 0 \le t \le s \,  , \, s \le T \} $. Then the mapping
$
( t, s , \vp) \mapsto \E^{\vp} \left[ e^{- I(t)} \right]
 \cdot e^{- \rho t }  \cdot  H\left(\vx( t, \vp ) \right)
 =  \left( \sum_{j \in E}  \pi_j  \E_j \left[ e^{- I(t)}  \right] \right)
 e^{- \rho t }  \cdot  H\left(\vx( t, \vp ) \right)$
is continuous on the compact set $\Upsilon_T \times D$ due to bounded convergence theorem, the continuity of $H(\cdot)$, and regularity of paths $t \mapsto \vx( t, \vp )$.

For a (bounded) continuous function $w(\cdot, \cdot )$ on $ [0,T] \times D $ , the function $S_iw(\cdot, \cdot )$ is again continuous for $i \in E $ due to bounded convergence theorem. Next let $(t_m, s_m , \vp_m)_{m \in \N}$ be a sequence converging to a point $(t, s , \vp) \in \Upsilon_T \times D $, and let us denote
$F_i(u,s, \vp) \triangleq C( \vx(u, \vp )) +\lambda_i S_i w(s-u, \vx(u, \vp ))$ for typographical convenience. Then
\begin{align*}
&\Bigg| \int_{0}^{t} e^{- \rho u}     \sum_{i \in E}  \E^{\vp}  \left[  1_{ \{ M_u =i \} } e^{- I(u)} \right] F_i(u,s, \vp) du
- \int_{0}^{t_m} e^{- \rho u}     \sum_{i \in E}  \E^{\vp_m}  \left[   1_{ \{ M_u =i \} } e^{- I(u)} \right]   F_i(u ,s_m , \vp_m ) \, du \Bigg|\\
&\le \Bigg| \int_{t_m}^{t} e^{- \rho u}     \sum_{i \in E}  \E^{\vp}  \left[   1_{ \{ M_u =i \} } e^{- I(u)} \right]  F_i(u,s, \vp) \, du \Bigg| \\
 &\qquad \qquad + \Bigg|  \int_{0}^{t_m} e^{- \rho u} \cdot  \sum_{i \in E} \Big( \E^{\vp}  \left[   1_{ \{ M_u =i \} } e^{- I(u)} \right]  F_i(u,s, \vp)
  -  \E^{\vp_m}  \left[   1_{ \{ M_u =i \} } e^{- I(u)} \right]  F_i(u,s_m, \vp_m)
  \Big)
  du \Bigg| \\
&\le ( \| C \| +  \bar{\lambda} \| w \| )  \int_{t_m}^{t} e^{- \rho u} \, du +   \int_{0}^{T}
e^{- \rho u} \sum_{i \in E} \Bigl| \Bigl( \ldots - \ldots \Bigr) \Bigr| du.
\end{align*}
Note that as $m \to \infty$, the second integrand above goes to $0$, and the whole expression vanishes due to dominated convergence theorem.
Hence, we conclude that $J w (t, s, \vp)$ in (\ref{eq:J-expectations}) is continuous on $ \Upsilon_T \times D $. Since this last set is compact, it follows that $J w (t, s, \vp)$ is uniformly continuous and $(s,\vp) \mapsto J_0 w(s,\vp) = \sup_{t \le s } J_0 w(t.s,\vp)$ is continuous on $[0,T] \times D $.
\end{proof}

\noindent To prove Proposition \ref{prop:V-n-epsilon}, we first establish the following intermediate
result.

\begin{prop}
\label{prop-with-S-eps-stopping-times} For every $\eps \ge 0$, let us define
\begin{align}\label{eq:defn-r-m}
r_{m}^{\eps}(s, \vec{\pi}) \triangleq \inf\{ t \in [0,s]: J v_{m}(t,s,\vec{\pi}) \ge J_0 v_{m}
(s,\vec{\pi})-\eps\}, \qquad
\vec{\pi} \in D,
\end{align}
\begin{align*}
S_1^{\eps}(s,\vp) \triangleq r_0^{\eps}(s, \vp) \wedge \sigma_1
\quad \text{and} \quad S_{m+1}^{\eps}(s,\vec{\pi}) \triangleq
\begin{cases}
r_{m}^{\eps/2}(s,\vp) & \text{if $\sigma_1>r_m^{\eps/2}(s,\vp)$},
\\ \sigma_1+ S_m^{\eps/2}
(s - \sigma_1 , \vP_{\sigma_1}) & \text{if $\sigma_1 \leq r_m^{\eps/2}(s,\vp)$}.
\end{cases}
\end{align*}
Then, for every $m \geq 1$ 
we have
\begin{equation}
\label{eq:eps-opt} \E^{\vp} \left[ \int_0^{S_m^{\eps}(s,\vp) } e^{- \rho t} C( \vP_{t} )  dt +
e^{- \rho \cdot S_m^{\eps}(s,\vp)  } H\left(\vP_{ S_m^{\eps}(s,\vp) } \right) \right] \ge
v_{m}(s, \vec{\pi})-\eps.
\end{equation}
\end{prop}
\begin{proof}
We will prove (\ref{eq:eps-opt}) by an induction on $m \in \mathbb{N}.$ For $m=1$,
thanks to
(\ref{def:J}) and (\ref{eq:defn-r-m}) the left-hand-side of (\ref{eq:eps-opt})
equals
$\E\left[  \int_0^{  r_0^{\eps}(s, \vec{\Pi}_0) \wedge \sigma_1  } e^{- \rho t} C( \vP_{t} )  dt
+ e^{- \rho   \cdot  r_0^{\eps}(s, \vec{\Pi}_0) \wedge \sigma_1  } H\left(\vP_{
r_0^{\eps}(s, \vec{\Pi}_0) \wedge \sigma_1 } \right)\right]$
$J H(r_0^{\eps}(s, \vec{\pi}), s, \vec{\pi}) \equiv J v_0 (r_0^{\eps}(s, \vec{\pi}), s, \vec{\pi}) \ge v_1(s, \vec{\pi})-\eps$, which proves (\ref{eq:eps-opt}) for $m=1$.

Now, let us suppose
(\ref{eq:eps-opt}) holds for $\eps \geq 0$, and for some $m>1$, and let us prove that it also
holds when $m$ is replaced by $m+1$. Since $S_{m+1}^{\eps}(s,\vp) \wedge \sigma_1 =
r_m^{\eps}(s, \vec{\Pi}_0) \wedge \sigma_1$, we have
\begin{multline*}
\hspace{-0.5cm}
\begin{aligned}
&\E^{\vp} \left[  \int_0^{  S_{m+1}^{\eps  }(s,\vp)   } \!\! e^{- \rho t} C( \vP_{t} )  dt +
e^{- \rho \cdot S_{m+1}^{\eps  }(s,\vp)  }  \cdot H\left(\vP_{ S_{m+1}^{\eps  }(s,\vp) } \right) \right]\\
&= \E^{\vp} \left[ \int_0^{  S_{m+1}^{\eps  }(s,\vp)  \wedge \sigma_1 } e^{- \rho t} C( \vP_{t}
)  dt +
1_{\{S_{m+1}^{\eps  }(s,\vp)<\sigma_1\}} e^{- \rho \cdot S_{m+1}^{\eps  }(s,\vp)  } H\left(\vP_{ S_{m+1}^{\eps  }(s,\vp) } \right) \right. \\
&\hspace{1.5in}
 \left.  + 1_{\{S_{m+1}^{\eps  } (s,\vp) \ge \sigma_1\}} \Bigl[
\int_{\sigma_1 }^{  S_{m+1}^{\eps  }(s,\vp)  \wedge \sigma_1 } \!\! e^{- \rho t} C( \vP_{t} )
dt + e^{- \rho \cdot S_{m+1}^{\eps  }(s,\vp)  }  H\left(\vP_{ S_{m+1}^{\eps  }(s,\vp) } \right)
\Bigr]  \right]\\
&= \E^{\vp} \left[ \int_0^{  r_m^{\eps /2 }(s, \vec{\Pi}_0)  \wedge \sigma_1 } \!\! e^{- \rho
t} C( \vP_{t} )  dt +
1_{\{ r_m^{\eps /2 }(s, \vec{\Pi}_0) <\sigma_1\}} H\left( \vP_{ r_m^{\eps /2 }(s, \vec{\Pi}_0) } \right)
+ 1_{\{ r_m^{\eps /2 }(s, \vec{\Pi}_0) \ge \sigma_1\}} \; \cdot
\right. \\
&\;
 \left.   \Bigl[
\int_{\sigma_1 }^{  \sigma_1+ S_m^{\eps /2 /2} (s - \sigma_1 , \vP_{\sigma_1}) } \!\! e^{- \rho
t} C( \vP_{t} )  dt + e^{- \rho \cdot  \left( \sigma_1+ S_m^{\eps /2 /2} (s - \sigma_1 ,
\vP_{\sigma_1}) \right)   } H\left(\vP_{\sigma_1+ S_m^{\eps /2 /2} (s - \sigma_1 ,
\vP_{\sigma_1} ) } \right)
\Bigr]  \; \right]\\
&= \E^{\vp} \left[  \int_0^{  r_m^{\eps /2 }(s, \vec{\Pi}_0)  \wedge \sigma_1 } \!\! e^{- \rho
t} C( \vP_{t} )  dt + 1_{\{ r_m^{\eps /2 }(s, \vec{\Pi}_0) <\sigma_1\}} H\left( \vP_{ r_m^{\eps
/2 }(s, \vec{\Pi}_0) } \right)
\right.
\left.
+1_{\{ r_m^{\eps /2 }(s, \vec{\Pi}_0) \ge \sigma_1\}} e^{- \rho \cdot \sigma_1 }  f_m( s -
\sigma_1 , \vec{\Pi}_{\sigma_1})
 \right]
\end{aligned}
\end{multline*}
where the last line follows from the strong Markov property and where 
\begin{equation*}
\label{ineq:for-f} f_{m}(u, \vec{\pi})=\E^{\vp} \left[ \int_0^{ S_m^{\eps/2}(u, \vp) } \!\!
e^{- \rho t} C( \vP_{t} )  dt + e^{- \rho \cdot S_m^{\eps/2}(u, \vp) } \cdot H\left(
\vec{\Pi}_{S_m^{\eps/2}(u, \vp) } \right) \right] \ge v_{m}(u, \vec{\pi})-\eps/2.
\end{equation*}
The inequality above follows from the induction hypothesis. Then we obtain
\begin{multline*}
\E^{\vp} \left[  \int_0^{  S_{m+1}^{\eps  }(s,\vp)   } \!\! e^{- \rho t} C( \vP_{t} )  dt +
e^{- \rho \cdot S_{m+1}^{\eps  }(s,\vp)  }  \cdot H\left(\vP_{ S_{m+1}^{\eps  }(s,\vp) } \right) \right] \ge \E^{\vp} \left[ \int_0^{  r_m^{\eps /2 }(s, \vec{\Pi}_0)  \wedge \sigma_1 } \!\! e^{- \rho t} C( \vP_{t} )  dt \right. \\
\begin{aligned}
&+ 1_{\{ r_m^{\eps /2 }(s, \vec{\Pi}_0) <\sigma_1\}} H\left( \vP_{ r_m^{\eps /2 }(s,
\vec{\Pi}_0) } \right)
+ 1_{\{ r_m^{\eps /2 }(s, \vec{\Pi}_0) \ge \sigma_1\}} e^{- \rho \cdot \sigma_1 } \cdot v_m( s
- \sigma_1 , \vec{\Pi}_{\sigma_1}) \Bigg] -\frac{\eps}{2}
\end{aligned}
\end{multline*}
$=J v_m (r_{m}^{\eps/2}(\vec{\pi}),s, \vec{\pi})-\frac{\eps}{2} \ge v_{m+1}(\vec{\pi})- \eps$. Here the  equality follows from the definition of the operator $J$ in (\ref{def:J}) and
the second equality follows from (\ref{eq:defn-r-m}). This concludes the proof of
(\ref{eq:eps-opt}).
\end{proof}

\begin{proof}[\textbf{Proof of Proposition~\ref{prop:V-n-epsilon}.}]
The inequality $V_m \ge v_m$
follows 
from (\ref{eq:eps-opt}) since $S_m^{\eps} (s,\vp) \leq s \wedge \sigma_m$
by construction. To prove the reverse inequality $V_m \le v_m$ we
will show 
\begin{equation}\label{eq:V-vs-v}
\E\left[  \int_0^{  \tau \wedge \sigma_m } e^{- \rho t} C( \vP_{t} )  dt + e^{- \rho \cdot \tau \wedge \sigma_m } \cdot  H\left( \vP_{\tau \wedge \sigma_m } \right)  \right] \le v_m(s, \vec{\pi}),
\end{equation}
for every bounded stopping time $\tau \le s$ and $m \in \mathbb{N}$, by showing 
\begin{multline}
\label{eq:V-vs-v-int}
\E\left[  \int_0^{  \tau \wedge \sigma_m } e^{- \rho t} C( \vP_{t} )  dt +
e^{- \rho \cdot \tau \wedge \sigma_m } \cdot  H\left( \vP_{\tau \wedge \sigma_m }\right)\right]
\le \E\left[  \int_0^{  \tau \wedge \sigma_{m-k+1} } e^{- \rho t} C( \vP_{t} )  dt \right. \\
\begin{aligned}
 &+\left.
1_{\{\tau \ge \sigma_{m-k+1}\}}
e^{- \rho \cdot \sigma_{m-k+1} } v_{k-1}\left(s  - \sigma_{m-k+1}  , \vP_{ \sigma_{m-k+1} } \right) \right.
\left.+1_{\{\tau < \sigma_{m-k+1}\}}  e^{- \rho \cdot \tau  } \cdot  H\left( \vP_{\tau} \right)
\right] =: RHS_{k-1},
\end{aligned}
\end{multline}
for $k=1, \cdots, m+1$. The inequality (\ref{eq:V-vs-v}) will then follow from
(\ref{eq:V-vs-v-int}) by taking $k=m+1$. For $k=1$, (\ref{eq:V-vs-v-int}) is satisfied as an equality since
$v_0(s, \cdot)=H(\cdot)$, for all $s \in [0,T]$. Now, let us assume (\ref{eq:V-vs-v-int}) holds
for some $1 \leq k < m+1$, and let us prove that it also holds for $k+1$.

Note that $RHS_{k-1}$ in \eqref{eq:V-vs-v-int} can be written as
$RHS_{k-1}=RHS_{k-1}^{(1)}+RHS_{k-1}^{(2)}$,
in terms of
\begin{align*}
\begin{aligned}
RHS_{k-1}^{(1)} & \triangleq \E\left[
\int_0^{  \tau \wedge \sigma_{m-k} } e^{- \rho t} C( \vP_{t} )  dt + 1_{\{\tau<\sigma_{m-k}\}} e^{- \rho \cdot \tau  } \cdot H(\vP_{\tau}  )\right],
\\RHS_{k-1}^{(2)} & \triangleq \E\Bigg[
1_{\{\tau \ge \sigma_{m-k}\}} \cdot
\Bigg( \int_{ \sigma_{m-k} }^{  \tau \wedge \sigma_{m-k+1} } e^{- \rho t} C( \vP_{t} )  dt
\\ &\quad  + 1_{\{\tau \ge \sigma_{m-k+1}\}}
e^{- \rho \cdot \sigma_{m-k+1} } \cdot v_{k-1}\left( s  - \sigma_{m-k+1}  , \vP_{ \sigma_{m-k+1} } \right)+1_{\{\tau<\sigma_{m-k+1}\}}
e^{- \rho \cdot \tau  } \cdot  H\left( \vP_{\tau}  \right) \Bigg)
 \Bigg].
\end{aligned}
\end{align*}
Lemma~\ref{lem:bremaud} implies that there exists an
$\Fc^X_{\sigma_{m-k}}$-measurable random variable $R_{m-k}   $ such that
\begin{align*}
\tau \wedge \sigma_{m-k+1}=(\sigma_{m-k}+R_{m-k}) \wedge
\sigma_{m-k+1} \quad \text{ on } \; \{\tau \geq \sigma_{m-k}\}.
\end{align*}
Moreover since $\tau \le s$, we have $R_{m-k} \le s -\sigma_{m-k} $ on $\{\tau \geq \sigma_{m-k}\}$. Then we obtain
$RHS_{k-1}^{(2)}=$
\begin{multline*}
 \E\Bigg[
1_{\{\tau \ge \sigma_{m-k}\}} \cdot \Bigg( \! \int_{ \sigma_{m-k} }^{ (\sigma_{m-k}+R_{m-k})
\wedge \sigma_{m-k+1} } \!\! e^{- \rho t} C( \vP_{t} )  dt + 1_{\{\tau \ge \sigma_{m-k+1}\}}
e^{- \rho  \sigma_{m-k+1} }  v_{k-1}\left( s  - \sigma_{m-k+1}  , \vP_{ \sigma_{m-k+1} }
\right)
\\  +1_{\{  \sigma_{m-k}+R_{m-k}   <\sigma_{m-k+1}\}}
e^{- \rho  ( \sigma_{m-k}+R_{m-k} )   } \cdot  H\left( \vP_{ \sigma_{m-k}+R_{m-k} }  \right) \Bigg)
 \Bigg].
 \end{multline*}
Due to strong Markov property, the last expression can be written as
\begin{align}
\label{eq:RHS-k-1-2}
RHS_{k-1}^{(2)}=  \E\left[1_{\{\tau \ge
\sigma_{m-k}  \}} \,  \cdot e^{- \rho \cdot \sigma_{m-k} }  g_{k-1} \left( R_{m-k},s - \sigma_{m-k}, \vec{\Pi}\left( \sigma_{m-k} \right) \right) \right],
\end{align}
where $g_{k-1}(r,u , \vec{\pi}) \triangleq$
\begin{align*}
 \E \Big[
\int_0^{  r \wedge \sigma_{1} } e^{- \rho t} C( \vP_{t} )  dt
+ 1_{\{r < \sigma_1\}}   e^{- \rho r }   H\left(  \vP \left( r \right) \right)
+
1_{\{r \ge \sigma_1\}}   e^{- \rho \sigma_1 }  v_{k-1} \left( u- \sigma_1,  \vP \left( \sigma_1 \right) \right)
\Big] , 
\end{align*}
for $r \le u$.
Then, using the definition of the operator $J$ in \eqref{def:J} we have
\begin{align*}
g_{k-1}(r,u , \vec{\pi}) = J v_{k-1}(r,u,\vp) \le J_0 v_{k-1}(u,\vp) = v_k (u,\vp).
\end{align*}
As a result, we obtain 
$ RHS_{k-1}^{(2)} \le
\E\left[  1_{\{\tau \geq \sigma_{m-k}\}}
 e^{- \rho \cdot \sigma_{m-k} } v_{k}  \left( u- \sigma_{m-k},  \vP \left( \sigma_1 \right) \right)
 \right] $,
and this further implies 
\begin{multline}
\E\left[
\int_0^{ \tau \wedge \sigma_m } e^{- \rho t} C( \vP_{t} )  dt  + e^{- \rho \cdot \tau \wedge \sigma_m }    H\left( \vP_{\tau \wedge \sigma_m }\right)\right] \\
\begin{aligned}
&\le RHS_{k-1} =
\E\left[
\int_0^{  \tau \wedge \sigma_{m-k} } e^{- \rho t} C( \vP_{t} )  dt + 1_{\{\tau<\sigma_{m-k}\}} e^{- \rho \cdot \tau  } \cdot H(\vP_{\tau}  )\right] + RHS^{(2)}_{k-1} \\
&\le
 \E\left[ \int_0^{  \tau \wedge \sigma_{m-k} } e^{- \rho t} C( \vP_{t} )  dt +
1_{\{\tau<\sigma_{m-k}\}} e^{- \rho \cdot \tau  }  H( \vP_{\tau} )
+
1_{\{\tau \geq \sigma_{m-k}\}}
\cdot e^{- \rho \cdot \sigma_{m-k} }  v_{k}  \left( u- \sigma_{m-k},  \vP_{ \sigma_1 } \right)
\right].
\end{aligned}
\end{multline}
Since the last term equals $RHS_k$, this completes the proof of (\ref{eq:V-vs-v-int}) by induction.
Equation (\ref{eq:V-vs-v}) follows when we set $k=m+1$. Finally, taking the
infimum of both sides in (\ref{eq:V-vs-v}), we arrive at the
desired inequality $V_m \le v_m$.
\end{proof}

\begin{proof}[\textbf{Proof of Lemma \ref{lemm:J-dyn-prog}.}]
Using the definition of the operator $J$ in \eqref{def:J} we obtain
\begin{multline*}
Jw(t,s,\vp) = \E^{\vp} \left[ \int_0^{t \wedge \sigma_1} \!\! e^{- \rho t} C(\vP_t ) \, dt +
1_{ \{ t < \sigma_1 \} } \cdot e^{-\rho t} H(\vP_t) + 1_{ \{ \sigma_1 \le t \}} \cdot e^{-\rho
\sigma_1} w(s-\sigma_1, \vP_{\sigma_1} ) \right] \\
\begin{aligned}
& = \E^{\vp}\Bigl[ \int_0^{u \wedge \sigma_1} \!\! e^{- \rho t} C(\vP_t ) \, dt + \int_{u
\wedge \sigma_1}^{t \wedge \sigma_1} \!\! e^{- \rho t} C(\vP_t ) \, dt
 - 1_{ \{ u < \sigma_1 \} }\cdot
e^{-\rho u} H(\vP_{u} ) +
1_{\{ u < \sigma_1 \} } \cdot e^{-\rho u} H(\vP_{u} ) \\
& \hspace{0.3in} +1_{\{ t < \sigma_1 \} } \cdot e^{-\rho t} H(\vP_t) + 1_{\{ \sigma_1 \le u\}}
e^{-\rho \sigma_1}   w(s-\sigma_1, \vP_{\sigma_1}) +
1_{\{ u < \sigma_1 \le t \} }\cdot e^{-\rho \sigma_1} w(s-\sigma_1, \vP_{\sigma_1}) \Bigr] \\
& = Jw(u,s,\vp) + \E^{\vp}\Bigl[ - 1_{ \{ \sigma_1 > u \} } \cdot e^{-\rho u}   H(\vP_{u} )
 + 1_{ \{ \sigma_1
> u \} }  \Big(
\int_{u }^{t \wedge \sigma_1} \!\! e^{- \rho t} C(\vP_t ) \,dt \Big) \\
& \hspace{0.3in} + 1_{ \{ \sigma_1
> u \} }   \Big(
 1_{ \{ \sigma_1 > t \} } \cdot e^{-\rho t}H(\vP_t)
+  1_{ \{ \sigma_1 \le t \} } \cdot e^{-\rho \sigma_1 } \cdot w(s-\sigma_1, \vP_{\sigma_1})
\Big) \Bigr] .
\end{aligned}
\end{multline*}
On $\{ \sigma_1 > u\}$, we have $\sigma_1 \wedge t = u + (\sigma_1 \wedge (t-u))\circ
\theta_u$. Then the Markov property of $\vP$ gives
\begin{align*}
 &Jw(t,s,\vp)= Jw(u,s,\vp) - \P^{\vp} \{ \sigma_1 > u\} e^{-\rho u}
H(\vx(u,\vp)) \\
&+ \E^{\vp} \Bigl[ 1_{\{ \sigma_1 > u \} } e^{-\rho u}\, \E^{\vP_{u} }\left[ \int_{0 }^{t -u}
e^{- \rho t} C(\vP_t ) dt +
1_{\{ \sigma_1 > t-u\} } e^{-\rho(t-u)}H(\vP_{t-u})  \right. \\
& \hspace{3in} \left. +  1_{\{ \sigma_1 \le (t-u) \} }
e^{-\rho \sigma_1} w(s-u-\sigma_1, \vP_{\sigma_1}) \right] \Bigr] \\
& = Jw(u,s,\vp) - \P^{\vp}\{  \sigma_1 > u\} e^{-\rho u} H(\vx(u,\vp)) + \E^{\vp}
\left[1_{\{ \sigma_1 > u \} } \cdot e^{-\rho u} Jw(t-u, s-u, \vP_{u} ) \right] \\
& = Jw(u,s,\vp) + \P^{\vp}\{  \sigma_1 > u\} e^{-\rho u} \left[Jw(t-u, s-u,
\vx(u,\vp))-H(\vx(u,\vp)) \right].
\end{align*} \end{proof}

\begin{proof}[\textbf{Proof of Lemma~\ref{lemm:continue-at-good-states}}]
Let $\vec{e}_i \in D $ denote the point whose $i$'th component is equal to 1. To establish the result it is sufficient to find a closed ball with strictly positive radius around $\vec{e}_i$ (e.g., a region of the form $\{ \vp \in D: || \vp - \vec{e}_i || \le \delta \}$ for some $\delta >0$, where $|| \cdot ||$ denotes the Euclidian norm on $\R^n$)
such that $H(\vp) < v_1 (s,\vp ) \le
V(s,\vp) $ for all points on this closed ball.
%

We first note that there exists a closed ball $B_0$ around $\vec{e}_i$ with positive radius such that $H(\vp) = \max_{k \in \mathcal{A}^*(i) } H_k(\vp)$, for $\vp \in B_0$.
Then on $B_0$ and for small $s > 0$ we have
$ v_1(s,\vp) = \sup_{t \le s } J_0 H(t,s,\vp)  = \max_{k \in \mathcal{A}^*(i)  } \sup_{t \in[0, s]} J^{(k)}_0 H(t, \vp) $,  where $J^{(k)}_0 H(t,\vp) \triangleq$
\begin{align*}
 \E^{\vp}\left[ e^{- I(t)} \right] e^{- \rho t}H_k(\vx(t,\vp))
 + \int_0^t e^{- \rho u } \sum_{j \in E} m_j(u,\vp) \bigl( C(\vx(u,\vp)) + \lambda_j S_j H( \vx(u,\vp) ) \bigr)
du.
\end{align*}
Then, using \eqref{eq:vx-dyn} we have 
$  d J^{(k)}_0
H(t,\vp) / dt\big|_{t=0}= $
\begin{multline*}
 \left( - \rho - \sum_{j \in E} \lambda_j \pi_j \right) H_k(\vp)
+ \sum_{j \in E} \mu_{k,j} \left( \sum_{l \in E} q_{l,j} \pi_l - \lambda_j \pi_j + \pi_j \sum_{l \in E} \lambda_l \pi_l \right)
 + C(\vp) + \sum_{j \in E} \lambda_j \pi_j S_j H(\vp) \ge \\
 \left( - \rho - \sum_{j \in E} \lambda_j \pi_j \right) H_k(\vp)
+ \sum_{j \in E} \mu_{k,j} \left( \sum_{l \in E} q_{l,j} \pi_l - \lambda_j \pi_j + \pi_j \sum_{l \in E} \lambda_l \pi_l \right)
 + C(\vp) + \sum_{j \in E} \lambda_j \pi_j S_j H_k(\vp)  .
\end{multline*}
The right hand side of the inequality above is 
uniformly continuous on the compact set $D$. Its value at the point $\vec{e}_i$ equals 
$c_i - \rho \mu_{k,i} + \sum_{j \ne i} (\mu_{k,j}- \mu_{k,i}) q_{k,j} > 0$. Hence for some $\delta_k
>0$ there exists an open ball (contained in $B_0$) with radius $\delta_k$ around $\vec{e}_i$
such that $d J^{(k)}_0 H(t,\vp) / dt\big|_{t=0}> 0$ for all the points in this ball. Let $B_k$
be the closed ball around the same point $\vec{e}_i$ with radius $\delta_k /2$. Then on the intersection
set $ \bigcap_{k \in \mathcal{A}^*(i)  } B_k  $ the mapping $\vp \mapsto d J^{(k)}_0 H(t,\vp) /
dt\big|_{t=0}$ is strictly positive and $ \sup_{t \ge 0 } J^{(k)}_0 H(t,\vp) > H_k(t,\vp)$ for all $k \in \mathcal{A}^*(i) $. This implies
that $v_1 (s,\vp ) > H(\vp)$ for all $s >0$ on $\bigcap_{k \in \mathcal{A}^*(i)  } B_k$.
\end{proof}

\begin{proof}[\textbf{Proof of Lemma~\ref{lem:nonempty-acceptance-region}}]
%
%
Let $i \in I^*$ for $I^*$ defined in \eqref{optimal-notation}.
To establish the result, we will find $\pi^{s}_{i} <1$ such that $H(\vp) = J_0 w(s,\vp) $ on $\{ (s, \vp)  \in [0,T] \times D \, : \, \pi^{s}_{i} \le \pi_{i} <1\}$ for a bounded function $w(\cdot) \le \|H \| = \bar{\mu} \triangleq \max_{i,k} \mu_{i,k} $. Since $ V $ is bounded by the same upper bound (recall that $c_i \le 0$ for $i \in E$ by assumption) and satisfies $V(s,\vp) = J_0 V(s,\vp)$
we will have $H(\cdot) = V(\cdot)$ on this region.

\noindent \textbf{Part I:} Let us first define 
\begin{align}
\label{def:L-k}
F_k(t,\vp)\triangleq \E^{\vp} \left[ e^{- I(t) - \rho t } \right]
    H_k(\vx( t, \vp ) )
 + \int_{0}^{t}    e^{- \rho u}  \sum_{j \in E} m_j (u,\vp)
 \Big[ C(\vx( t, \vp ) )+ \lambda_j \bar{\mu}  \Big]  du .
\end{align}
Since $H(\vp) \le J_0 w(s, \vp) = \sup_{ t \in [0,s] } Jw(t,s,\vp) \le \sup_{ t \in [0,s] } \max_{ k \in A } F_k(t,\vp) =  \max_{k \in A} \sup_{ t  \in [0,s] } F_k(t,\vp)$ (see \eqref{eq:J-expectations}), it is enough to show that for some $  \pi^{s}_{i}  <  1$ we have $\sup_{t \ge 0} F_k(t,\vp) = H_k(\vp)$ for all $k \in A$.

Let $\hat{\pi}_{i} < 1$ be a value such that $H(\vp) = \max_{k \in \mathcal{A}^*} h_{k}(\vp)$,
where $\mathcal{A}^* \triangleq \{ k \in \mathcal{A} :\, \mu_{k,i} = \bar{\mu}  \}$. That is,
we have $\mu_{k,i} =\bar{\mu} $ for all $k \in \mathcal{A}^*$ (and $i \in I^*$). Note that
$\hat{\pi}_{i}$ can for instance be selected as
 \begin{align*}
 \hat{\pi}_{i} = \max_{k \notin \mathcal{A}^*} \frac{\bar{\mu} - \min_{k,j} \mu_{k,j} }{2\, \bar{\mu} - \min_{k,j} \mu_{k,j} - a_{k,i}}.
 \end{align*}
Let us then define the hitting time $ T( \vp,  \hat{\pi}_{i} ) \triangleq \inf \left\{ t \ge 0 \, : \, x_{i}(t,\vp) \le  \hat{\pi}_{i} \right\}$. For $t \le T( \vp,  \hat{\pi}_{i} )$, we have $  \max_{k \in \mathcal{A}  } H_k \left(\vx( t, \vp ) \right) = \max_{k \in \mathcal{A}^* }  H_k(\vx( t, \vp ) )$, which implies $  \max_{k \in \mathcal{A}} F_k ( t, \vp ) = \max_{k \in \mathcal{A}^* }   F_k( t, \vp ) $. Note that we have
\begin{align}
\label{derivative:L-k}
\frac{ d F_k(t,\vp) }{dt}  =    \sum_{i \in E}  \E^{\vp}  \left[ 1_{ \{ M_t =i \} } e^{- I(t)- \rho t } \right]
 \left\{  - (\lambda_i + \rho)   \cdot H_k(\vx(t,\vp)) +  \frac{ d H_k(\vx(t,\vp))  }{dt} + C( \vx(t,\vp) ) + \lambda_i  \| H \| \right\}
\end{align}
where
\begin{align}
\label{derivative:h-x}
\frac{ d H_k(\vx(t,\vp))  }{dt}  =  \sum_{i \in E} \mu_{k,i} \left( \sum_j^n q_{ji}  x_j (t,\vp) - \lambda_i  x_i (t,\vp)  + x_i (t,\vp)   \sum_j^n \lambda_{j}  x_j (t,\vp)  \right)
\end{align}
due to \eqref{eq:vx-dyn}.
Let us denote $\underline{\mu} \triangleq \min_{k,i} \mu_{k,i} $. For $k \in A^*$, we have $H_k (\vx(t,\vp)) = \bar{\mu} x_{i}(t,\vp) + \sum_{i \ne i} \mu_{k,i} x_{i}(t,\vp) \ge \bar{\mu} x_{i}(t,\vp) + \underline{\mu}( 1- x_{i}(t,\vp) )$. Using this inequality, we get 
an upper bound for the derivative in \eqref{derivative:L-k} as
\begin{align}
\begin{aligned}
\label{upper-bound-on-deriv-K}
\frac{ d F_k(t,\vp) }{dt}
 \le  \E^{\vp}  \left[ e^{- I(t)- \rho t } \right] \Bigg\{   \Big( \bar{\lambda} (\bar{\mu} - \underline{\mu}) - \rho \underline{\mu}\Big) (1- x_{i}(t,\vp)) - \rho \bar{\mu} x_{i}(t,\vp) + c_i x_{i}(t,\vp)  + \frac{ d H_k (\vx(t,\vp))  }{dt}  \Bigg\},
\end{aligned}
\end{align}
where 
$\bar{\lambda} \triangleq \max_{i \in E} \lambda_i $. Moreover, using \eqref{derivative:h-x} it can be shown that for $k \in \mathcal{A}^*$ we have
\begin{align}
\label{upper-bound-on-deriv-h-x}
\begin{aligned}
\frac{ d H_k (\vx(t,\vp))  }{dt}  &= \sum_{j \in A} x_j (t,\vp) \sum_{l \in E} \mu_{k,l}  q_{jl}   - \sum_{l \in E} \mu_{k,l}  \lambda_l  x_l (t,\vp)  + \sum_{l \in E} \mu_{k,l}  x_l (t,\vp)   \sum_{j \in A} \lambda_{j}  x_j (t,\vp)
\\
&\le n \bar{\mu} \left(\max_{l,j} | q_{lj} | \right) \Big(1-x_{i}(t,\vp) \Big)  - \sum_{l \ne i} \mu_{k,l}  \lambda_l  x_l (t,\vp)  +   \bar{\mu} x_{i} (t,\vp)  \sum_{j \ne k} \lambda_{j}  x_j (t,\vp)
\\
&\qquad \quad + \lambda_{i} x_{i} (t,\vp) \sum_{l \ne i} \mu_{k,l}  x_l (t,\vp)
+ \left(    \sum_{l \ne i} \mu_{k,l}  x_l (t,\vp) \right) \left(   \sum_{j \ne k} \lambda_{j}  x_j (t,\vp) \right)
\\
&\le \Big(1-x_{i}(t,\vp) \Big) \cdot \left( 3 \cdot  \bar{\mu}  \cdot  \bar{\lambda} +   n \cdot \left(\max_{l,j} | q_{lj} | \right)  \cdot \bar{\mu}   \right)
\end{aligned}
\end{align}
where the second line follows from the inequality $  \sum_{l \in E} \mu_{k,l} \,  q_{ l l } \le 0$ (recall that $\bar{\mu} = \mu_{k,i} = \max_{k,l} \mu_{k,l}$ and $q_{ii} = - \sum_{i \ne i} q_{i i}$).
The equations \eqref{upper-bound-on-deriv-K} and \eqref{upper-bound-on-deriv-h-x} then imply that for $t < T( \vp, \hat{\pi}_{i} ) $, and for $k \in \mathcal{A}^*$;
\begin{align}
\label{upper-upper-bound-on-deriv-K}
\frac{ d F_k(t,\vp) }{dt} \le  \E^{\vp}  \left[ e^{- I(t)- \rho t } \right]
 \cdot  \Bigg\{   - \rho \bar{\mu} x_{i}(t,\vp) + c_i x_{i}(t,\vp)  + \Big(1- x_{i}(t,\vp) \Big) \cdot  G  \Bigg\}.
\end{align}
where
$ G \triangleq 4 \cdot  \bar{\mu}  \cdot  \bar{\lambda} +   n \cdot \left(\max_{l,j} | q_{lj} | \right)  \cdot \bar{\mu}   - ( \rho + \bar{\lambda})  \cdot \underline{\mu}.
$ 
Note that the assumption '$\rho > 0$ or $c_i>0$' in Lemma~\ref{lem:nonempty-acceptance-region} assures that $dF_k(t,vp) / dt \big|_{t=0}$ is negative as $\pi_i \to 1$. Therefore, if we define
\begin{align*}
\hat{ \hat{\pi} }_{i} \triangleq \max \left\{  \hat{\pi}_{i} \, , \, \frac{G  }{\rho \bar{\mu} - c_i  + G} \right\}
= \max \left\{  \hat{\pi}_{i} \, , \, \frac{4   \bar{\mu}    \bar{\lambda} +   n  \left(\max_{l,j} | q_{lj} | \right)   \bar{\mu}   - ( \rho + \bar{\lambda})    \underline{\mu}  }{ - c_i+  n \bar{\mu}\left(\max_{l,j} | q_{lj} | \right) + 3  \bar{\lambda} \bar{\mu}+ (\bar{\mu} - \underline{\mu}) (\rho +  \bar{\lambda} )} \right\} < 1,
\end{align*}
we have $d F_k(t,\vp) / dt \le 0 $ on $t \in [0, T( \vp, \hat{\pi}_n )]$ for all $k \in A^*$ and for all $\vp $ such that $\pi_{i} > \hat{ \hat{\pi} }_{i}$. This implies that $JH(t,s, \vp) \le H(\vp)$ on this region.

\noindent \textbf{Part II:} Next, let $T( \vp, \hat{ \hat{\pi} }_{i} )$ be the hitting time of the deterministic path $x_{i}(t,\vp)$ to the level $\hat{ \hat{\pi} }_{i}$. Below we show that there exists $\pi^{s}_{i}$ such that
\begin{align}
\label{remaining-inequalities}
F_k (t , \vp) \le \E^{\vp}  \left[ \int_0^{t \wedge \sigma_1} e^{ - \rho t }c\, dt +
e^{ - \rho t \wedge \sigma_1 }\bar{\mu} \right] \le \bar{\mu} \,  \pi^{s}_{i} + m (1 - \pi^{s}_{i}) \le H(\vp)
\end{align}
for all $k \in \mathcal{A}$ (not just $\mathcal{A}^*$) and for all $t \ge T( \vp, \hat{ \hat{\pi} }_{i} )$ on the region $\{  \vp \in D ; \,  \pi_{i} \ge  \pi^{s}_{i}\}$. This will further imply that $JH(t,s, \vp) \le H(\vp)$ for all $t \ge 0$ for a point $\vp$ falling on the latter region, and we will have $H(\vp) \le J_0H(s, \vp) = \sup_{t \in [0,s] } JH(t,s, \vp) \le H(\vp)$.

Note that the first inequality in \eqref{remaining-inequalities} follows from $C(\cdot) \le  c $ and $H(\cdot) \le \bar{\mu}$. For a given value $\pi^{s}_{i}$ the last inequality is true for all the points on $\{  \vp \in D ; \,  \pi_{i} \ge  \pi^{s}_{i}\}$ since
\begin{align*}
H(\vp) = \sup_{k \in A^*} H_k(\vp) = \bar{\mu} \pi_{i} +  \sup_{k \in A^*} \sum_{i \ne i} \mu_{k,i} \pi_i \ge
\bar{\mu} \pi_{i}  + m ( 1 - \pi_{i} ) \ge \bar{\mu} \, \pi^{s}_{i}  + m ( 1 -\pi^{s}_{i} ).
\end{align*}
Hence it remains to show that the second inequality holds for some $\pi^{s}_{i}$.

For $\pi_{i} > \hat{ \hat{\pi} }_i $ we have
$
 \hat{ \hat{\pi} }_{i} = \pi_{i} + \int_0^{T( \vp, \hat{ \hat{\pi} }_{i} )}  \frac{ d ( x_{i} (t,\vp)) }{dt} \, dt.
$ 
Then, thanks to \eqref{eq:vx-dyn} we get $0 \ge \hat{ \hat{\pi} }_{i} - \pi_{i} =$
\begin{align*}
  \int_0^{T( \vp, \hat{ \hat{\pi} }_{i} )} \left( \sum_{j \in E} q_{ji}  x_j (t,\vp) - \lambda_{i}  x_{i} (t,\vp)  + x_{i} (t,\vp)   \sum_{j \in E} \lambda_{j}  x_j (t,\vp) \right)  dt
\ge  \int_0^{T( \vp, \hat{ \hat{\pi} }_{i} )} \Big(  q_{ii}  - \lambda_{i}   \Big)   dt
\end{align*}
$= \Big( q_{ii}  - \lambda_{i}  \Big) \cdot T( \vp, \hat{ \hat{\pi} }_{i} )$, which further implies
\begin{align}
\label{exit-time-lower-bond}
T( \vp, \hat{ \hat{\pi} }_{i} ) \ge (\pi_{i} - \hat{ \hat{\pi} }_{i} ) / (- q_{ii}  - \lambda_{i}      ) .
\end{align}

\noindent \textbf{Case I: $\rho > 0$.}
By \eqref{exit-time-lower-bond} we get the inequality
$\E^{\vp} \exp\left(- \rho \cdot T( \vp, \hat{ \hat{\pi} }_{i} ) \wedge \sigma_1 \right) \le$
\begin{align*}
 \E^{\vp} \exp\left(- \rho \Bigl[  \frac{ \pi_{i} - \hat{ \hat{\pi} }_{i} }{ - q_{ii}  + \lambda_{i}  } \,  \wedge \sigma_1\Bigr]\right)
 = \int_0^{\infty} \exp\left(- \rho \Bigl[  \frac{ \pi_{i} - \hat{ \hat{\pi} }_{i} }{ - q_{ii}  + \lambda_{i}  } \, \wedge u \Bigr]  \right) \sum_{i \in E} \E^{\vp} \left[ 1_{ \{ M_u =i \}} e^{- I(u)} \lambda_i \right] \, du .
\end{align*}
The last expression above is strictly decreasing in $\pi_{i}$ and equals $1$ at $\pi_{i} = \hat{ \hat{\pi} }_{i}$. Moreover the mapping $\pi_{i} \mapsto \bar{\mu} \pi_{i}  + \underline{\mu} (1-\pi_{i} )$ is increasing and equals $\bar{\mu}$ at $\pi_{i} =1$. Therefore there exists a unique $\pi^{s}_{i} \in [\hat{ \hat{\pi} }_{i}, 1)$ defined as
\begin{align}
\label{def:pi-i-star}
\pi^{s}_{i} \triangleq \inf  \left\{ \pi_{i} \ge \hat{ \hat{\pi} }_{i}  \, : \,
 \bar{\mu}\,  \E^{\vp} \exp\left(- \rho \Bigl[  \frac{ \pi_{i} - \hat{ \hat{\pi} }_{i} }{ - q_{ii}  + \lambda_{i}  } \,  \wedge \sigma_1\Bigr]\right)
 \le \bar{\mu} \pi_{i}  + \underline{\mu} (1-\pi_{i} )\right\}
 < 1,
\end{align}
such that the inequality in \eqref{def:pi-i-star} holds for all $\pi_{i} \in [ \pi^{s}_{i}, 1]$. The definition of $\pi^{s}_{i}$ implies that for all the points $\vp$ with $\pi_{i} \ge \pi^{s}_{i}$ and for $t \ge T( \vp, \hat{ \hat{\pi} }_{i} )$ we have
\begin{multline*}
\E  \Bigl[ \int_0^{t \wedge \sigma_1} e^{ - \rho t }c\, dt +
e^{ - \rho t \wedge \sigma_1 } \bar{\mu} \Bigr]
\le
\bar{\mu} \, \E  \left[ e^{ - \rho T( \vp, \hat{ \hat{\pi} }_{i} ) \wedge \sigma_1 }  \right] = \bar{\mu} \, \E  \left[ e^{ - \rho T( \vp, \hat{ \hat{\pi} }_{i} ) \wedge \sigma_1 }  \right] \\
 \le \bar{\mu} \, \E^{\vp} \exp\left(- \rho \Bigl[  \frac{ \pi_{i} - \hat{ \hat{\pi} }_{i} }{ - q_{ii}  + \lambda_{i}  } \,  \wedge \sigma_1\Bigr]\right)
 \le \bar{\mu} \pi_{i}  + \underline{\mu} (1-\pi_{i} ) \le H(\vp) .
\end{multline*}
This establishes \eqref{remaining-inequalities} and concludes the proof when $\rho > 0$.

\noindent \textbf{Case II: $c > 0$.} If $\rho > 0$, arguments given for Case I still holds. Hence we assume that $\rho =0$.
Using \eqref{exit-time-lower-bond} again, we obtain
\begin{align*}
\E^{\vp}  \Bigl[ T( \vp, \hat{ \hat{\pi} }_{i} ) \wedge \sigma_1 \Bigr]   \ge \E^{\vp} \Bigl[  \frac{ \pi_{i} - \hat{ \hat{\pi} }_{i} }{ - q_{ii}  + \lambda_{i}  } \wedge \sigma_1 \Bigr]
= \int_0^{\infty} \! \Bigl[  \frac{ \pi_{i} - \hat{ \hat{\pi} }_{i} }{ - q_{ii}  + \lambda_{i}  } \, \wedge u \Bigr]   \sum_{j \in E} \lambda_j m_j(u,\vp) du.
\end{align*}
The last expression above equals to $0 $ at $\pi_{i} = \hat{ \hat{\pi} }_{i} $ and it is strictly increasing in $\pi_{i} $ for $\pi_{i} \ge \hat{ \hat{\pi} }_{i} $. Therefore there exists a unique point
\begin{align*}
\pi^{s}_{i} \triangleq \inf  \left\{ \pi_{i} \ge \hat{ \hat{\pi} }_{i}  \, : \,
-c \, \E^{\vp} \Bigl[  \frac{ \pi_{i} - \hat{ \hat{\pi} }_{i} }{ - q_{ii}  + \lambda_{i}  } \wedge \sigma_1 \Bigr] + \bar{\mu}
 \le \bar{\mu} \pi_{i}  + \underline{\mu} (1-\pi_{i} )\right\}
 < 1,
\end{align*}
Then for the points $\vp$ with $\pi_{i} \ge \pi^{s}_{i}$ and for $t \ge T( \vp, \hat{ \hat{\pi} }_{i} )$ we have
\begin{multline*}
\E  \left[ \int_0^{t \wedge \sigma_1} c\,  dt +
\bar{\mu} \right]
=   c \, \E  \left[  t \wedge \sigma_1  \right]   + \bar{\mu}
\le  c \, \E  \left[ T( \vp, \hat{ \hat{\pi} }_{i} ) \wedge \sigma_1 \right]  + \bar{\mu}
  \\
 \le c \, \E^{\vp} \Bigl[  \frac{ \pi_{i} - \hat{ \hat{\pi} }_{i} }{ - q_{ii}  + \lambda_{i}  } \wedge \sigma_1 \Bigr] + \bar{\mu}
  \le \bar{\mu} \pi_{i}  + \underline{\mu} (1-\pi_{i} ) \le H(\vp),
\end{multline*}
and this concludes the proof.
\end{proof}

\begin{proof}[\textbf{Proof of Lemma~\ref{lem:convergence-of-finite-to-infinite}}]
The first inequality in \eqref{eq:convergence-of-finite-to-infinite}  
is obvious. To show the second inequality let $\tau $ be an $\F$-stopping time.
Then, we have
\begin{multline}
\label{proving-the-second-inequality}
 \E^{\vp} \left[  \int_0^{\tau} e^{- \rho t}k(\vP_t)dt + e^{- \rho \tau} H\left( \vP_{\tau} \right) \right]
\le  \E^{\vp} \left[ \int_0^{\tau \wedge T}  e^{- \rho t}C(\vP_t)dt + e^{- \rho \tau  \wedge T} H( \vP_{\tau \wedge T}  )
\right]
\\
+ \E^{\vp} \left[
 1_{ \{ \tau \ge T \} } \left(\int_T^{\tau} e^{- \rho t}C(\vP_t)dt  + e^{- \rho \tau} H( \vP_{\tau}  )  - e^{- \rho T} H( \vP_{T} ) \right)  \right].
\end{multline}
If $\rho > 0$, the last expectation above is bounded above by
$e^{- \rho T}   ( \|  C \| + 2 \cdot \|  H \|)$. Then taking the supremumover all $\tau$'s on both sides we obtain \eqref{eq:convergence-of-finite-to-infinite}.

On the other hand, if $\rho = 0$ and $ \max_{i \in E} c_i < 0$, we may safely restrict ourselves to the set of stopping times $\tau$ for which
$\E [ \tau ] \le \big( \min_{k,i}\mu_{k,i} - \max_{k,i} \mu_{k,i}\big) / \max_{i \in E} c_i $: the expected reward associated with any stopping time having a higher expected value is dominated by the reward achieved upon stopping immediately. Then, the second expectation in \eqref{proving-the-second-inequality} is bounded above by
\begin{align*}
2 \cdot \|  H \| \cdot \P \{\tau > T \}
\le 2 \cdot \|  H \| \frac{ \E [\tau ]}{T}
\le  \frac{2 \cdot \|  H \|}{T} \, \frac{\big( \min_{k,i}\mu_{k,i} - \max_{k,i} \mu_{k,i} \big)}{ \max_{i \in E} c_i},
\end{align*}
thanks to  Markov's inequality. Then, the inequality in \eqref{eq:convergence-of-finite-to-infinite} follows after taking the supremums over $\tau$ again.
\end{proof}

\begin{proof}[\textbf{Proof of \eqref{explicit-error-for-T-truncated-stopping-time}}]
Let $U^{(m)}_{\varepsilon}$ denote the stopping rule in \eqref{def:eps-stopping-time-for-truncated-inifinite-horizon} 
for notational convenience.
Since $ U^{(m)}_{\varepsilon}  \wedge T  \le U^{(m)}_{\varepsilon}  \le  U_{0} (\infty, \vp)$,
the arguments of \cite[Proposition 3.11 and Section 4.1]{dps} give
%
 \begin{align*}
V(T,\vp) \le V(\infty,\vp)  =  \E^{\vp}\left[
  \int_{0 }^{ U^{(m)}_{\varepsilon} \wedge T } \!\! e^{- \rho t} C(\vP_t ) \,dt +
  e^{-\rho ( U^{(m)}_{\varepsilon} \wedge T) } V \left(\infty, \vP_{U^{(m)}_{\varepsilon} \wedge T} \right) \right].
\end{align*}
On the event $\{ U^{(m)}_{\varepsilon}  \le T \}$, we use the inequality
$V\left( \infty , \vP_{U^{(m)}_{\varepsilon} } \right) - \varepsilon - Err_{\infty}(m) \le H\left( \vP_{U^{(m)}_{\varepsilon} } \right) $, $\P^{\vp}$-a.s., to obtain 
 \begin{align*}
 V(T,\vp) &\le \E^{\vp}\left[
  \int_{0 }^{ U^{(m)}_{\varepsilon}  \wedge T } \!\! e^{- \rho t} C(\vP_t ) \,dt +
  e^{-\rho ( U^{(m)}_{\varepsilon}  \wedge T)} H\left(  \vP_{U^{(m)}_{\varepsilon} \wedge T} \right)
  + \varepsilon + Err_{\infty}(m)   \right.
  \\
  &\hspace{6cm}\left.
  + 1_{ \{ U^{(m)}_{\varepsilon}(\infty, \vp) > T \}  }
    e^{-\rho T} \left[  V \left(\infty, \vP_{ T} \right) - H\left(  \vP_{U^{(m)}_{\varepsilon}  } \right) \right]
    \right]
    \\
    &\le \E^{\vp}\left[
  \int_{0 }^{ U^{(m)}_{\varepsilon}  \wedge T } \!\! e^{- \rho t} C(\vP_t ) \,dt +
  e^{-\rho (U^{(m)}_{\varepsilon}  \wedge T)} H\left(  \vP_{U^{(m)}_{\varepsilon}  \wedge T} \right)
  \right]  \\
  &\hspace{6cm}
   + \varepsilon + Err_{\infty}(m) +   e^{-\rho T}Err_{\infty}(0) \; \P  \{ U^{(m)}_{\varepsilon}  > T \}   .
 \end{align*}
 If $\rho > 0$, we obtain \eqref{explicit-error-for-T-truncated-stopping-time} by removing the last probability. Otherwise we can use Markov's inequality
 $ \P  \{ U^{(m)}_{\varepsilon}  > T \} \le \E [U^{(m)}_{\varepsilon}  / T]
 \le  \E [U^{0}_{\varepsilon}]  / T \le \max_{k,i} \mu_{k,i} / [ ( \min_{i \in E} c_i ) T ]  $ as in the proof of Lemma \ref{eq:convergence-of-finite-to-infinite}, and \eqref{explicit-error-for-T-truncated-stopping-time} follows.
 \end{proof}

\begin{proof}[\textbf{Proof of Remark~\ref{rem:two-hypothesis}}] The first claim on immediate stopping if $ \mu_{2,1} \mu_{1,2} (\lambda_2 - \lambda_1 ) \le \mu_{2,1} + \mu_{1,2}$ is an immediate corollary of \cite[Theorem 2.1]{PeskirShiryaev}.

Let us now assume that $ \mu_{2,1} \mu_{1,2} (\lambda_2 - \lambda_1 ) > \mu_{2,1} + \mu_{1,2}$. For the problem with two hypotheses, we have $H(\vp) = \min\{  \mu_{1,2} \pi_2 \, ; \,  \mu_{2,1} \pi_1 \}$, and recall that 
$v_1(T,\vp) = \inf_{t \in [0,s] }
JH(t,\vp)$. 
For $\vp = (\pi_1, \pi_2)$ with $ \pi_2 \in \big( \lambda_1 \mu_{2,1} /(\lambda_2 \mu_{1,2} + \lambda_1 \mu_{2,1}) \, , \, \mu_{2,1} / ( \mu_{2,1} +\mu_{1,2} ) \big) $ 
and for small $t>0$, evaluating the expression $ JH(t, \vp)  $ gives
\begin{align*}
 \left[ \pi_1 e^{- \lambda_1 t} + \pi_2 e^{- \lambda_2 t}\right] \mu_{1,2} x_2(t,\vp)
 + \int_0^t \sum_{j =1}^2 \pi_i e^{- \lambda_i u} \left( 1 + \lambda_j  \left(
\mu_{2,1} \frac{\lambda_1 x_1(u,\vp) }{\lambda_1 x_1(u,\vp) + \lambda_2 x_2(u,\vp)}
 \right) \right)
du,
\end{align*}
and using the dynamics of $t \mapsto \vx(t, \vp)$ in \eqref{eq:vx-dyn} we obtain
\begin{align}
\label{eq:dJH-example} \frac{d JH(t ,\vp)}{dt} =  \left[ 1 + \mu_{2,1} \lambda_1 \right]
\cdot \pi_1 e^{- \lambda_1 t } +  \left[ 1 - \mu_{1,2} \lambda_2 \right] \cdot \pi_2  e^{-
\lambda_2 t } . 
\end{align}
With $t=0$ and $\vp = (  \mu_{1,2} / ( \mu_{2,1} +\mu_{1,2} ) + \delta ,  \mu_{2,1} / ( \mu_{2,1} +\mu_{1,2} ) - \delta ) $, for $\delta > 0$ small, the derivative becomes
\begin{align*}
\frac{d JH(t, \vp)}{dt}\Big|_{t =0, \, \vp = (\cdot, \cdot)} =
\frac{1}{ \mu_{2,1} +\mu_{1,2}  } \left[ \mu_{2,1} +\mu_{1,2}  + \mu_{2,1} \mu_{1,2} (\lambda_1 - \lambda_2) \right] + \delta (\mu_{2,1} \lambda_1 + \mu_{1,2} \lambda_2) .
\end{align*}
Under the assumption $ \mu_{2,1} \mu_{1,2} (\lambda_2 - \lambda_1 ) > \mu_{2,1} + \mu_{1,2}$,  the last expression is negative for $\delta$ sufficiently small. This implies that $v_1(T, \vp) < H(\vp)$ for small values of $T>0$ at points $\vp $, for which
$\pi_2 = \mu_{2,1} / ( \mu_{2,1} +\mu_{1,2} ) - \delta$ where
\begin{align*}
 \delta <  \frac{  \mu_{2,1} \mu_{1,2} (\lambda_2 - \lambda_1 ) - \mu_{2,1} - \mu_{1,2}  }{ (\mu_{2,1} +\mu_{1,2}) \, ( \mu_{2,1} \lambda_1 + \mu_{1,2} \lambda_2)}.
\end{align*}
%
%
Since $b_1(0) = \mu_{2,1} / ( \mu_{2,1} +\mu_{1,2} )$, it follows that the boundary curve $T \mapsto b_1(T) $ is discontinuous at $T=0$ (see the lower curve in Figure~\ref{fig:hypothesis}).

The expression in \eqref{eq:dJH-example} with $t=0$ indicates that $d JH(t,\vp) / dt|_{t=0}$ is decreasing in $\pi_2$ and vanishes at the point $\vp$ with
\begin{align*}
 \pi_2 = \frac{ 1 + \mu_{2,1}\lambda_1 }{  \mu_{2,1}\lambda_1 + \mu_{1,2}\lambda_2 }
\le \frac{ \mu_{2,1} }{  \mu_{2,1} + \mu_{1,2}},
\end{align*}
where the inequality is due to the assumption $ \mu_{2,1} \mu_{1,2} (\lambda_2 - \lambda_1 ) > \mu_{2,1} + \mu_{1,2}$. This implies that
\begin{align*}
 \left\{ (T,\vp) : \; \pi_2 \le  \frac{ \mu_{2,1} }{  \mu_{2,1} + \mu_{1,2}} \; \text{and }
\; V_1 (T,\vp) = H(\vp) \right\} \subseteq
 \left\{ (T,\vp) : \pi_2 \le  \frac{ 1 + \mu_{2,1}\lambda_1 }{  \mu_{2,1}\lambda_1 + \mu_{1,2}\lambda_2 }    \right\}
\end{align*}
At the point $\vp $ with $\pi_2 = (1 + \mu_{2,1}\lambda_1) / ( \mu_{2,1}\lambda_1 + \mu_{1,2}\lambda_2)$ the expression for $d JH(t,\vp) / dt  $ in \eqref{eq:dJH-example} is strictly positive for small $t > 0$. Then, we can find a value of $u > 0$ such that
\begin{align*}
v_1(T, \vp ) = H(\vp ), \quad \text{for $\vp = \left( \frac{\mu_{1,2}\lambda_2}{\mu_{2,1}\lambda_1 + \mu_{1,2}\lambda_2} , \frac{1 + \mu_{2,1}\lambda_1}{ \mu_{2,1}\lambda_1 + \mu_{1,2}\lambda_2 } \right)$ and $T \in [0,u]$.}
\end{align*}
This further implies
\begin{align*}
v_1(T ,\vp ) = H(\vp)
\quad \text{on} \quad \left\{ (T,\vp) \, : \, T\in [0,u] \; \; \text{and}\; \;  \pi_2 \le \frac{1 + \mu_{2,1}\lambda_1}{ \mu_{2,1}\lambda_1 + \mu_{1,2}\lambda_2 } \right\},
\end{align*}
since
the region $\{ \vp \in D : V(T,\vp) = H(\vp) \}$ is convex for each $T$ (see Remark~\ref{lem:convexAdoption}),
and we have $v_1(T ,\vp ) = H(\vp)$, for all $T>0$ at $\vp = (1,0)$.
Recall that the deterministic part $t \mapsto \vx(t,\pi)$ drifts towards the point $(1,0)$. Then, by induction we conclude that $v_n(T ,\vp ) = H(\vp)$ for all $n \in \N$, which implies that $\lim_{n \to \infty}v_n(T ,\vp )= V(T ,\vp ) = H(\vp)$ on the same region.

As a result, we see that if the solution of the problem is not trivial,
the lower boundary curve $b_1(T)$ is discontinuous at $T=0$, and there is an initial region over which the curve stays flat at level $\pi_2 = (1 + \mu_{2,1}\lambda_1) / (\mu_{2,1}\lambda_1 + \mu_{1,2}\lambda_2)$ as in Figure~\ref{fig:hypothesis}.
\end{proof}

%
\bibliography{LS-references}
\bibliographystyle{siam}
\end{document}